\documentclass{amsart}

\usepackage{lineno}
\usepackage{amssymb}
\usepackage{amsmath}
\usepackage{amsthm}
\usepackage{mathrsfs}
\usepackage{bbm}
\usepackage{cite}
\usepackage{color}
\usepackage{pgfplots}
\usepackage{enumerate}

\graphicspath{{figures/}}
\usepackage{colortbl}

\usepackage[colorlinks,hypertexnames=false]{hyperref}
\usepackage[msc-links]{amsrefs}

\usepackage{changes}

\usepackage{mathtools}
\usepackage{extarrows}
\usepackage{setspace}

\usepackage{imakeidx}
\makeindex

\hypersetup{
	colorlinks=true,
	linkcolor=black,
	filecolor=black,
	urlcolor=black,
	citecolor=black,
}

\newtheorem{thm}{Theorem}[section]
\newtheorem{cor}[thm]{Corollary}
\newtheorem{defn}[thm]{Definition}

\newtheorem{prop}[thm]{Proposition}

\newtheorem{rmk}[thm]{Remark}

\newtheorem{definition}[thm]{Definition}

\newtheorem{remark}[thm]{Remark}
\newtheorem{example}[thm]{Example}

\newcommand{\ext}{\mathop{\mathrm{ext}}}
\newcommand{\clo}{\mathop{\mathrm{clo}}}
\newcommand{\dist}{\mathop{\mathrm{dist}}}
\newcommand{\re}{\mathop{\mathrm{Re}}}
\newcommand{\oz}{\mathop{{\scriptscriptstyle\mathbb{O}}_L}}
\newcommand{\oo}{\mathop{{\scriptscriptstyle\mathbb{O}}_R}}

\newcommand{\Fi}{{\imath}}
\newcommand{\Fj}{{\jmath}}
\newcommand{\Fk}{{\kappa}}
\newcommand{\Fl}{{\rho}}

\numberwithin{equation}{section}

\pgfplotsset{compat=1.18}

\begin{document}

\title[Convergence domains of sedenionic star-power series]{Convergence domains\\ of sedenionic star-power series}
\author{Xinyuan Dou}
\email[Xinyuan Dou]{douxinyuan@ustc.edu.cn}
\address{Department of Mathematics, University of Science and Technology of China, Hefei 230026, China}
\author{Ming Jin}
\email[Ming Jin]{mjin@must.edu.mo}
\address{Faculty of Innovation Engineering, Macau University of Science and Technology, Macau, China}
\author{Guangbin Ren}
\email[Guangbin Ren]{rengb@ustc.edu.cn}
\address{Department of Mathematics, University of Science and Technology of China, Hefei 230026, China}
\author{Irene Sabadini}
\email[Irene Sabadini]{irene.sabadini@polimi.it}
\address{Dipartimento di Matematica, Politecnico di Milano, Via Bonardi, 9, 20133 Milano, Italy}
\thanks{This work was supported by Xiaomi Young Talents Program. The fourth author is partially supported by PRIN 2022 {\em Real and Complex Manifolds: Geometry and Holomorphic Dynamics} and is member of GNASAGA of INdAM.} 


\begin{abstract}
The theory of slice regular  (also called hyperholomorphic) functions is a generalization of complex analysis originally given in the quaternionic framework, and then further extended to Clifford algebras, octonions, and to real alternative algebras. Recently, we have extended this theory to the case of (real) even-dimensional Euclidean space. We provided several tools for studying slice regular functions in this context, such as a path-representation formula, slice-solutions, hyper-solutions and hyper-sigma balls. When adding an algebra structure to Euclidean space we have more properties related to the possibility of
multiplying elements in the algebra. In this paper, we focus on exploring the properties specific to the context of the algebra of sedenions which is rather peculiar, being the algebra non associative and non alternative. The consideration of a non alternative algebra is a novelty in the context of slice regularity which opens the way to a number of further progresses. We offer an explicit characterization for slice-solutions and hyper-solutions, and we determine the convergence domain of the power series in the form $q\mapsto\sum_{n\in\mathbb{N}}(q-p)^{*n}a_n$ with sedenionic coefficients. Unlike the case of complex numbers and quaternions, the convergence domain in this context is determined by two convergence radii instead of one. These results are closely tied to the presence of zero divisors in the algebra of sedenions.

\end{abstract}

\maketitle

\section{Introduction}

A very successful analog of the Cauchy-Riemann equations in the quaternionic case, was introduced by Moisil and Fueter and further developed by Fueter and his school, see \cites{Fueter1934001, Fueter1935001,Moisil}. A quaternionic function $f:U\subset\mathbb{H}\rightarrow\mathbb{H}$ is ``holomorphic" (Fueter-regular, or hyperholomorphic, in modern terms) if, roughly speaking,
\begin{equation*}
    \frac{1}{4}\left(\frac{\partial}{\partial x_0}+i\frac{\partial}{\partial x_1}+j\frac{\partial}{\partial x_2}+k\frac{\partial}{\partial x_3}\right)f(x_0+x_1i+x_2j+x_3k)=0.
\end{equation*}
This function theory is well developed, see e.g., \cites{MR2089988,MR2369875,Sudbery1979001, Qian1998001, Frenkel2008001, Li2013001}. It important to note that none of the monomials $q\mapsto q^n$, $n\in\mathbb N$ is Fueter-regular and to overcome this problem, other function theories have been explored.

In 2007, Gentili and Struppa \cite{Gentili2007001} introduced the theory of slice regular function in one quaternionic variable inspired by an idea to to Cullen \cite{Cullen1965001} and treating, basically, the case of power series on balls centered at the origin. Later, it has been proved that convergent  series of the form $q\mapsto\sum_{n\in\mathbb{N}}(q-p)^{*n}a_n$ are slice regular functions. Then, the theory of slice regular functions has been extended to more general domains and also to the case of Clifford algebras \cite{Colombo2009002}, octonions \cite{Gentili2010001}, real alternative $*$-algebras \cite{Ghiloni2011001} and, more generally, to even-dimensional Euclidean spaces \cite{Dou2023002}. Furthermore, the slice regular functions in several variables \cites{Colombo2012002, Ghiloni2012001, Ghiloni2022001, Dou2023002, Xu2024001} and vector-valued functions \cite{Alpay2016001B} have been studied. The theory of slice regular functions provides powerful mathematical tools for studying quaternionic operators \cites{Colombo2007001, Gantner2020001, Colombo2022001, Wang2022001, Colombo2024001} and quantum mechanics \cites{Khokulan2015001, Muraleetharan2017001, Muraleetharan2018001, Thirulogasanthar2017001}, see also books \cites{Colombo2011001B, Colombo2018001B, Colombo2019001B}.

The natural product of slice regular functions is the so-called $*$-product \cite{Colombo2009001} which coincides with the product of polynomials or convergent series with noncommuting coefficients \cites{Gentili2008001,Gentili2008004}. There are other ways to define the $*$-product to meet different constraints on the domain of definition of the functions. For example, a $*$-product for non-axially symmetric domains is defined in \cites{Dou2024001,Dou2024002,Dou2024003}, and allows to consider the $*$-product $f_1*f_2:\Sigma(I,2)\rightarrow\mathbb{H}$, $I\in\mathbb{S}_\mathbb{H}$, on the so-called $\sigma$-ball $\Sigma(I,2)$ with center at $I$ and radius $2$, where
\begin{equation*}
	\begin{split}
		f_\ell: \quad \Sigma(I,2^\ell) \quad &\xlongrightarrow[\hskip1cm]{}\qquad \mathbb{H},
		\\ q\qquad &\shortmid\!\xlongrightarrow[\hskip1cm]{}\ \sum_{\ell\in\mathbb{N}}\left(\frac{q-I}{2^\ell}\right)^{*2^n},
	\end{split}\qquad\qquad \ell=1,2.
\end{equation*}
The definition in \cites{Dou2024001,Dou2024002,Dou2024003} is treats a very general case and thus it is complicated. In this paper it is enough to consider the $*$-product of power series, which is a classical notion going back to \cite{Fliess} and is defined by taking the Cauchy product of the coefficients, i.e.
\begin{equation}\index{$*$-product}\label{eq-starproduct}
	\left(\sum_{\ell\in\mathbb{N}}q^\ell a_\ell\right)
    *\left(\sum_{\ell\in\mathbb{N}}q^\ell b_\ell\right)
    =\sum_{\ell\in\mathbb{N}} \sum_{k=0}^n q^\ell \left(a_k b_{\ell-k}\right),
\end{equation}
where $a_\ell, b_\ell\in\mathfrak{S}$ are sedenions. The $*$-product is  {ubiquitous} in the theory of slice regular functions and it is used everywhere one needs a multiplication which is an internal operation.

The convergence set of a power series with $*$-product is a $\sigma$-ball in the case of quaternions; see \cite{Gentili2012001}. However, the convergence of such a power series in the case of a real alternative $*$-algebra $A$ is hard to study, see \cite{Ghiloni0214002}*{Theorem 3.4}. In fact, to check that the power series
\begin{equation}\label{eq-powerseries}
    P:\quad q\shortmid\!\xlongrightarrow[\hskip1cm]{} \sum_{\ell\in\mathbb{N}}(q-p)^{*\ell} a_\ell
\end{equation} converges in the $\sigma$-ball follows a standard reasoning. On the other hand, it is hard to prove that $P$ diverges in the exterior of the $\sigma$-ball unless a special condition is added:
\begin{equation*}
    a_\ell\in\mathcal{Q}_A=\bigcup_{I\in\mathbb{S}_A}\mathbb{C}_I.
\end{equation*}
Thus, it is an interesting and difficult question to characterize the convergence of the power series $P$ with $a_\ell\in A$. This question is even more interesting in the general case of slice regular functions in which the codomain is considered a real vector space of even dimension, see \cite{Dou2023002}. It is also possible to check that the power series converges in $\sigma$-ball by the similar method in the proof of \cite{Dou2021001}*{Theorem 8.8}. However, some power series do not diverge in the exterior of the $\sigma$-ball, see the proof of \cite{Dou2021001}*{Proposition 7.9}. This fact shows that the convergence domain of the power series is dependent on the algebraic structure under consideration. We have many concrete examples in the general case, such as Cayley-Dickson algebras. In this paper, we will use the algebraic properties of sedenions, which form the real $16$-dimenional Cayley-Dickson algebra, to characterize the convergence domain of the power series $P$ with sedenionic coefficients.

The Cayley-Dickson algebras $A_\ell$, $\ell\in\mathbb{N}$, play a crucial role in mathematical physics, e.g. \cite{Mirzaiyan2023001} exploits the power of these algebras
to generate stationary rotating black hole solutions in one fell
swoop. Quaternions $\mathbb{H}=A_2$ contribute to the study of quantum mechanics and quantum field, especially the Schr{\"o}dinger equation \cite{Adler1995001B}. Octonions are closely related to $G_2$-manifold \cite{Corti2015001}, $M$-theory \cite{Gunaydin2016001}, and Yang-Mills field theory \cite{Figueroa1998001}. {The algebraic structure of sedenions $\mathfrak{S}=A_4$ is more intricated, and still has applications \cites{Chanyal2014001,Gillard2019001}.} This paper is devoted to the convergence domain of sedenionic power series with $*$-product, aiming to provide more mathematical tools for mathematical physics in higher dimensions.

Our result show that the domain of convergence of the sedenionic $*$-power series $P$   in \eqref{eq-powerseries}, may be the intersection of a $\sigma$-ball $\Sigma(p,R_a^p)$, see \eqref{eq-sigmapab},  and a hyper-$\sigma$ ball $\Sigma(p,R_a,(I_p,J))$, see Theorem \ref{thm-lpmw}. Thus, there are two convergence radii, $R_a$ and $R_a^{p,J}$. This is a new phenomenon that does not appear in complex numbers, quaternions, and octonions. The definition of hyper-$\sigma$-ball is related to the hyper-solutions and the zero divisors of sedenions, see Definition \ref{eq-hypersigmaball} and Corollary \ref{co-lj}. To this end, we introduce a general polar coordinate system of slice-unit $I\in\mathcal{C}_\mathfrak{S}$ with the coordinate
\begin{equation*}
    (\alpha_{[I]},\theta_{[I]},\jmath_{[J]})\in [0,\pi]\times [0,\pi)\times\mathbb{S}_\mathbb{O},
\end{equation*}
see Proposition \ref{pr-fei}. This general polar coordinate is used to describe some intermediate tools, see Propositions \ref{pr-ijj}, \ref{pr-lemi1i2} and \ref{pr-st}.

The paper is organized as follows. In Section \ref{sc-mainresults}, we collect some notations and definitions and we state our main results. In Section \ref{sc-preliminaries}, we recall some properties of sedenions and sedenionic slice analysis useful for the study in the paper. In Section \ref{sc-hyper}, we give a general polar coordinate system of slice-units. Based on this, we prove that a $2$-tuple $(J_1,J_2)$ with $J_1\neq J_2$ is a hyper-solution if and only if $J_1-J_2$ is a zero divisor. In Section \ref{sc-hypersigmaball}, we classify the hyper-$\sigma$ balls, and prove that both hyper-$\sigma$-balls and $\sigma$-balls are domains of regularity. In Section \ref{sc-ods}, we will give two orthogonal decompositions of sedenions. In Section \ref{sc-domainofconvergence}, we prove our main result, that is: the domain of convergence of power series is the intersection of a hyper-$\sigma$-ball and a $\sigma$-ball. It is clear that this new and peculiar phenomenon is caused by the presence of zero divisors  {and to the peculiarities of} the algebra of sedenions.

\section{Main results}\label{sc-mainresults}
In this section, we state our main results. To this end, we need to recall some notations and definitions. The Cayley–Dickson construction \cite{Dickson1919001} inductively produces a sequence of algebras over the field of real numbers, by setting
\begin{equation*}\index{$\mathbb{R}$}
    A_0:=\mathbb{R},
\end{equation*}
and defining
\begin{equation}\index{$A_\ell$}\index{$\mathbb{N}$}\label{eq-el}
    A_{\ell+1}:=A_\ell+A_\ell e_{2^\ell},\qquad \ell\in\mathbb{N}:=\{0,1,2,\ldots\}.
\end{equation}
Each $A_{\ell +1}$ has twice the dimension of the previous $A_\ell$, and $a+be_{2^\ell}$, $c+de_{2^\ell}$ satisfy:\index{Cayley-Dickson construction}
\begin{equation}\label{eq-aell}
    \begin{cases}
        (a+be_{2^\ell})(c+de_{2^\ell}):=ac-\overline{d}b+(da+b\overline{c})e_{2^\ell},\\
        \overline{a+be_{2^\ell}}:=a-be_2^\ell,
    \end{cases}\qquad\forall\ a,b,c,d\in A_\ell,
\end{equation}
where $\overline{x}:=x$ when $x\in\mathbb{R}$, and $e_{2^\ell}$ \index{$e_{2^\ell}$} is an imaginary unit in $A_{\ell+1}$.

The first five Cayley-Dickson algebra respectively are: real numbers $\mathbb{R}=A_0$, complex numbers $\mathbb{C}=A_1$, \index{$\mathbb{C}$} quaternions $\mathbb{H}=A_2$, \index{$\mathbb{H}$} octonions $\mathbb{O}=A_3$ \index{$\mathbb{O}$} and sedenions $\mathfrak{S}=A_4$. \index{$\mathfrak{S}$}

Let $e_0:=1\in\mathbb{R}$, {and $e_{2^n}\ (n=1,2,\cdots)$ be defined as in \eqref{eq-el}. The remaining elements are determined inductively for $n=1,2,3,\cdots$ by
\begin{equation*}\index{$e_\ell$}
    e_{m+2^n}:=e_m e_{2^n},\qquad\qquad m=1,2,...,2^n-1.
\end{equation*}
With these conventions, $\{e_\ell\}_{\ell=0}^{2^{n}-1}$ forms a basis of $A_{n}$, $n\in\mathbb{N}$, as a real algebra.}
In particular, each $s\in\mathfrak{S}$ is of the form
\begin{equation*}
	s=\sum_{m=0}^{15}x_m e_m,\qquad x_m\in\mathbb{R}.
\end{equation*}
{By \eqref{eq-aell},} elements $\{e_\ell\}_{\ell=0}^{15}$ satisfy the multiplication table in the Appendix. For the properties of sedenions, we refer the reader to the papers \cites{Schafer1954001,Moreno1998001,Muses1980001}.

\begin{remark}\label{rmk-zd}
The algebra of sedenions contains zero divisors; for example, direct calculations show that
    \begin{equation}\label{eq-e1e10}
        (e_1-e_{10})(e_4+e_{15})=0,
    \end{equation}
    and so $e_1-e_{10}$ and $e_4+e_{15}$ are zero divisors.
\end{remark}
\begin{definition}
By $L_s$ we denote the left multiplication by $s\in A_\ell$, namely \index{$L_s$} $$L_s:{A_\ell}\rightarrow {A_\ell},\qquad \  a\mapsto sa.$$
The set $\mathcal{S}_{A_\ell}$ of the so-called slice-units in $A_\ell$ is defined by
\begin{equation*}\index{$\mathcal{S}_{A_\ell}$}
	\mathcal{S}_{A_\ell}:=\left\{s\in A_\ell:L_s^2=-id_{A_\ell}\right\};
\end{equation*}
moreover we introduce the notation
\begin{equation*}\index{$\mathcal{C}_{A_\ell}$}
	\mathcal{C}_{A_\ell}:=\left\{L_s:s\in\mathcal{S}_{A_\ell}\right\},
\end{equation*}
while \index{$\mathbb{S}_{A_\ell}$}
\begin{equation*}
	\mathbb{S}_{A_\ell}:=\{\mathfrak{i}\in A_\ell:\mathfrak{i}^2=-1\}
\end{equation*}
denotes the set of imaginary units of the algebra $A_\ell$.
\end{definition}
By direct calculations, in the case of the algebra of sedenions $\mathfrak{S}$, we obtain \index{$\mathcal{S}_{\mathfrak{S}}$}
\begin{equation}\label{eq-sf}
	\mathcal{S}_{\mathfrak{S}}
    (=\mathcal{S}_{{A_4}})
    =\left\{a+be_8\in\mathbb{S}_{\mathfrak{S}}:a,b\in\mathbb{O},\ ab=ba\right\}.
\end{equation}

The so-called slice cone of sedenions $\mathfrak{S}$ is denoted by
\begin{equation*}\index{$\mathcal{W}_\mathfrak{S}$}
	\mathcal{W}_\mathfrak{S}
    :=\bigcup_{I\in\mathcal{C}_\mathfrak{S}}\mathbb{C}_I,
\end{equation*}
and is equipped with the slice topology
\begin{equation*}
\index{$\tau_s(\mathcal{W}_\mathfrak{S})$}
	\tau_s(\mathcal{W}_\mathfrak{S}):=\{\Omega\subset \mathcal{W}_\mathfrak{S}\ :\Omega_I\in\tau(\mathbb{C}_I),\ \forall\ I\in\mathcal{C}_\mathfrak{S}\}.
\end{equation*}
Note that we denote an element in $\mathcal{C}_\mathfrak{S}$ by $I$ but also by $L_s$ when we make explicit that the element corresponds to the left multiplication by $s\in\mathfrak{S}$.
Open sets, connected sets, and paths in $\tau_s(\Omega)$ are called slice-open sets, slice-connected sets, and slice-paths in $\Omega$, respectively. \index{Slice-open, slice-connected, slice-paths} Following \cite{Dou2021001} we introduce the next definition:

\begin{defn}\index{Slice regular}
	Let $\Omega\in\tau_s(\mathcal{W}_\mathfrak{S})$. A function $f:\Omega\rightarrow \mathfrak{S}$ is called slice regular, if for each $I\in\mathcal{C}_\mathfrak{S}$, $f_I:=f|_{\Omega_I}$ is real differentiable and
	\begin{equation*}
		\frac{1}{2}\left(\frac{\partial}{\partial x}+I\frac{\partial}{\partial y}\right)f_I(x+yI)=0,\qquad\mbox{on}\qquad\Omega_I.
	\end{equation*}
\end{defn}

Let $z\in\mathbb{C}$ and $r\in [0,+\infty]$. Let us set
\begin{equation*}\index{$B^*(z,r)$,$B(z,r)$}
	B^*(z,r):=\{w\in\mathbb{C}:|w-z|<r\}\bigcup\{z\}.
\end{equation*}
If $r>0$, then $B^*(z,r)$ is an open ball, and we abbreviate $B^*(z,r)$ as $B(z,r)$. If $r\in[0,\infty)$ is the radius of convergence of $\{a_n\}_{n\in\mathbb{N}}$ with $a_n\in\mathbb{C}$, then the complex power series $w\mapsto(w-z)^{n}a_n$ converges on $B^*(z,r)$.

Let $b=\sum_{\imath=0}^{15} b_\imath e_\imath$ and $c=\sum_{\imath=0}^{15} c_\imath e_\imath$ with $b_\imath,c_\imath\in\mathbb{R}$. Set
\begin{equation}\index{$\langle \cdot,\cdot\rangle$}\label{eq-lr}
	\langle b,c\rangle=\sum_{\imath=0}^{15} b_\imath c_\imath,
	\qquad\mbox{and define}\qquad
	\langle L_b,L_c\rangle:=\langle b, c\rangle.
\end{equation}

Fix $I_0\in\mathcal{C}_\mathfrak{S}$. Let $p\in\mathcal{W}_\mathfrak{S}$ and
\begin{equation*}\index{$Re(p)$}\index{$I_p$}
    Re(p):=\langle p,1\rangle,\qquad\qquad
	I_p:=
    \begin{cases}
        I_0,\qquad & p\in\mathbb{R},\\
	    \frac{p-\re(p)}{|p-\re(p)|},\qquad &\mbox{otherwise},
	\end{cases}
\end{equation*}
\begin{equation*}\index{$Im(p)$}\index{$z_p$}
    Im(p):=\langle p, I_p\rangle,\qquad\qquad\mbox{and}\qquad\qquad z_p:=Re(p)+Im(p)i\in\mathbb{C}.
\end{equation*}

Let $I\in\mathcal{C}_\mathfrak{S}\cup\{0\}$ and $J\in\mathcal{C}_\mathfrak{S}$ and set
\begin{equation*}\index{$\Psi_i^I$}
	\begin{split}
		\Psi_i^I:\quad \mathbb{C}\quad &\xlongrightarrow[\hskip1cm]{}\quad \mathbb{R}+\mathbb{R}I,
		\\ x+yi\ &\shortmid\!\xlongrightarrow[\hskip1cm]{}\quad x+yI,
	\end{split}
\end{equation*}
and
\begin{equation*}\index{$\Psi_J^I$}
    \Psi_J^I:=\Psi_i^I\circ\left(\Psi_i^J\right)^{-1}.
\end{equation*}

\begin{defn}
	Let $p\in\mathcal{W}_\mathfrak{S}$ and $r\in [0,+\infty]$ and
	\begin{equation}\index{$\Sigma(p,r)$}\index{$\sigma$-ball}\label{eq-sigmapr}
		\Sigma(p,r):=\left\{\bigcup_{J\in\mathcal{C}_\mathfrak{S}} \Psi_i^J\left[B^*(z_p,r)\cap B^*(\overline{z_p},r)\right]\right\}\bigcup\Psi_i^{I_p}\left[B^*(z_p,r)\right].
	\end{equation}
	If $r\in(0,+\infty)$, then we call $\Sigma(p,r)$ the \emph{$\sigma$-ball with center $p$ and radius $r$}.
\end{defn}

Our next task is to define hypersolutions, with the ulterior goal of defining the so-called hyper-$\sigma$-balls, and to this end we need some more terminology and concepts.

Let $J=(J_1,...,J_m)\in\left(\mathcal{C}_\mathfrak{S}\right)^m$. Define
\begin{equation}\label{eq-msmfs}\index{$\mathcal{C}_\mathfrak{S}^{\ker}(J)$}
	\begin{split}
    	\mathcal{C}_\mathfrak{S}^{\ker}(J)
        :=&\left(\mathcal{C}_\mathfrak{S}\right)_{\ker}(J)
        \\:=&\left\{I\in\mathcal{C}_\mathfrak{S}:\ker(1,I)\supset\ker[\zeta(J)]=\bigcap_{\ell=1}^m\ker(1,J_\ell)\right\},
	\end{split}
\end{equation}
where
\begin{equation*}
	\zeta(J):=\begin{pmatrix}
		1&J_1\\\vdots&\vdots\\1&J_k
	\end{pmatrix},
\end{equation*}
and
\begin{equation*}
	\ker(1,I):=\left\{\begin{pmatrix} a\\b \end{pmatrix}\in\mathfrak{S}^{2\times 1}:(1,I)\begin{pmatrix} a\\b \end{pmatrix}=0\right\}.
\end{equation*}

\begin{example}\label{EX24}
Some direct computations which will be fully justified in Example \ref{ex-cker} show that
    \begin{equation*}
        \begin{split}
            &\mathcal{C}_{\mathfrak{S}}^{\ker}(L_{e_1},L_{e_{10}})
            \\=&\bigg\{L\left(\cos^2\theta e_1+\cos\theta\sin\theta e_2+[\sin\theta\cos\theta e_1+\sin^2\theta e_2]e_8\right):\theta\in[0,\pi)\bigg\}
        \end{split}
    \end{equation*}
    (where we write $L(s)$ instead of $L_s$) is a closed curve in $\mathcal{C}_\mathfrak{S}$.
\end{example}

\begin{defn}\index{Slice-solution}
The element	$J=(J_1,...,J_m)\in\mathcal{C}_\mathfrak{S}^m$ is called a slice-solution, if $\mathcal{C}_\mathfrak{S}^{\ker}(J)=\mathcal{C}_\mathfrak{S}$.
\end{defn}

\begin{defn}\index{Hyper-solution}
The element	$J=(J_1,...,J_m)\in\mathcal{C}_\mathfrak{S}^m$ is called a hyper-solution, if $J$ is not a slice-solution, and for each $I\in\mathcal{C}_\mathfrak{S}\backslash\mathcal{C}_\mathfrak{S}^{\ker}(J)$, $(J_1,...,J_m,I)$ is a slice-solution.
\end{defn}

\begin{defn}\label{eq-hypersigmaball}
	Let $p\in\mathcal{W}_\mathfrak{S}$, $J=(J_1,J_2,...,J_m)\in\mathcal{C}_\mathfrak{S}^m$ and $r\in[0,+\infty]$. Denote
	\begin{equation}\label{eq-sqrj}\index{$\Sigma(p,r,J)$}\index{Hyper-$\sigma$-ball}
		\begin{split}
		    \Sigma(p,r,J):=&\left\{\bigcup_{K\in\mathcal{C}_\mathfrak{S}} \Psi_i^K\left[B^*(z_p,r)\cap B^*(\overline{z_p},r)\right]\right\}
            \\&\bigcup\left\{\bigcup_{K\in\mathcal{C}_\mathfrak{S}^{\ker}(J)} \Psi_i^{K}\left[B^*(z_p,r)\right]\right\}.
		\end{split}
	\end{equation}
    If $r>0$, $J$ is a hyper-solution with $I_p=J_1$, then we call $\Sigma(p,r,J)$ the \emph{hyper-$\sigma$-ball with hyper-solution $J$, center $p$ and radius $r$}.
\end{defn}
	
\begin{rmk}
	Let $J=(J_1,J_2,...,J_m)\in\mathcal{C}_\mathfrak{S}^m$ and $K_1,...,K_\ell\in\{J_1,J_2,...,J_m\}$. If $\ker(\zeta(J))=\ker(\zeta(K))$, then it is immediate to check that
	\begin{equation*}
		\mathcal{C}_\mathfrak{S}^{\ker}(J)=\mathcal{C}_\mathfrak{S}^{\ker}(K).
	\end{equation*}
	Let $q=x+yJ_1\in\mathbb{C}_{J_1}$ and $r\in[0,+\infty)$. Then
	\begin{equation*}
		\Sigma(q,r,J)=\Sigma(x+yK_1,r,K).
	\end{equation*}
\end{rmk}

{\bf Convention:} As customary, we define
\begin{equation*}
    \frac{1}{0}:=+\infty,\qquad\mbox{and}\qquad \frac{1}{+\infty}:=0.
\end{equation*}

\begin{definition}
Let $p\in\mathcal{W}_{\mathfrak{S}}$, $J\in\mathcal{C}_{\mathfrak{S}}$ and $a:=\{a_\ell\}_{\ell\in\mathbb{N}}$ with $a_\ell\in \mathfrak{S}$. We define two convergence radii by
\begin{equation}\label{eq-Ra}\index{$R_a$}
	R_a:=\frac{1}{\limsup\limits_{\ell\rightarrow +\infty}|a_\ell|^{\frac{1}{\ell}}}
\end{equation}
and
\begin{equation}\label{eq-rapj}\index{$R_{a}^{p,J}$}
	R_{a}^{p,J}:=\begin{cases}
		R_a,\qquad & p\in\mathbb{R}\quad
        \mbox{or}\quad J=I_p,\\
		\left[{\limsup\limits_{\ell\rightarrow +\infty} \dist\big(a_\ell,\ker(I_p-J)\big)^{\frac{1}{\ell}}}\right]^{-1},\qquad\qquad &\mbox{otherwise},
	\end{cases}
\end{equation}
where
\begin{equation*}
    \ker(I_p-J):=\{c\in\mathfrak{S}:(I_p-J)c=0\}.
\end{equation*}
\end{definition}

It is easy to check that
\begin{equation}\label{eq-gek}
    R_a^{p,K}\ge R_a,\qquad\forall\ K\in\mathcal{C}_\mathfrak{S}.
\end{equation}

\begin{example}
Since $\mathfrak S$ contains zero divisors (see Remark \ref{rmk-zd}) according to \eqref{eq-kerkerc}, it follows by taking $u=e_1$, $v=-e_2$ and $w=e_4$ in Proposition \ref{pr-zd} that
    \begin{equation}\label{eq-kere1e10}
        \ker(e_1-e_{10})={\mathrm{span}}_\mathbb{R}\left(e_4+e_{15},e_5-e_{14},e_6+e_{13},e_7-e_{12}\right).
    \end{equation}
    It follows that
    \begin{equation}\label{eq-1perpker}\index{$\perp$}
        1\perp\ker(e_1-e_{10}),
    \end{equation}
    that is $\langle 1,\ker(e_1-e_{10})\rangle=0$.
\end{example}

Let $p\in\mathcal{W}_{\mathfrak{S}}$ and $a:=\{a_\ell\}_{\ell\in\mathbb{N}}$ with $a_\ell\in \mathfrak{S}$. As we shall see in \eqref{eq-lrapk}, there is  $J\in\mathcal{C}_{\mathfrak{S}}$ such that $R_a^{p,J}=R_a^p$, where
\begin{equation}\index{$R_a^p$}\label{eq-rap}
	R_a^p:=\sup_{K\in\mathcal{C}_\mathfrak{S}}\left\{R_a^{p,K}\right\}.
\end{equation}
According to \eqref{eq-gek} and \eqref{eq-rap},
\begin{equation*}
    R_a^p\ge R_a.
\end{equation*}
We now introduce the notation
\begin{equation*}\index{$B^*_{\mathcal{W}_{\mathfrak{S}}}(p,R_a)$}
    B^*_{\mathcal{W}_{\mathfrak{S}}}(p,R_a):=\{q\in\mathcal{W}_{\mathfrak{S}}:|p-q|<R_a\}\cup\{p\}.
\end{equation*}
and
\begin{equation}\index{$\Sigma(p,a)$}\label{eq-sigmapab}
	\Sigma(p,a):=\begin{cases}
		B^*_{\mathcal{W}_{\mathfrak{S}}}(p,R_a),\qquad & p\in\mathbb{R},\\
        \Sigma(p, R_a^p),\qquad & R_a^p=R_a,\\
		\Sigma(p,R_a,(I_p,J))\cap\Sigma(p, R_a^p),\qquad\qquad& \mbox{ortherwise}.
	\end{cases}
\end{equation}
As we shall fully justify later in Proposition \ref{pr-lpimw}, \eqref{it-rapj1} the notation is well-posed.
Moreover, if $R_a>0$, then for $\varepsilon\in\left[0,R_a\right)$
we set
\begin{equation*}\index{$\Sigma(p,a,\varepsilon)$}
	\Sigma(p,a,\varepsilon):=\begin{cases}
		B_{\mathcal{W}_{\mathfrak{S}}}(p,R_a-\varepsilon),\qquad & p\in\mathbb{R},\\
        \Sigma(p, R_a^p-\varepsilon),\qquad & R_a^p=R_a,\\
		\Sigma(p,R_a-\varepsilon,(I_p,J))\cap\Sigma(p, R_a^p-\varepsilon),\qquad\qquad & \mbox{ortherwise},
	\end{cases}
\end{equation*}
where we write \index{$B_{\mathcal{W}_{\mathfrak{S}}}(p,R_a)$} $B_{\mathcal{W}_\mathfrak{S}}(p,r)$ instead of $B^*_{\mathcal{W}_\mathfrak{S}}(p,r)$ when $r>0$.

{\bf Convention}. In the paper, we shall often write $s\in L(\mathfrak{S}):=\{L_a:a\in\mathfrak{S}\}$ instead of $L_s$, $s\in\mathfrak{S}$.

\begin{example}
    Let $p=L_{e_1}$, $J=L_{e_{10}}$ and $a_\ell=\frac{1}{3^\ell}+\frac{e_4+e_{15}}{2^\ell}$, $\ell\in\mathbb{N}_+:=\mathbb{N}\backslash\{0\}$\index{$\mathbb{N}_+$}. According to \eqref{eq-kere1e10} and \eqref{eq-1perpker},
    \begin{equation*}
        R_a=2,\qquad\mbox{and}\qquad R_a^{L_{e_1},J}=3,
    \end{equation*}
    where $a=\{a_\ell\}_{\ell\in\mathbb{N}_+}$. To elaborate this example we need to appeal to Proposition \ref{pr-lpimw} in which we take $J_1=J_2=L_{e_{10}}$, so that
    \begin{equation*}
        R_a^p=R_a^{L_{e_1},J}=3.
    \end{equation*}
    Then, one can check by
    \begin{equation*}
        \Sigma(L_{e_1},a)=\Sigma(L_{e_1},2,(L_{e_1},L_{e_{10}}))\cap\Sigma(L_{e_1},3)
    \end{equation*}
    that
    \begin{equation*}
        [\Sigma(e_1,a)]_I
        =\begin{cases}
            \Psi_i^{L_{e_1}}[B(i,2)],\qquad\qquad &I=\pm L_{e_1},\\
            \Psi_i^{I}[B(i,2)\cap B(-i,3)],\qquad\qquad & I\in\mathcal{C}_\mathfrak{S}^{\ker}(L_{e_1},L_{e_{10}})\backslash\{L_{e_1}\},\\
            \Psi_i^{I}[B(-i,2)\cap B(i,3)],\qquad\qquad & I\in -\left[\mathcal{C}_\mathfrak{S}^{\ker}(L_{e_1},L_{e_{10}})\backslash\{L_{e_1}\}\right],\\
            \Psi_i^{I}[B(i,2)\cap B(-i,2)],\qquad\qquad & \mbox{otherwise}.
        \end{cases}
    \end{equation*}
Figure $1$ displays $[\Sigma(L_{e_1},a)]_I$ for all possible $I\in\mathcal{C}_\mathfrak{S}$.
    \begin{figure}[htbp]
        \includegraphics[width=12.5cm]{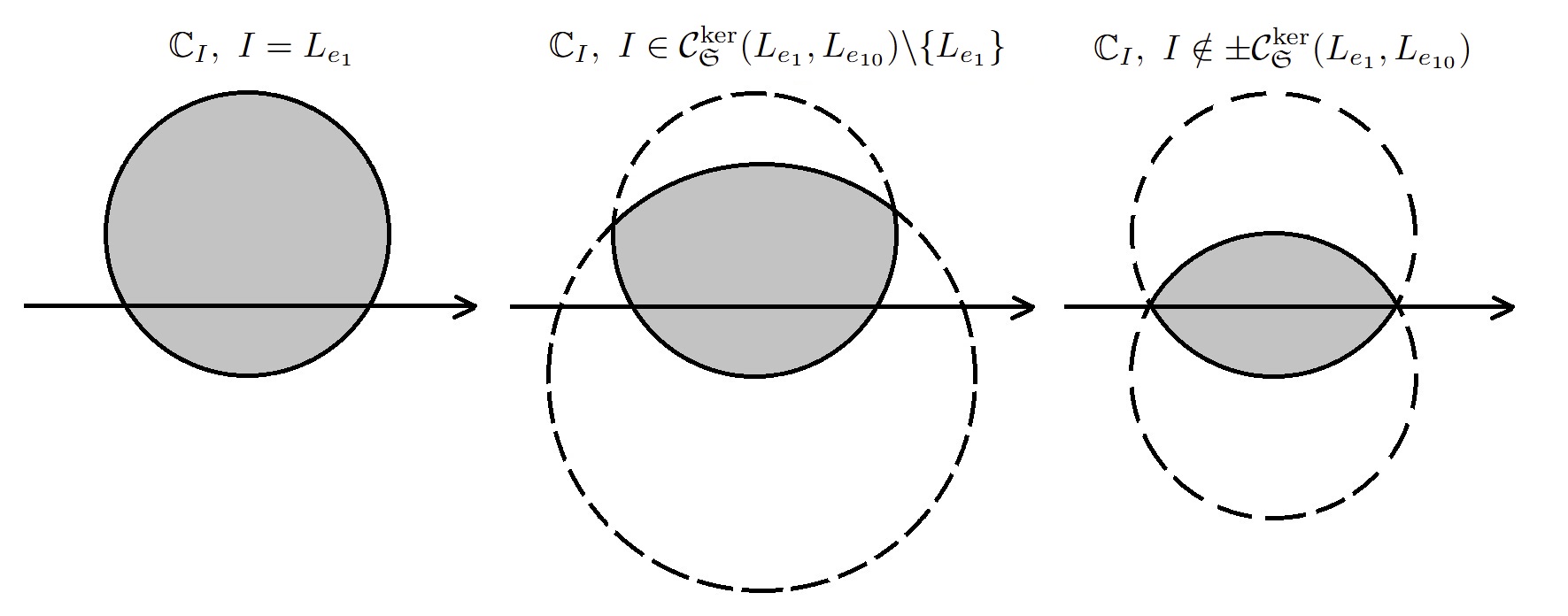}\caption{}
    \end{figure}
\end{example}

{\bf Notation}. Let $(X,\tau)$ be a topological space and $U\subset X$.
\index{${\clo}_{\tau(X)}(U)$, ${\ext}_{\tau(X)}(U)$}
Denote ${\clo}_{\tau(X)}(U)$, ${\ext}_{\tau(X)}(U)$ is the closure and the exterior of $U$ in $\tau(X)$.  In particular,  for $V\subset\mathcal{W}_{\mathfrak{S}}$, we write for short $\ext_{\tau_s}(V)$ instead of  $\ext_{\tau_s(\mathcal{W}_{\mathfrak{S}})}(V)$, and $\clo_{\tau_s}(V)$ instead of $\clo_{\tau_s(\mathcal{W}_{\mathfrak{S}})}(V)$. \index{$\clo_{\tau_s}(V)$, $\ext_{\tau_s}(V)$}

To state the main theorem, studying $*$-power series and their convergence, we recall that the $*$-product is defined by \eqref{eq-starproduct}. Let $p\in\mathcal{W}_{\mathfrak{S}}$. By induction, the $*$-product of polynomials with the variable $q\in\mathcal{W}_{\mathfrak{S}}$ satisfies
\begin{equation*}
	\begin{split}
		(q-p)^{*(n+1)}
        =&\left[(q-p)^{*n}\right]*(q-p)
        \\=&\left[\sum_{\ell=0}^n \begin{pmatrix}n\\ \ell\end{pmatrix} q^{n-\ell}(-p)^\ell\right]*(q-p)
        \\=&\ \sum_{\ell=0}^{n+1} \begin{pmatrix}n+1\\ \ell\end{pmatrix} q^{n+1-\ell}(-p)^\ell.
	\end{split}
\end{equation*}

\begin{thm}\label{thm-lpmw}
	Let $p\in\mathcal{W}_{\mathfrak{S}}$ and $a=\{a_\ell\}_{\ell\in\mathbb{N}}$ with $a_\ell\in\mathfrak{S}$. Then the power series
	\begin{equation*}
		\begin{split}
			P:\ \mathcal{W}_\mathfrak{S}\ &\xlongrightarrow[\hskip1cm]{}\ \mathfrak{S},
			\\ q\ \ &\shortmid\!\xlongrightarrow[\hskip1cm]{}\ P(q)=\sum_{\ell\in\mathbb{N}} (q-p)^{*\ell} a_\ell,
		\end{split}
	\end{equation*}
	converges absolutely and it is slice regular in $\Sigma(p,a)$, moreover it diverges in
    ${\ext}_{\tau_s}\left[\Sigma(p,a)\right]$.
	
	Moreover, if $R_a\in(0,+\infty]$ and $\varepsilon\in\left(0,R_a\right)$, then the power series $P$ converges uniformly in
    \begin{equation*}
        \begin{cases}
            \clo_{\tau_s}\left[\Sigma(p,\varepsilon)\right],\qquad\qquad&R_a=+\infty,
            \\\clo_{\tau_s}\left[\Sigma(p,a,\varepsilon)\right],\qquad\qquad&\mbox{otherwise.}
        \end{cases}
    \end{equation*}
\end{thm}
The proof is postponed to Section \ref{sc-domainofconvergence}.

\section{Some main properties of sedenions}\label{sc-preliminaries}

In this section, we recall some classical properties of the algebra of sedenions (see e.g. \cite{Moreno1998001} for further information), as well as some fundamental facts of slice analysis in the particular case of sedenions, which were originally treated in \cite{Dou2023002}*{Section 9}. Following \cite{Dou2021001}*{Definition 7.1}, we define first domains of slice regularity in $\mathcal{W}_\mathfrak{S}$.
\begin{defn}\index{Domain of slice regularity}
	A slice-open set $\Omega\subset\mathcal{W}_\mathfrak{S}$ is called a domain of slice regularity if there are no slice-open sets $\Omega_1$ and $\Omega_2$ in $\mathcal{W}_\mathfrak{S}$ with the following properties:
	\begin{enumerate}[\upshape (i)]
		\item $\varnothing\neq\Omega_1\subset\Omega_2\cap\Omega$.
		\item $\Omega_2$ is slice-connected and not contained in $\Omega$.
		\item For any slice regular function $f$ on $\Omega$, there is a slice regular function $\widetilde{f}$ on $\Omega_2$ such that $f=\widetilde{f}$ in $\Omega_1$.
	\end{enumerate}
\end{defn}

In the sequel, we shall denote the set of imaginary units in the algebra of octonions by
\begin{equation*}\index{$\mathbb{S}_\mathbb{O}$}\index{$\mathcal{C}_\mathbb{O}$}
	\mathbb{S}_\mathbb{O}:=\{\Fi\in\mathbb{O}:\Fi^2=-1\},
\end{equation*}
and
\begin{equation*} \mathcal{C}_\mathbb{O}:=\{L_{\Fi}:\Fi\in\mathbb{S}_\mathbb{O}\}.
\end{equation*}

{
\begin{defn}\index{Special triple}
    Let $\Fi,\Fj,\Fk\in\mathbb{O}$ with $|\Fi|=|\Fj|=|\Fk|=1$ and
    \begin{equation*}
        (\Fi\Fj)\Fk=-\Fi(\Fj\Fk).
    \end{equation*}
    Then we say that $(\Fi,\Fj,\Fk)$ is a special triple of $\mathbb{O}$. We also say that $(L_\Fi,L_\Fj,L_\Fk)$ is a special triple of $L(\mathbb{O})$.
\end{defn}

\begin{prop}
    Let $(\Fi,\Fj,\Fk)$ be a special triple of $\mathbb{O}$. Then
    \begin{enumerate}[\upshape (i)]
        \item\label{it-ijk} $\Fi,\Fj,\Fk\in\mathbb{S}_\mathbb{O}$.
        \item\label{it-ijji} $\Fi\Fj=-\Fj\Fi$, $\Fj\Fk=-\Fk\Fj$ and $\Fk\Fi=-\Fi\Fk$.
    \end{enumerate}
\end{prop}

\begin{proof}
    If there is $\Fl\in\mathbb{S}_\mathbb{O}$ such that $\Fi,\Fj\in\mathbb{C}_L$, then
    \begin{equation*}\index{$\mathbb{H}_{L,K}$}
        \Fi,\Fj,\Fk\in\mathbb{H}_{\Fl,\Fk}:=\mathbb{R}\langle 1,\Fk,\Fl,\Fk\Fl\rangle.
    \end{equation*}

    It implies that
    \begin{equation*}
        (\Fi\Fj)\Fk=\frac{(\Fi\Fj)\Fk-\Fi(\Fj\Fk)}{2}=\frac{(\Fi\Fj)\Fk-(\Fi\Fj)\Fk}{2}=0,
    \end{equation*}
    a contradiction to $\Fi,\Fj,\Fk\neq 0$.

    Otherwise, $\mathbb{H}_{\Fi,\Fj}$ is an algebra of quaternions. Write
	\begin{equation*}
		\Fk=\Fk_1+\Fk_2 e_4',
	\end{equation*}
	for some $\Fk_1,\Fk_2\in\mathbb{H}'$ and $e_4'\in\mathbb{S}_\mathbb{O}\cap(\mathbb{H}_{L,K})^\perp$. Then
	\begin{equation*}
    	\begin{split}
			2=&|2(\Fi\Fj)\Fk|=|(\Fi\Fj)\Fk-\Fi(\Fj\Fk)|
            \\=&|(\Fi\Fj)(\Fk_2e_4')-\Fi[\Fj(\Fk_2e_4')]|\le 2|\Fk_2|\le2.
        \end{split}
	\end{equation*}
	It implies that
    \begin{equation*}
        |\Fk_2|=1,\qquad \Fk_1=0\qquad\mbox{and}\qquad \Fk=\Fk_2e_4'\in\mathbb{S}_\mathbb{O}\cap\mathbb{H}_{\Fi,\Fj}.
    \end{equation*}
    Then $\Fk\perp \Fi$ and $\Fk\Fi=-\Fi\Fk$. Similarly, \eqref{it-ijk} and \eqref{it-ijji} hold.
\end{proof}
}

\begin{rmk} We note that
	$(\Fi,\Fj,\Fk)$ is a special triple of $\mathbb{O}$ if and only if it is a triple of unit imaginary octonions such that each of them is orthogonal to the algebra generated by the other two.
\end{rmk}

\begin{prop}\label{pr-lij}
	Let $(\Fi,\Fj,\Fk)$ be a special triple of $\mathbb{O}$. For each $\theta\in[0,2\pi)$, let us set
    \begin{equation*}
		\begin{cases}
			\Fi'=\cos(\theta)\Fi+\sin(\theta)\Fj,\\
			\Fj'=\cos\left(\theta+\frac{\pi}{2}\right)\Fi+\sin\left(\theta+\frac{\pi}{2}\right)\Fj=-\sin(\theta)\Fi+\cos(\theta)\Fj.
		\end{cases}
	\end{equation*}
  Then  $(\Fi',\Fj',\Fk)$ and $(\Fi',-\Fj',\Fk)$ are special triples of $\mathbb{O}$.
\end{prop}

\begin{proof}
	This proposition holds by direct calculation and elementary trigonometric identities.
\end{proof}

The following Proposition \ref{eq-ab8} is proved in \cite{Moreno1998001}*{Proposition 2.11}. However, to avoid the use of several additional concepts in \cite{Moreno1998001},  we give another, direct proof. Note that we identify the algebra of octonions with the algebra generated by the basis elements $e_0,\ldots, e_7$ obeying the standard defining relations.

\begin{prop}\label{eq-ab8}
	Let $a+be_8,c+de_8$ be non-zero sedenions with $a,b,c,d\in\mathbb{O}$. Then
	\begin{equation}\label{eq-abe}
		(a+be_8)(c+de_8)=0,
	\end{equation}
	if and only if $a,b,c\neq 0$, $d=\frac{a(bc)}{|a||b|}=\frac{(ba)c}{|a||b|}$, $|a|=|b|$ and $\left\{\frac{a}{|a|},\frac{b}{|b|},\frac{c}{|c|}\right\}$ is a special triple of $\mathbb{O}$.
\end{prop}

{
\begin{proof}
    ``$\Leftarrow$" Suppose that $a,b,c\neq 0$, $d=a(bc)$, $|a|=|b|$ and $\left\{\frac{a}{|a|},\frac{b}{|b|},\frac{c}{|c|}\right\}$ is a special triple of $\mathbb{O}$. Then
	\begin{equation*}
		\left\{
		\begin{aligned}
			&\overline d=\frac{\overline{a(bc)}}{|a||b|}=-\frac{a(bc)}{|a||b|}=\frac{a(cb)}{|a||b|}=-\frac{(ac)b}{|b|^2},
			\\&d=\frac{a(bc)}{|a||b|}=-\frac{(bc)a}{|a|^2}.
		\end{aligned}
		\right.
	\end{equation*}
	It implies that
	\begin{equation}\label{eq-oo}
		\left\{
		\begin{aligned}
			&ac-\overline{d}b=ac-\left(-\frac{(ac)b}{|b|^2}\right) b=ac-ac=0,
			\\&da+b\overline c=\left(-\frac{(bc)a}{|a|^2}\right) a-bc=bc-bc=0.
		\end{aligned}
		\right.
	\end{equation}
	Then \eqref{eq-abe} holds by \eqref{eq-aell} and \eqref{eq-oo}.
	
	``$\Rightarrow$'' Suppose that \eqref{eq-abe} holds. Then by \eqref{eq-aell},
	\begin{equation}\label{eq-lb}
		\left\{
		\begin{aligned}
			&ac-\overline{d}b=0,
			\\&da+b\overline{c}=0.
		\end{aligned}
		\right.
	\end{equation}
	Since $\mathbb{O}$ is a normed division algebra,
	\begin{equation*}
		\left\{
		\begin{aligned}
			&|a||c|=|d||b|,
			\\&|d||a|=|b||c|.
		\end{aligned}
		\right.
	\end{equation*}
	Thence
	\begin{equation*}
		|a|^2\big(|c|^2+|d|^2\big)=|b|^2\big(|c|^2+|d|^2\big).
	\end{equation*}
	According to $c+de_8\neq 0$, we have $|c|^2+|d|^2\neq 0$. Thence
	\begin{equation}\label{eq-asb}
		|a|=|b|.
	\end{equation}
	Since $a+be_8$ is non-zero, we have $a,b\neq 0$. It follows that
	\begin{equation}\label{eq-am1}
		a^{-1}=\frac{\overline{a}}{|a|^2}\qquad\mbox{and}\qquad b^{-1}=\frac{\overline{b}}{|b|^2}.
	\end{equation}
	On the other hand, by \eqref{eq-lb},
	\begin{equation}\label{eq-obm}
		\overline{b^{-1}}(\overline{c}\cdot\overline{a})=d=-(b\overline{c})a^{-1}.
	\end{equation}
	By \eqref{eq-asb}, \eqref{eq-am1} and \eqref{eq-obm},
	\begin{equation*}
		b(\overline{c}\cdot\overline{a})=-(b\overline{c})\overline{a}.
	\end{equation*}
	By definition, $\left\{\frac{b}{|b|},\frac{\overline{c}}{|c|},\frac{\overline{a}}{|a|}\right\}$ is a special triple of $\mathbb{O}$. It implies that
	\begin{equation*}
		\overline{c}=-c\qquad\mbox{and}\qquad\overline{a}=-a.
	\end{equation*}
	Thus
	\begin{equation*}
		b(ca)=-(bc)a
	\end{equation*}
	and $\left\{\frac{b}{|b|},\frac{c}{|c|},\frac{a}{|a|}\right\}$ is also a basic triple of $\mathbb{O}$.

	By \eqref{eq-lb},
	\begin{equation*}
		d=-(b\overline c)a^{-1}=-\big[b(-c)\big]\frac{-a}{|a|^2}=-\frac{(bc)a}{|a|^2}=\frac{a(bc)}{|a||b|}.
	\end{equation*}
\end{proof}}

\begin{prop}\label{pr-ab}
	Let $(a_\ell+b_\ell e_8)(c+de_8)=0$, $\ell=1,2$, for some $a_\ell,b_\ell,c,d\in\mathbb{O}$. If $(c+d e_8)\neq 0$, then
	\begin{equation}\label{eq-la}
		\langle a_1,a_2\rangle=\langle b_1,b_2\rangle.
	\end{equation}
\end{prop}

\begin{proof}
	Obviously, \eqref{eq-la} holds for $a_1+b_1e_8=0$ or $a_2+b_2e_8=0$. Otherwise, by Proposition \ref{eq-ab8},
    \begin{equation*}
        d=\frac{(b_\ell a_\ell)c}{|a_\ell||b_\ell|},\qquad
        |a_\ell|=|b_\ell|,\qquad\mbox{and}\qquad |c_\ell|=|d_\ell|.
    \end{equation*}
    It follows from
    \begin{equation*}
        \left\{\frac{b_\ell}{|b_\ell|},\frac{b_\ell a_\ell}{|a_\ell||b_\ell|},\frac{c}{|c|}\right\}
        \qquad\mbox{and}\qquad
        \left\{\frac{b_\ell}{|b_\ell|},\frac{(b_\ell a_\ell)c}{|a_\ell||b_\ell||c|},\frac{c}{|c|}\right\}
    \end{equation*}
    being special triple of $\mathbb{O}$ that
    \begin{equation*}
        \begin{split}
            a_\ell=&[(-1)a_\ell](-1)=\left[\left(\frac{b_\ell}{|b_\ell|}\right)^2a_\ell\right]\left(\frac{c}{|c|}\right)^2
            =\left[\left(b_\ell\frac{b_\ell a_\ell}{|a_\ell||b_\ell|}\right)c\right]\frac{c}{|c||d|}
            \\=&-\left[b_\ell\frac{(b_\ell a_\ell)c}{|a_\ell||b_\ell|}\right]\frac{c}{|c||d|}
            =-\left[b_\ell d\right]\frac{c}{|c||d|}
            =\frac{b_\ell(dc)}{|c||d|}.
        \end{split}
    \end{equation*}
    Therefore
	\begin{equation*}
		\langle a_1,a_2\rangle=\left\langle b_1\frac{dc}{|c||d|},b_2\frac{dc}{|c||d|}\right\rangle=\langle b_1,b_2\rangle.
	\end{equation*}
\end{proof}

\section{Hyper-solutions in sedenions}\label{sc-hyper}

In \cite{Dou2023002}*{Section 6} we introduced the general definition of slice-solutions and hyper-solutions in the general case of $\mathbb R^{2n}$-valued functions. In this section, we fully describe these two concepts in the case of sedenions. Specifically, for a pair $ (J_1,J_2) \in\mathcal{C}_\mathfrak{S}^2$ with $ J_1 \ne J_2 $, we give three equivalent characterizations for this pair being a hyper-solution.
 In addition, this part of the paper prepares the study of  hyper-$\sigma$-balls which is provided in Section \ref{sc-hypersigmaball}.

Let $I\in\mathcal{C}_\mathfrak{S}$ and,  {according to \eqref{eq-lr},} we set
\begin{equation}\label{eq-alpha}\index{$\alpha_{[I]}$}
    \alpha_{[I]}:=\arccos\ \langle I, L_{e_8}\rangle.
\end{equation}

\begin{prop}\label{pr-fei}
	For each $I\in\mathcal{C}_\mathfrak{S}$, there is $\theta\in[0,\pi)$ and $\jmath\in\mathbb{S}_\mathbb{O}$ such that
	\begin{equation}\label{eq-ref}
		I=L\bigg(\sin\left(\alpha_{[I]}\right)\cos(\theta)\jmath+\left[\cos\left(\alpha_{[I]}\right)+\sin\left(\alpha_{[I]}\right)\sin(\theta)\jmath\right]e_8\bigg),
	\end{equation}
	where $L(s):=L_s$\index{$L(s)$}, for all $s\in\mathfrak{S}$  {and $\alpha_{[I]}$ as in \eqref{eq-alpha}}.
	
	Moreover, if $I\notin\{\pm L_{e_8}\}$, then the pair $(\theta,\jmath)$ is unique.
\end{prop}

\begin{proof}
	By \eqref{eq-sf},
	\begin{equation}\label{eq-abc}
		I=L\bigg(a\jmath+[b+c\jmath]e_8\bigg)
	\end{equation}
	for  {$\jmath\in\mathcal{C}_\mathbb{O}$ and some $a,b\in\mathbb{R}$, $c\in\mathbb{R}_+$ }  satisfying
    \begin{equation*}
        a\ge 0,\qquad\mbox{when}\qquad c=0.
    \end{equation*}
    Note that $a^2+b^2+c^2=1$. Then \eqref{eq-ref} holds by \eqref{eq-alpha} and taking
	\begin{equation}\label{eq-theta}
		\theta:=\begin{cases}
			0,\qquad&I\in\{\pm L_{e_8}\},\\
			\arccos\left(\frac{a}{
				 {\sin(\alpha_{[I]})}}\right),\qquad\qquad &\mbox{otherwise}.
		\end{cases}
	\end{equation}

	Suppose that $I\notin\{\pm L_{e_8}\}$ and
	\begin{equation*}
		I=L\bigg(\sin(\alpha_{[I]})\cos(\theta')\jmath'+\left[\cos(\alpha_{[I]})+\sin(\alpha_{[I]})\sin(\theta')\jmath'\right]e_8\bigg),
	\end{equation*}
	for some $\theta'\in[0,\pi)$ and $\jmath'\in\mathcal{C}_\mathbb{O}$. Since $I\notin\{\pm L_{e_8}\}$, $a\neq 0$ or $c\neq 0$. If $a\neq 0$, then
	\begin{equation*}
		\tan(\theta')=\frac{\sin(\theta')}{\cos(\theta')}=\frac{\sin(\alpha_{[I]})\sin(\theta')\jmath'}{\sin(\alpha_{[I]})\cos(\theta')\jmath'}=\frac{\sin(\alpha_{[I]})\sin(\theta)\jmath}{\sin(\alpha_{[I]})\cos(\theta)\jmath}=\frac{\sin(\theta)}{\cos(\theta)}=\tan(\theta).
	\end{equation*}
	We have
	\begin{equation}\label{eq-tt}
		\theta'=\theta,\qquad\mbox{and then}\qquad\jmath'=\jmath.
	\end{equation}
	Similarly if $c\neq 0$, then \eqref{eq-tt} holds since $\cot(\theta')=\cot(\theta)$.
\end{proof}

{\bf Notations}. When $I\notin\{\pm L_{e_8}\}$, by the uniqueness of the pair $(\theta,\jmath)$  we shall set
\begin{equation*}\index{$\theta_{[I]}$, $\jmath_{[I]}$}
	\begin{cases}
		\theta_{[I]}:=\theta,\\
		\jmath_{[I]}:=\jmath.
	\end{cases}
\end{equation*}
For $I\in\{\pm L_{e_8}\}$, we set
\begin{equation*}
	\begin{cases}
		\theta_{[I]}:=0,\\
		\jmath_{[I]}:=e_1.
	\end{cases}
\end{equation*}

Denote
\begin{equation*}\index{$\mathbb{S}^{2,\perp}_{\mathbb{O}}$}
	\mathbb{S}^{2,\perp}_{\mathbb{O}}:=\{(I,J)\in\mathbb{S}_\mathbb{O}^2:I\perp J\},
\end{equation*}
where the orthogonality is induced by the Euclidean inner product in \eqref{eq-lr}.

\begin{example}\label{exa-e1e10}
	By \eqref{eq-alpha}, \eqref{eq-abc} and \eqref{eq-theta},
    \begin{equation*}
        \begin{cases}
            \alpha_{L_{e_1}}=\arccos\langle L_{e_1}, L_{e_8}\rangle=\frac{\pi}{2},\\
            \jmath_{L_{e_1}}=e_1,\\
            \theta_{L_{e_1}}=0,
        \end{cases}
    \end{equation*}
    and
    \begin{equation*}
        \begin{cases}
            \alpha_{L_{e_{10}}}=\arccos\langle L_{e_{10}}, L_{e_8}\rangle=\frac{\pi}{2},\\
            \jmath_{L_{e_{10}}}=e_2,\\
            \theta_{L_{e_{10}}}=\frac{\pi}{2}.
        \end{cases}
    \end{equation*}
\end{example}

\begin{prop}\label{pr-ijj}
	If $J=(J_1,J_2)\in\mathcal{C}_\mathfrak{S}^2$ is not a slice-solution, then
    \begin{equation*}
        \index{$\alpha_{[J]}, \alpha_{[(J_1,J_2)]}$}
        \alpha_{[J_1]}=\alpha_{[J_2]}=:\alpha_{[J]}.
    \end{equation*}

	Moreover if $J_1\neq J_2$, then $\alpha_{[J]}\neq 0$, $\theta_{[J_1]}\neq\theta_{[J_2]}$ and there is {a unique element} in $\mathbb{S}_\mathbb{O}^2$ denoted by
	\index{$\imath^{[J]}$, $\imath_1^{[J]}$, $\imath_2^{[J]}$}
	$\imath^{[J]}=\left(\imath_1^{[J]},\imath_2^{[J]}\right)$ such that
	\begin{equation}\label{eq-js}
		\begin{cases}
			\jmath_{[J_1]}=\cos\left(\theta_{[J_1]}\right)\imath_1^{[J]}+\sin\left(\theta_{[J_1]}\right)\imath_2^{[J]},\\
			\jmath_{[J_2]}=\cos\left(\theta_{[J_2]}\right)\imath_1^{[J]}+\sin\left(\theta_{[J_2]}\right)\imath_2^{[J]}.
		\end{cases}
	\end{equation}
	Furthermore,
	\begin{equation}\label{eq-iperp}
		\imath_1^{[J]}\perp\imath_2^{[J]},\qquad\mbox{i.e.}\qquad \imath^{[J]}\in\mathbb{S}_\mathbb{O}^{2,\perp}.
	\end{equation}
\end{prop}

\begin{proof}
	Obviously, the statement holds when $J_1=J_2$, so we assume $J_1\neq J_2$.

    We divide the proof in steps.

	(i) We prove that $\alpha_{[J_1]}=\alpha_{[J_2]}=\alpha_{[J]}$.

    Since $(J_1,J_2)$ is not a slice-solution, it follows by the reasoning in \cite{Dou2021001}*{Remark 5.9} that there is non-zero $c=(c_1,c_2)^T\in\mathfrak{S}^2$ such that
	\begin{equation}\label{eq-p1}
		\begin{pmatrix}
			1&J_1\\1&J_2
		\end{pmatrix}
		\begin{pmatrix}
			c_1\\c_2
		\end{pmatrix}=0.
	\end{equation}
	It is immediate to check that $c_2\neq 0$, otherwise,
	by (\ref{eq-p1}), we have $c_1=0$,
		and then $c=0$, a contradiction, and
	\begin{equation}\label{eq-jj}
		(J_1-J_2)c_2=0.
	\end{equation}
	
	According to Proposition \ref{eq-ab8}, it follows from \eqref{eq-ref} that
	\begin{equation*}
		0=\langle J_1-J_2,L_{e_8}\rangle=\cos(\alpha_{[J_1]})-\cos(\alpha_{[J_2]}).
	\end{equation*}
	It implies that $\alpha_{[J_1]}=\alpha_{[J_2]}=:\alpha_{[J]}$.
	
	(ii) We prove that $\alpha_{[J]}\neq 0$.

    Since $J_1\neq J_2$, $\sin\left(\alpha_{[J]}\right)\neq 0$ (otherwise $J_1=J_2\in\{\pm L_{e_8}\}$) and the conclusion follows.

	(iii) We prove that if $J_1\not= J_2$ then $\theta_{[J_1]}\not= \theta_{[J_2]}$.

    We write
	\begin{equation*}
		J_1-J_2=L\bigg(o_1+o_2 e_8\bigg),\qquad\mbox{for some}\qquad o_1,o_2\in\mathbb{O}.
	\end{equation*}
	Again by Proposition \ref{eq-ab8} and \eqref{eq-jj}, $o_1\perp o_2$, i.e.
	\begin{equation}\label{eq-ll}
		\begin{split}
			&0=\langle o_1,o_2\rangle
			\\=&\sin^2(\alpha_{[J]})\left\langle\cos(\theta_{[J_1]})\jmath_{[J_1]}-\cos(\theta_{[J_2]})\jmath_{[J_2]},\sin(\theta_{[J_1]})\jmath_{[J_1]}-\sin(\theta_{[J_2]})\jmath_{[J_2]}\right\rangle.
		\end{split}
	\end{equation}
	
	It follows from \eqref{eq-ll} that
	\begin{equation*}
		\begin{split}
			\langle\jmath_{[J_1]},\jmath_{[J_2]}\rangle
			=&\frac{\cos(\theta_{[J_1]})\sin(\theta_{[J_1]})+\cos(\theta_{[J_2]})\sin(\theta_{[J_2]})}{\cos(\theta_{[J_2]})\sin(\theta_{[J_1]})+\cos(\theta_{[J_1]})\sin(\theta_{[J_2]})}
			\\=&\cos(\theta_{[J_1]})\cos(\theta_{[J_2]})+\sin(\theta_{[J_1]})\sin(\theta_{[J_2]})
            \\=&\cos(\theta_{[J_1]}-\theta_{[J_2]}).
		\end{split}
	\end{equation*}
	
	If $\theta_{[J_1]}=\theta_{[J_2]}$, then
	\begin{equation*}
		\left\langle \jmath_{[J_1]},\jmath_{[J_2]}\right\rangle=\cos\left(\theta_{[J_1]}-\theta_{[J_2]}\right)= 1.
	\end{equation*}
	It implies that $\jmath:=\jmath_{[J_1]}=\jmath_{[J_2]}$ and $o_1,o_2\in\mathbb{C}_{\jmath}$. By Proposition \ref{eq-ab8}, the non-zero sedenion $o_1+o_2 e_8$ is not a zero divisor, a contradiction to \eqref{eq-jj}. Therefore, $\theta_{[J_1]}\neq\theta_{[J_2]}$.
	
	(iv) We show that \eqref{eq-js} and \eqref{eq-iperp} hold.

    Let
	\begin{equation*}
		\left\{
            \begin{aligned}
                \imath_1^{[J]}:=&\cos\left(\theta_{[J_1]}\right)\jmath_{[J_1]}-\sin\left(\theta_{[J_1]}\right)\kappa,
                \\=&\cos\left(-\theta_{[J_1]}\right)\jmath_{[J_1]}+\sin\left(-\theta_{[J_1]}\right)\kappa,
			    \\\imath_2^{[J]}
                :=&\sin\left(\theta_{[J_1]}\right)\jmath_{[J_1]}
                +\cos\left(\theta_{[J_1]}\right)\kappa,
                \\=&\cos\left(-\theta_{[J_1]}+\frac{\pi}{2}\right)\jmath_{[J_1]}+\sin\left(-\theta_{[J_1]}+\frac{\pi}{2}\right)\kappa,
			\end{aligned}
		\right.
	\end{equation*}
	where
	\begin{equation*}
		\kappa:=\frac{\jmath_{[J_2]}-\cos\left(\theta_{[J_2]}-\theta_{[J_1]}\right)\jmath_{[J_1]}}{\sin\left(\theta_{[J_2]}-\theta_{[J_1]}\right)}\in\mathbb{S}_\mathbb{O},
	\end{equation*}
	i.e. $\kappa$ is the unique imaginary unit in $\mathbb{S}_\mathbb{O}$ with $\jmath_{[J_1]}\perp\kappa$ such that
	\begin{equation*}
		\jmath_{[J_2]}=\cos\left(\theta_{[J_2]}-\theta_{[J_1]}\right)\jmath_{[J_1]}+\sin\left(\theta_{[J_2]}-\theta_{[J_1]}\right)\kappa.
	\end{equation*}
	By directly calculation, \eqref{eq-js} and \eqref{eq-iperp} hold.
	
	(v) We prove the uniqueness of $\iota^{[J]}$.
    \\
    By \eqref{eq-js} and $\theta_{[J_1]}\neq\theta_{[J_2]}$,
	\begin{equation}\label{eq-ij}
		\imath^{[J]}=\begin{pmatrix}
			\imath_1^{[J]}\\\imath_2^{[J]}
		\end{pmatrix}
		=\begin{pmatrix}
			\cos\left(\theta_{[J_1]}\right)&\sin\left(\theta_{[J_1]}\right)\\
			\cos\left(\theta_{[J_2]}\right)&\sin\left(\theta_{[J_2]}\right)
		\end{pmatrix}^{-1}
		\begin{pmatrix}
			\jmath_{[J_1]}\\\jmath_{[J_2]}
		\end{pmatrix}.
	\end{equation}
	By Proposition \ref{pr-fei}, $\theta_{[J_\ell]}$ and $\jmath_{[J_\ell]}$, $\ell=1,2$, are unique. It implies by \eqref{eq-ij} that $\imath^{[J]}$ is unique.
\end{proof}

Although for the sake of consistency, the following example should be put after Corollary \ref{co-lj}, we write it here for the reader's convenience.
\begin{example}
By \eqref{eq-e1e10}, we know that $e_1-e_{10}$ is a  {zero divisor} and Corollary \ref{co-lj} shows that $J=(L_{e_1},L_{e_{10}})$ is not a slice-solution. By \eqref{eq-js} and Example \ref{exa-e1e10} we obtain
\begin{equation}\label{eq-alphae1e10}
    \alpha_{[(L_{e_1},L_{e_{10}})]}=\frac{\pi}{2}
\end{equation}
and
\begin{equation}\label{eq-imathe1e10}
    \begin{split}
        \imath^{[(L_{e_1},L_{e_{10}})]}=&
        \left[\begin{pmatrix}
            \cos\left(\theta_{L_{e_1}}\right) & \sin\left(\theta_{L_{e_1}}\right)\\
            \cos\left(\theta_{L_{e_{10}}}\right) & \sin\left(\theta_{L_{e_{10}}}\right)
        \end{pmatrix}^{-1}
        \begin{pmatrix}
            \jmath_{L_{e_1}}\\\jmath_{L_{e_{10}}}
        \end{pmatrix}\right]^T
        \\=&\left[\begin{pmatrix}
            1 & \\ & 1
        \end{pmatrix}^{-1}\begin{pmatrix}
            e_1 \\ e_2
        \end{pmatrix}\right]^T
        =(e_1,e_2).
    \end{split}
\end{equation}
\end{example}

Let \index{$\psi$}$\psi:[0,\pi]\times[0,\pi)\times \mathbb{S}^{2,\perp}_\mathbb{O}\rightarrow \mathcal{C}_\mathfrak{S}$ be defined by
\begin{equation}\label{eq-defpsi}\index{$\psi$}
	\psi\bigg(\alpha,\theta,(\imath_1,\imath_2)\bigg):= L\bigg(\sin(\alpha)\cos(\theta)\kappa+[\cos(\alpha)+\sin(\alpha)\sin(\theta)\kappa]e_8\bigg),
\end{equation}
where
\begin{equation*}
	\kappa:=\cos(\theta)\imath_1+\sin(\theta)\imath_2.
\end{equation*}

\begin{example}
	By direct calculations we have that
	\begin{equation}\label{eq-psib}
        \begin{split}
            & \psi \bigg(\frac{\pi}{2},\theta,(e_1,e_2)\bigg)
            \\=&L\left(\cos^2\theta e_1+\cos\theta\sin\theta e_2+[\sin\theta\cos\theta e_1+\sin^2\theta e_2]e_8\right).
        \end{split}
	\end{equation}
\end{example}

\begin{prop}
	Let $\alpha'\in(0,\pi)$, $\theta'\in[0,\pi)$, $\imath=(\imath_1,\imath_2)\in\mathbb{S}_\mathbb{O}^{2,\perp}$ and $J=\psi(\alpha,\theta',\imath)$. Then
	\begin{equation}\label{eq-pri}
		\alpha_{[J]}=\alpha',\qquad \theta_{[J]}=\theta'\qquad\mbox{and}\qquad \jmath_{[J]}=\cos(\theta')\imath_1+\sin(\theta')\imath_2.
	\end{equation}
\end{prop}

\begin{proof}
It is a consequence of \eqref{eq-defpsi} and Proposition \ref{pr-fei}.
\end{proof}

\begin{prop}
	If $J=(J_1,J_2)\in\mathcal{C}_\mathfrak{S}^2$ is not a slice-solution and $J_1\neq J_2$, then
	\begin{equation}\label{eq-nss}
		J_\ell=\psi\left(\alpha_{[J]},\theta_{[J_\ell]},\imath^{[J]}\right),\qquad\ell=1,2.
	\end{equation}
\end{prop}

\begin{proof}
	It follows directly by definition and formulas \eqref{eq-pri} and \eqref{eq-js}.
\end{proof}

\begin{prop}\label{pr-lemi1i2}
	Let $\alpha\in[0,\pi]$, $\theta_1,\theta_2\in[0,\pi)$, $(\imath_1,\imath_2)\in\mathbb{S}_\mathbb{O}^{2,\perp}$, and
    \begin{equation*}
        I_\ell=\psi(\alpha,\theta_\ell,(\imath_1,\imath_2)),\qquad \ell=1,2.
    \end{equation*}
    Then the following statements hold.
	\begin{enumerate}[\upshape (i)]
		\item \label{eq-i1neqi1} Let $\imath_3\in\mathbb{S}_\mathbb{O}$ be such that $\{\imath_1,\imath_2,\imath_3\}$ is a special triple of $\mathbb{O}$; then
		\begin{equation}\label{eq-i1i2zero}
			(I_1-I_2)(\imath_3+((\imath_1 \imath_2)\imath_3) e_8)=0.
		\end{equation}
		
		\item \label{pr-lemit2} If
		\begin{equation}\label{eq-i1neqi2}
			I_1\neq I_2,\qquad{\mbox{and}\qquad (I_1-I_2)c=0,\qquad\mbox{for some}\ c\in\mathfrak{S},}
		\end{equation}
		then
		\begin{equation}\label{eq-defd}
			c=d+((\imath_1 \imath_2)d) e_8
		\end{equation}
		for some $d\in\mathbb{O}$ with $\left\{\imath_1,\imath_2,\frac{d}{|d|}\right\}$ special triple of $\mathbb{O}$.
	\end{enumerate}
\end{prop}

\begin{proof}
	(i) If $\theta_1=\theta_2$ or $\sin(\alpha)=0$, then $I_1=I_2$ and \eqref{eq-i1i2zero} holds. Otherwise, $\theta_1\neq\theta_2$ and $\sin(\alpha)\neq 0$. Therefore, $I_1,I_2\notin\{\pm L_{e_8}\}$. By \eqref{eq-pri},
	\begin{equation}\label{diesis}
		\theta_{[I_\ell]}=\theta_{\ell},\qquad\mbox{and}\qquad\jmath_{[I_\ell]}=\cos(\theta_\ell)\imath_1+\sin(\theta_\ell)\imath_2.
	\end{equation}
	By \eqref{eq-ref} and direct calculation,
	\begin{equation}\label{eq-i1i2}
		I_1-I_2=L\bigg(\sin\left(\alpha\right)\left(a+b e_8\right)\bigg),
	\end{equation}
	where
	\begin{equation}\label{eq-aemi}
		\begin{split}
			a:=&\cos(\theta_{[I_1]})\jmath_{[I_1]}-\cos(\theta_{[I_2]})\jmath_{[I_2]}
			\\=&\left[\cos^2(\theta_{1})-\cos^2(\theta_{2})\right]\imath_1
			+\left[\cos(\theta_{1})\sin(\theta_{1})-\cos(\theta_{2})\sin(\theta_{2})\right]\imath_2,
		\end{split}
	\end{equation}
	and
	\begin{equation}\label{eq-bemi}
		\begin{split}
			b:=&\sin(\theta_{[I_1]})\jmath_{[I_1]}-\sin(\theta_{[I_2]})\jmath_{[I_2]}
			\\=&[\cos(\theta_{1})\sin(\theta_{1})-\cos(\theta_{2})\sin(\theta_{2})]\imath_1
			+[\sin^2(\theta_{1})-\sin^2(\theta_{2})]\imath_2.
		\end{split}
	\end{equation}
	If $a=0$, then by \eqref{eq-aemi}
	\begin{equation*}
		\begin{cases}
			\cos(2\theta_1)-\cos(2\theta_2)=2[\cos^2(\theta_{1})-\cos^2(\theta_{2})]=0,\\
			\sin(2\theta_1)-\sin(2\theta_2)= {2[\cos(\theta_{1})\sin(\theta_{1})-\cos(\theta_{2})\sin(\theta_{2})]=0}.
		\end{cases}
	\end{equation*}
	It implies that $\theta_1=\theta_2$, a contradiction, therefore, $a\neq 0$. Similarly, if we assume $b\neq 0$. It follows from
	\begin{equation*}
		\left[\cos^2(\theta_1)-\cos^2(\theta_2)\right]=-[\sin^2(\theta_1)-\sin^2(\theta_2)],\qquad \imath_2\imath_1\imath_2=-\imath_1\imath_2\imath_2=\imath_1,
	\end{equation*}
	that
	\begin{equation*}
		a\cdot\imath_1\cdot\imath_2=b,\qquad\mbox{and}\qquad|a|=|b|.
	\end{equation*}
	By $b^{-1}=\frac{-b}{|b|^2}\in\mathbb{S}_\mathbb{O}$,
	\begin{equation}\label{eq-ab}
		\frac{ba}{|a||b|}
		=\imath_1\imath_2,\qquad\mbox{and}\qquad \frac{(ba)\imath_3}{|a||b|}=(\imath_1\imath_2)\imath_3.
	\end{equation}
	By Proposition \ref{pr-lij}, $\left\{\frac{a}{|a|},\frac{b}{|b|},\imath_3\right\}$ is a special triple of $\mathbb{O}$. Then \eqref{eq-i1i2zero} holds by \eqref{eq-i1i2}, \eqref{eq-ab} and Proposition \ref{eq-ab8}.
	
	(ii) Suppose that \eqref{eq-i1neqi2} holds. By Proposition \ref{eq-ab8}, there is $d\in\mathbb{O}$ such that \eqref{eq-defd} holds and $\left\{\frac{a}{|a|},\frac{b}{|b|},\frac{d}{|d|}\right\}$ is a special triple of $\mathbb{O}$. According to \eqref{eq-aemi}, \eqref{eq-bemi} and Proposition \ref{pr-lij}, $\left\{\imath_1,\imath_2,\frac{d}{|d|}\right\}$ is also a special triple of $\mathbb{O}$.
\end{proof}

\begin{prop}\label{pr-st}
	Assume that $J=(J_1,J_2)\in\mathcal{C}_\mathfrak{S}^2$ is not a slice-solution and $J_1\neq J_2$, and let $K\in\mathcal{C}_\mathfrak{S}$. Then
	\begin{equation*}
		\mathcal{C}_\mathfrak{S}^{\ker}(J)=\psi\bigg(\left\{\alpha_{[J]}\right\}\times[0,\pi)\times\left\{\imath^{[J]}\right\}\bigg)
	\end{equation*}
	and the following statements are equivalent:
	\begin{enumerate}[\upshape (i)]
		\item $K=\psi(\alpha_{[J]},\theta_{[K]},\imath^{[J]})$;
		\item $K\in\mathcal{C}_\mathfrak{S}^{\ker}(J)$;
		\item $(J_1,J_2,K)$ is not a slice-solution.
	\end{enumerate}
\end{prop}

\begin{proof}
	
	(i)$\Rightarrow$(ii) Let $c=(c_1,c_2)^T\in\ker[\zeta(J)]\backslash\{0\}$. By definition and direct calculation, $c_2\neq 0$ and $(J_1-J_2)c_2=0$. According to \eqref{eq-iperp}, $\imath^{[J]}=\left(\imath_1^{[J]},\imath_2^{[J]}\right)\in\mathbb{S}_\mathbb{O}^{2,\perp}$. By \eqref{eq-nss} and Proposition \ref{pr-lemi1i2} \eqref{pr-lemit2}, there is $d\in\mathbb{O}$ such that
	\begin{equation}\label{eq-c2}
		c_2=d+\left((\imath_1^{[J]}\imath_2^{[J]})d\right)e_8
	\end{equation}
	and $\left\{\imath_1^{[J]},\imath_2^{[J]},\frac{d}{|d|}\right\}$ is a special triple of $\mathbb{O}$. According to \eqref{eq-i1i2zero} and \eqref{eq-c2}, $(K-J_1)c_2=0$. Therefore
	\begin{equation*}
		(1,K)c=(1,K)\begin{pmatrix}
			c_1\\c_2
		\end{pmatrix}=c_1+Kc_2=c_1+J_1c_2=(1,J_1)c=0.
	\end{equation*}
	It implies that $\ker(1,K)\supset\ker[\zeta(J)]$ and (ii) holds.
	
	(ii)$\Rightarrow$(iii) Suppose that $K\in\mathcal{C}_\mathfrak{S}^{\ker}(J)$. By definition,
	\begin{equation*}
		\ker\left(\zeta\left((J_1,J_2,K)\right)\right)=\ker(\zeta(J))\neq\{0\}.
	\end{equation*}
	According to \cite{Dou2022001}*{Remark 5.9}, $(J_1,J_2,K)$ is not a slice-solution.
	
	(iii)$\Rightarrow$(i) Let $I=\psi(\alpha_{[J]},\theta_{[K]},\imath^{[J]})$. By \eqref{eq-pri},
	\begin{equation}\label{eq-ik}
		\alpha_{[I]}=\alpha_{[J]},\qquad\mbox{and}\qquad\theta_{[I]}=\theta_{[K]}.
	\end{equation}
	It follows from the method in the proof of (i)$\Rightarrow$(ii) that $I\in\mathcal{C}_\mathfrak{S}^{\ker}(J)$. Suppose that $(J_1,J_2,K)$ is not a slice-solution. By \cite{Dou2021001}*{Remark 5.9}, there is non-zero $c=(c_1,c_2)^T\in\mathfrak{S}^{2\times 1}$ such that
	\begin{equation*}
		\begin{pmatrix}
			1&J_1\\1&J_2\\1&K
		\end{pmatrix}
		\begin{pmatrix}
			c_1\\c_2
		\end{pmatrix}=0.
	\end{equation*}
	Since $I\in\mathcal{C}_\mathfrak{S}^{\ker}(J)$,
	\begin{equation*}
		(1,I)c=0.
	\end{equation*}
	According to Proposition \ref{pr-ijj},
	\begin{equation}\label{eq-ak1}
		\alpha_{[K]}=\alpha_{[L]}=\alpha_{[J_1]}=\alpha_{[J]}=\alpha_{[I]}.
	\end{equation}
	If $K=I$, then (i) holds. Otherwise $K\neq I$, then $K-I=L\left(u+v e_8\right)$, where
	\begin{equation*}
		\begin{cases}
			u=\sin(\alpha_{[J]})\cos(\theta_{[K]})(\jmath_{[K]}-\jmath_{[I]}),\\
			v=\sin(\alpha_{[J]})\sin(\theta_{[K]})(\jmath_{[K]}-\jmath_{[I]}).
		\end{cases}
	\end{equation*}
	According to $(K-I)c_2=0$, $\left\{\frac{u}{|u|},\frac{v}{|v|},\frac{d_1}{|d_1|}\right\}$ is a special triple of $\mathbb{O}$, where $c_2=d_1+d_2 e_8$ for some $d_1,d_2\in\mathbb{O}$. It implies that $u,v\neq 0$, $\sin(\theta_{[K]}),\cos(\theta_{[K]})\neq 0$, and
	\begin{equation*}
		0=\left\langle\frac{u}{|u|},\frac{v}{|v|}\right\rangle
		=\sin^2(\alpha_{[J]})\cos(\theta_{[K]})\sin(\theta_{[K]})
		\left\langle\jmath_{[K]}-\jmath_{[I]},\jmath_{[K]}-\jmath_{[I]}\right\rangle.
	\end{equation*}
	Therefore $\jmath_{[K]}=\jmath_{[I]}$. By \eqref{eq-ik} and \eqref{eq-ak1} we  have $I=K$ hence (i) holds.
\end{proof}

\begin{example}\label{ex-cker}
 {Here we expand the calculations which led to Example \ref{EX24}. From Proposition \ref{pr-st}, \eqref{eq-alphae1e10}, \eqref{eq-imathe1e10} and \eqref{eq-psib}, we deduce}
	\begin{equation*}
        \begin{split}
	        &\mathcal{C}_{\mathfrak{S}}^{\ker}(L_{e_1},L_{e_{10}})
            \\=&\psi\bigg(\{\alpha_{[(L_{e_1},L_{e_{10}})]}\}\times[0,\pi)\times\left\{\imath^{[(L_{e_1},L_{e_{10}})]}\right\}\bigg)
            \\=&\psi\bigg(\left\{\frac{\pi}{2}\right\}\times[0,\pi)\times\{(e_1,e_2)\}\bigg)
            \\=&\bigg\{L\left(\cos^2\theta e_1+\cos\theta\sin\theta e_2+[\sin\theta\cos\theta e_1+\sin^2\theta e_2]e_8\right):\theta\in[0,\pi)\bigg\}.
        \end{split}
	\end{equation*}
    It is easy to check that
    \begin{equation*}
        \lim_{\theta\rightarrow \pi}\psi\left(\frac{\pi}{2},\theta,(e_1,e_2)\right)=L_{e_1}=\psi\left(\frac{\pi}{2},0,(e_1,e_2)\right).
    \end{equation*}
    Therefore, $\mathcal{C}_{\mathfrak{S}}^{\ker}(L_{e_1},L_{e_{10}})$ is a closed curve in $\mathcal{C}_\mathfrak{S}$.
\end{example}

\begin{cor}\label{co-lj}
	Let $J=(J_1,J_2)\in\mathcal{C}_\mathfrak{S}^2$ with $J_1\neq J_2$. Then the following statements are equivalent:
	\begin{enumerate}[\upshape (i)]
		\item $J$ is a hyper-solution;
		\item $J_1-J_2$ is a zero divisor, that is $J_1-J_2=L_a$, for some zero divisor $a\in\mathfrak{S}$;
		\item $J$ is not a slice-solution.
	\end{enumerate}
\end{cor}

\begin{proof}
	(i)$\Rightarrow$(ii) Suppose that $J$ is a hyper-solution. Let  $c=(c_1,c_2)^T\in\ker(\zeta(J))$ be a non-zero constant. Then it is easy to check that $c_2\neq 0$ and $(J_1-J_2)c_2=0$. Therefore, $J_1-J_2$ is a zero divisor.
	
	(ii)$\Rightarrow$(iii) Suppose that $J_1-J_2$ is a zero divisor. Then there is non-zero $c_2\in\mathcal{C}_\mathfrak{S}$ such that $(J_1-J_2)c_2=0$. Therefore $0\neq(-J_1c_2,c_2)^T\in\ker(\zeta(J))$. It follows from \cite{Dou2021001}*{Remark 5.9} that $J$ is not a slice-solution.
	
	(iii)$\Rightarrow$(i) Suppose that $J$ is not a slice-solution. Suppose that $J$ is not a hyper-solution. Then we can choose $ K\in\mathcal{C}_\mathfrak{S}\backslash\mathcal{C}_\mathfrak{S}^{\ker}(J)$ such that $(J_1,J_2,K)$ is not a slice-solution. By Proposition \ref{pr-st}, $ K\in\mathcal{C}_\mathfrak{S}^{\ker}(J)$, a contradiction.
\end{proof}

\section{Hyper-$\sigma$-balls in sedenions}\label{sc-hypersigmaball}

In this section, we show that there are only two types of hyper-$\sigma$-balls in sedenions.
One  {belongs to the plane associated with $ \{\pm e_8\} $, the other belongs to the planes associated with $ (J_1, J_2)\in\mathcal{C}_\mathfrak{S}^2$ that meets certain conditions}.
In other words, a hyper-$\sigma$-ball is either a ball extending in just one slice plane, or a ball extending in exactly two slice planes.
The former is consistent with a $\sigma$-ball, while the latter is
a new feature appearing in sedenions. We also show that both hyper-$\sigma$-balls and $\sigma$-balls are domains of regularity.

In the results below, for any $I\in\mathcal{C}_\mathfrak{S}$ we shall use the set defined by
\begin{equation*}\index{$\mathcal{H}(I)$}
	\mathcal{H}(I):=\{K\in \mathcal{C}_\mathfrak{S}\backslash\{I\}:(I,K)\mbox{ is a hyper-solution}\},
\end{equation*}
and  {recalling \eqref{diesis}, we set}
\begin{equation*}\index{$\mathbb{S}_\mathbb{O}(I)$}
	\mathbb{S}_\mathbb{O}(I):=\{\kappa\in\mathbb{S}_\mathbb{O}:\kappa\perp \jmath_{[I]}\}.
\end{equation*}
It is possible to check that there is $e_1',e_2',\cdots,e_6'\in\mathbb{S}_\mathfrak{S}$ such that $\{1,\jmath_{[I]},e_1',e_2',\cdots,e_6'\}$ is an orthonormal basis of $\mathbb{O}$. It follows that
\begin{equation}\label{eq-5sphere}
    \mathbb{S}_\mathbb{O}(I)=\left\{\sum_{\imath=1}^6\lambda_ie_i':\sum_{\imath=1}^6 \lambda_\imath^2=1\right\}
\end{equation}
is a $5$-sphere in the $8$-dimensional real Euclidean space $\mathbb{O}$.

\begin{rmk}
	Let $I\in\{\pm L_{e_8}\}$. It is easy to verify, using  Proposition \ref{pr-ijj}, that $\mathcal{H}(I)=\varnothing$, moreover $(I)$ is a hyper-solution.
\end{rmk}

Let $I\in\mathcal{C}_\mathfrak{S}\backslash\{\pm L_{e_8}\}$. We define a map $\varphi[I]:\mathbb{S}_\mathbb{O}(I)\times[0,\pi)\rightarrow\mathcal{C}_\mathfrak{S}$ by
\begin{equation}\label{eq-varphidef}\index{$\varphi[I]$}
	\varphi[I](\kappa,\vartheta):=\psi\left(\alpha_{[I]},\vartheta,\imath^{\{ I,\kappa\} }\right),
\end{equation}
where
\begin{equation}\label{eq-ii}\index{$\imath^{\{ I,\kappa\} }$}
	\imath^{\{ I,\kappa\} }=\begin{pmatrix}\imath_1^{\{ I,\kappa\} }\\\imath_2^{\{ I,\kappa\} }\end{pmatrix}=\begin{pmatrix}
		\cos(\theta_{[I]})&-\sin(\theta_{[I]})\\
		\sin(\theta_{[I]})&\cos(\theta_{[I]})
	\end{pmatrix}
	\begin{pmatrix}
		\jmath_{[I]}\\ \kappa
	\end{pmatrix}.
\end{equation}
By definition,
\begin{equation}\label{eq-kappa}
	I=\varphi[I](\kappa,\theta_{[I]})=\psi\left(\alpha_{[I]},\theta_{[I]},\imath^{\{ I,\kappa\} }\right),\qquad\forall\ \kappa\in\mathbb{S}_\mathbb{O}(I).
\end{equation}

\begin{prop}\label{pr-li}
	Let $I\in\mathcal{C}_\mathfrak{S}\backslash\{\pm L_{e_8}\}$. Then $\varphi[I]|_{E_{[I]}}$ is injective and
    \begin{equation}\label{eq-mhi}
        \mathcal{H}(I)=\varphi[I](	E_{[I]})
    \end{equation}
    is a $6$-manifold, where
	\begin{equation*}\index{$E_{[I]}$}
		E_{[I]}:=\mathbb{S}_\mathbb{O}(I)\times([0,\pi)\backslash\{\theta_{[I]}\}).
	\end{equation*}
\end{prop}

\begin{proof}
	(i) Let $J\in\varphi[I](E_{[I]})$. Then there is $(\kappa,\vartheta)\in E_{[I]}$ such that
	\begin{equation*}
		J=\varphi[I](\kappa,\vartheta)=\psi\left(\alpha_{[I]},\vartheta,\imath^{\{ I,\kappa\} }\right).
	\end{equation*}
	According to \eqref{eq-kappa} and Proposition \ref{pr-lemi1i2} \eqref{eq-i1neqi1}, $I-J$ is a zero divisor. It follows from $\theta_{[J]}=\vartheta\neq\theta_{[I]}$ and Proposition \ref{pr-ijj} that $J\neq I$.
	By Corollary \ref{co-lj}, $(I,J)$ is a hyper-solution, i.e. $J\in\mathcal{H}(I)$. Therefore
	\begin{equation*}
		\mathcal{H}(I)\supset\varphi[I](E_{[I]}).
	\end{equation*}
	
	(ii) Let $J\in\mathcal{H}(I)$. Then $(I,J)$ is a hyper-solution, and then not a slice-solution. It is easy to check by Proposition \ref{pr-ijj} that $\theta_{[I]}\neq\theta_{[J]}$ and
	\begin{equation*}
		J=\varphi[I](\kappa,\theta_{[J]}),
	\end{equation*}
	where
	\begin{equation*}
		\kappa=\sin\left(-\theta_{[I]}\right)\imath_1^{[(I,J)]}+\cos\left(-\theta_{[I]}\right)\imath_2^{[(I,J)]}.
	\end{equation*}
	Therefore $J\in\varphi[I](E_{[I]})$. It follows that $\mathcal{H}(I)\subset\varphi[I](E_{[I]})$ and \eqref{eq-mhi} holds.

	{
	(iii) Denote
	\begin{equation*}
		\begin{split}
			t_{[I]}:\ (\theta_{[I]}-\pi,\theta_{[I]})\ &\xlongrightarrow[\hskip1cm]{}\ [0,\pi)\backslash\{\theta_{[I]}\},
			\\ \vartheta\quad\qquad &\shortmid\!\xlongrightarrow[\hskip1cm]{}\ \begin{cases}
				\vartheta,\qquad\qquad\qquad&\vartheta\in[0,\theta_{[I]}),\\
            	\vartheta-\pi,&\mbox{otherwise},
			\end{cases}
		\end{split}
	\end{equation*}
	and
	\begin{equation*}
		\begin{split}
			T_{[I]}:\quad E'_{[I]}\ \ &\xlongrightarrow[\hskip1cm]{}\ \quad\ E_{[I]}\quad=\quad\mathbb{S}_\mathbb{O}(I)\times [0,\pi)\backslash\{\theta_{[I]}\},
			\\ (\kappa,\vartheta)\ &\shortmid\!\xlongrightarrow[\hskip1cm]{}\ \left(\kappa,t_{[I]}(\vartheta)\right),
		\end{split}
	\end{equation*}
	where
	\begin{equation*}
    	E'_{[I]}:=\mathbb{S}_\mathbb{O}(I)\times(\theta_{[I]}-\pi,\theta_{[I]}).
	\end{equation*}
	According to \eqref{eq-5sphere}, $E'_{[I]}$ is a $6$-manifold.

	It follows from  trigonometric identities and
	\begin{equation*}
    	\begin{split}
        	\sin(2\vartheta)=&\sin\big(2 t_{[I]}(\vartheta)\big),\\
        	\cos(2\vartheta)=&\cos\big(2 t_{[I]}(\vartheta)\big),
    	\end{split}
	\end{equation*}
	that
	\begin{equation*}
    	\begin{split}
        	\cos^2(t_{[I]}(\vartheta))=&\cos^2(\vartheta),\\
        	\sin(t_{[I]}(\vartheta))\cos(t_{[I]}(\vartheta))=&\sin(\vartheta)\cos(\vartheta),\\
        	\sin^2(t_{[I]}(\vartheta))=&\sin^2(\vartheta).
    	\end{split}
	\end{equation*}
	By definition, see \eqref{eq-defpsi}, \eqref{eq-varphidef} and \eqref{eq-ii},
	\begin{equation*}
		\begin{split}
	    	&\quad\varphi[I]\circ T_{[I]}(\kappa,\vartheta)
        	=\varphi[I](\kappa,t_{[I]}(\vartheta))
        	=\psi\left(\alpha_{[I]},T_{[I]}(\vartheta),\imath^{\{ I,\kappa\} }\right)
        	\\&
        	\begin{split}
            	=L\bigg(&\sin(\alpha_{[I]})\cos(t_{[I]}(\vartheta))
            	\big(\cos(t_{[I]}(\vartheta)),\sin(t_{[I]}(\vartheta))\big)g(\kappa)
            	\\&+\big[\cos(\alpha_{[I]})+\sin(\alpha_{[I]})\sin(t_{[I]}(\vartheta))\big(\cos(t_{[I]}(\vartheta)),\sin(t_{[I]}(\vartheta))\big)g(k)\big]e_8\bigg)
        	\end{split}\\&
        	\begin{split}
            	=L\bigg(&\sin(\alpha_{[I]})\cos(\vartheta)
            	\big(\cos(\vartheta),\sin(\vartheta)\big)g(\kappa)
            	\\&+\big[\cos(\alpha_{[I]})+\sin(\alpha_{[I]})\sin(\vartheta)
                \big(\cos(\vartheta),\sin(\vartheta)\big)g(k)\big]e_8\bigg),
        	\end{split}
		\end{split}
	\end{equation*}
	where
	\begin{equation*}
    	g(\kappa)=\begin{pmatrix}\cos(\theta_{[I]}) & -\sin(\theta_{[I]})\\\sin(\theta_{[I]}) & \cos(\theta_{[I]})\end{pmatrix}
    	\begin{pmatrix} \jmath_{[I]}\\ \kappa\end{pmatrix}.
	\end{equation*}
	Let
	\begin{equation*}
    	\vartheta_s=\arg_{2\cdot\theta_{[I]}} \left(\frac
    	{a_s\cdot c_s}
    	{a_s^2+c_s^2}+\frac{a_s^2-c_s^2}{a_s^2+c_s^2}\cdot i\right),
	\end{equation*}
	and
	\begin{equation*}
    	\kappa_s=
        	\frac{a_s\cdot\cos(\vartheta_s)+c_s\cdot\sin(\vartheta_s)-\sin(\alpha_{[I]})\cos(\vartheta_s-\theta_{[I]})\jmath_{[I]}}{\sin(\alpha_{[I]})\sin(\vartheta_s-\theta_{[I]})}.
	\end{equation*}
	where
	\begin{equation*}
    	a_s:=\sum_{\ell=1}^7\left\langle s,L_{e_{\ell}}\right\rangle e_\ell,
    	\qquad\mbox{and}\qquad
    	c_s:=\sum_{\ell=1}^7\left\langle s,L_{e_{\ell+8}}\right\rangle e_\ell.
	\end{equation*}
	and $\arg_{2\cdot\theta_{[I]}}$ is the angle of the complex number with the range of principal value $(2\cdot\theta_{[I]}-2\pi,2\cdot\theta_{[I]}]$.
	By direct calculation, $\varphi[I]\circ T_{[I]}$ is invertible and
	\begin{equation*}
    	\left(\varphi[I]\circ T_{[I]}\right)^{-1}(s)=(\vartheta_s,\kappa_s).
	\end{equation*}
	It implies that $\varphi[I]\circ T_{[I]}$ and $\left(\varphi[I]\circ T_{[I]}\right)^{-1}$ are continuous. According to \eqref{eq-mhi},
    \begin{equation}\label{eq-mhii}
        \mathcal{H}(I)=\varphi[I](E_{[I]})=\varphi[I]\circ T_{[I]}(E'_{[I]}).
    \end{equation}
    It implies that $\varphi[I]|_{E_{[I]}}$ is injective. Again by \eqref{eq-mhii}, $\mathcal{H}(I)$ and $E_{[I]}'$ are homeomorphic. Since $E_{[I]}'$ is a $6$-dimensional manifold, so is $\mathcal{H}(I)$.
	}
\end{proof}

\begin{defn}\index{Simple}
	$J=(J_1,J_2,...,J_m)\in\mathcal{C}_\mathfrak{S}^m$ is called simple if there are no $K_1,...,K_\ell\in\{J_1,J_2,...,J_m\}$ with $\ell<m$ such that
	\begin{equation*}
		\ker(\zeta(J))=\ker(\zeta(K)),
	\end{equation*}
	where $K=(K_1,...,K_\ell)$.
\end{defn}

Let $p,q\in\mathcal{W}_{\mathfrak{S}}$ and $\ell\in\mathbb{N}$. According to the definition of $*$-powers of a binomial, also in this setting we have:
\begin{equation*}
	(p+q)^{*\ell}=\sum_{\imath=0}^\ell\begin{pmatrix}
		\ell\\ \imath
	\end{pmatrix} p^\imath q^{\ell-\imath}.
\end{equation*}

\begin{rmk}\label{rm-ta}
i) There are only two kinds of simple hyper-solutions $J=(J_1,...,J_m)$ in $\mathcal{W}_\mathfrak{S}$. One is $(J_1)$, $J_1\in\{\pm L_{e_8}\}$ and the other is $(J_1,J_2)$, where $J_1\in\mathcal{C}_\mathfrak{S}\backslash\{\pm L_{e_8}\}$ and $J_2\in\mathcal{H}(J_1)$. Therefore, correspondingly, there are only two kinds of  hyper-$\sigma$-balls.

 \noindent ii) Reasoning as in \cite{Dou2021001}*{Proposition 7.9}, we deduce that every hyper-$\sigma$-ball is a domain of regularity.

\noindent ii) Let $J=(J_1,J_2)\in\mathcal{S}_\mathfrak{S}^2$ be a simple hyper-solution, i.e. $J_1\notin\{\pm L_{e_8}\}$ and $J_2\in\mathcal{H}(J_1)$.
 Let $ p \in\mathbb{C}_{J_1}$ and $ f: \Sigma(p,r,J)\rightarrow\mathfrak{S}$ be the slice regular function  defined as
	\begin{equation*}
		f(q)=\sum_{n\in\mathbb{N}}\left[\left(\frac{q-p}{r}\right)^{*2^n}c_2\right].
	\end{equation*}
where  $ c_2\in\mathfrak{S}\backslash\{0\}$ with $(J_1-J_2)c_2=0$ is fixed. Following the proof of \cite{Dou2021001}*{Proposition 7.9}, one may show that the domain of existence of $ f $ is just the hyper-$\sigma$-ball $\Sigma(p,r,J)$, that is, $ f $ cannot be extended to a slice regular function near any point in the boundary of $\Sigma (p,r,J)$.			
\end{rmk}

\begin{prop}
	Every $\sigma$-ball in $\mathcal{W}_\mathfrak{S}$ is a domain of slice regularity.
\end{prop}

\begin{proof}
	Let $\Sigma(p,r)$ be a $\sigma$-ball in $\mathcal{W}_\mathfrak{S}$ with center $p\in\mathcal{W}_\mathfrak{S}$ and radius $r\in (0,+\infty)$. If $p\in\mathbb{C}_{L_{e_8}}$, then $\Sigma(p,r)=\Sigma(p,r,L_{e_8})$ is a hyper-$\sigma$-ball. By Remark \ref{rm-ta}, $\Sigma(p,r)$ is a domain of regularity.
	
	Otherwise, assume that $p\in\mathbb{C}_{J_1}$ for some $J_1\in\mathcal{C}_\mathfrak{S}\backslash\{ \pm e_8\}$ and $p\notin\mathbb{R}$. By the reasoning in \cite{Dou2021001}*{Theorem 8.8},
there is a slice regular function $f:\Sigma(p,r)\rightarrow\mathfrak{S}$ defined by
	\begin{equation*}
		f(q)=\sum_{n\in\mathbb{N}}\left(\frac{q-p}{r}\right)^{*2^n}\cdot 1.
	\end{equation*}
	Let $J_2\in\mathcal{H}(J_1)$. Then $(J_1,J_2)$ is a hyper-solution and there is a non-zero $c_2\in\mathfrak{S}$ such that
	\begin{equation*}
		(J_1-J_2)c_2=0.
	\end{equation*}
	By Remark \ref{rm-ta}, there is a slice regular function $g:\Sigma(p,r,(J_1,J_2))\rightarrow\mathfrak{S}$ defined by
	\begin{equation*}
		g(q)=\sum_{n\in\mathbb{N}}\left[\left(\frac{q-p}{r}\right)^{*2^n}c_2\right].
	\end{equation*}
	For each $I\in\mathcal{S}_\mathfrak{S}$, $g_I$ cannot be extended to a slice regular function near any point in $\partial_I\Sigma_I(p,r,(J_1,J_2))$.
	
	On the other hand, by Proposition \ref{pr-st}, $\mathcal{C}_\mathfrak{S}^{\ker}\left((J_1,J_2)\right)$ is a $1$-manifold and by Proposition \ref{pr-li}, $\mathcal{H}(J_1)$ is a $6$-manifold. Therefore, we can choose
	\begin{equation*}
		J_2'\in\mathcal{H}(J_1)\backslash\mathcal{C}_\mathfrak{S}^{\ker}\left((J_1,J_2)\right)
	\end{equation*}
	and a non-zero $c_2'\in\mathfrak{S}$ such that
	\begin{equation*}
		(J_1-J_2')c_2'=0
	\end{equation*}
	Similarly, there is a slice regular function $h:\Sigma(p,r,(J_1,J_2'))\rightarrow\mathfrak{S}$ defined by
	\begin{equation*}
		h(q)=\sum_{n\in\mathbb{N}}\left[\left(\frac{q-p}{r}\right)^{*2^n}c_2'\right].
	\end{equation*}
	
	Note that $g=f\cdot c_2$ and $h=f\cdot c_2'$ on
	\begin{equation*}
		\Sigma(q,r)=\Sigma\left(p,r,(J_1,J_2)\right)\bigcup\Sigma\left(p,r,(J_1,J_2')\right).
	\end{equation*}
	It is clear that for each $I\in\mathcal{C}_\mathfrak{S}$, $f_I$ cannot be extended to a slice regular function near any point in
	\begin{equation*}
		\partial_I(\Sigma(q,r)\cap\mathbb{C}_I)=\partial_I(\Sigma(q,r,(J_1,J_2))\cap\mathbb{C}_I)\bigcup\partial_I(\Sigma(q,r,(J_1,J_2'))\cap\mathbb{C}_I),
	\end{equation*}
	since at least one of $g_I$ and $h_I$ cannot be extended.
\end{proof}

\section{Orthogonal decompositions of sedenions}\label{sc-ods}

In this section, we will study some orthogonality properties in the real vector space of sedenions. We will give two orthogonal decompositions \eqref{eq-od} and \eqref{eq-skipiq}, which are used to study the convergence of a $*$-power series in the next section.

\begin{prop}\label{pr-ijim}
	Let $I,J\in\mathcal{C}_{\mathfrak{S}}$. Then
	\begin{equation*}
		\ker(\zeta(I,J))=\left\{(-Ic,c)^T:c\in \ker(I-J)\right\}.
	\end{equation*}
\end{prop}

\begin{proof}
	In the proof we shall use the notation
	\begin{equation*}
		A:=\left\{(-Ic,c)^T:c\in \ker(I-J)\right\}.
	\end{equation*}
	(i) Let $(a,b)^T\in \ker(\zeta(I,J))$. Then
	\begin{equation*}
		\begin{pmatrix} 1 & I\\ 1& J \end{pmatrix}
		\begin{pmatrix} a\\b \end{pmatrix}
		=\begin{pmatrix} a+Ib\\a+Jb\end{pmatrix}
		=\begin{pmatrix} 0\\0 \end{pmatrix}.
	\end{equation*}
	It implies that $(I-J)b=0$, $b\in \ker(I-J)$ and $a=-Ib$. Therefore,
	\begin{equation*}
		(a,b)=(-Ib,b)\in A,\qquad\mbox{and}\qquad \ker(\zeta(I,J))\subset A.
	\end{equation*}
	
	(ii) Let $(-Ic,c)\in A$. Then $c\in \ker(I-J)$ and
	\begin{equation*}
		\begin{pmatrix} 1 & I\\ 1& J \end{pmatrix}
		\begin{pmatrix} -Ic\\c \end{pmatrix}
		=\begin{pmatrix} 0\\(-I+J)c\end{pmatrix}
		=\begin{pmatrix} 0\\0 \end{pmatrix}.
	\end{equation*}
	It implies that $(-Ic,c)^T\in \ker(\zeta(I,J))$ and $A\subset \ker(\zeta(I,J))$. The statement follows.
\end{proof}

Let $p,q\in \mathcal{W}_{\mathfrak{S}}\backslash\mathbb{R}$,  {$d\in\mathfrak{S}$, and let $(b,c)$ be the unique $2$-triple in $\ker(I_p-I_q)\times \left[\ker(I_p-I_q)\right]^\perp$ when $\{p,q\}\cap\mathbb{R}=\varnothing$.} Define
\begin{equation}\label{eq-parallel}\index{$d_{=}^{p,q}$}
	d_{=}^{p,q}:=\begin{cases}
		0,\qquad \{p,q\}\cap\mathbb{R}\neq\varnothing,\\
		b,\qquad \mbox{otherwise},
	\end{cases}
\end{equation}
\begin{equation}\index{$d_{\perp}^{p,q}$}\label{eq-paperp}
	d_{\perp}^{p,q}:=\begin{cases}
		d,\qquad \{p,q\}\cap\mathbb{R}\neq\varnothing,\\
		c,\qquad \mbox{otherwise},
	\end{cases}
\end{equation}
and
\begin{equation*}\index{$d_{\pm}^{p,q}$}
    d^{p,q}_\pm:=d-d_{=}^{p,q}-d_{=}^{p,-q}.
\end{equation*}

Moreover, let $a=\{a_\ell\}_{\ell\in\mathbb{N}}$ with $a_\ell\in\mathfrak{S}$. We shall write
\begin{equation*}
	\index{$a_{\ell,=}^{p,q}$, $a_{\ell,\perp}^{p,q}$, $a_{\ell,\pm}^{p,q}$}
	a_{\ell,=}^{p,q}:=(a_\ell)_{=}^{p,q},
    \qquad a_{\ell,\perp}^{p,q}:=(a_\ell)_{\perp}^{p,q},
    \qquad a_{\ell,\pm}^{p,q}:=(a_\ell)_{\pm}^{p,q}
\end{equation*}
for short.
{It is easy to check by $a_{\ell,=}^{p,\pm I_q}=a_\ell-a_{\ell,\perp}^{p,\pm I_q}\in\ker(I_p\mp I_q)$ that
\begin{equation}\label{eq-ipiq}
    (I_p\mp I_q) a_\ell=(I_p\mp I_q)a_{\ell,\perp}^{p,\pm I_q}.
\end{equation}
According to \eqref{eq-rapj},
\begin{equation}\label{eq-rapq}
    \begin{split}
        R_a^{p,I_q}
        =&\left[\limsup_{\ell\rightarrow +\infty} \dist\left(a_\ell,\ker(I_p-I_q)\right)^{\frac{1}{\ell}}\right]^{-1}
        \\=& \left[\limsup_{\ell\rightarrow+\infty}\left|a_{\ell,\perp}^{I_p,I_q}\right|^{\frac{1}{\ell}}\right]^{-1}
        =\left[\limsup_{\ell\rightarrow+\infty}\left|a_{\ell,\perp}^{p,q}\right|^{\frac{1}{\ell}}\right]^{-1}
        =R_{a_{\perp}^{p,q}}.
    \end{split}
\end{equation}
where $a_{\perp}^{p,q}:=\left\{a_{\ell,\perp}^{p,q}\right\}_{\ell\in\mathbb{N}}$.
}

\begin{prop}\label{pr-lpimw}
	Let $p\in\mathcal{W}_{\mathfrak{S}}$, $J_1,J_2\in\mathcal{C}_{\mathfrak{S}}$ and $a=\{a_\ell\}_{\ell\in\mathbb{N}}$ with $a_\ell\in\mathfrak{S}$ and $R_a^{p,J_1},R_a^{p,J_2}>R_a$. Then
	\begin{enumerate}[\upshape (i)]
		\item\label{it-pnmr} $p\notin\mathbb{R}$, $J_1,J_2\notin\{ I_p\}$,
        and $J_2\in\mathcal{C}_{\mathfrak{S}}^{\ker}(I_p, J_1)$.
		\item\label{it-rapj1} $R_a^{p,J_1}=R_a^{p,J_2}=R_a^p$,
                \begin{equation}\label{eq-kerJ1J2}
        	\ker(I_p-J_1)=\ker(I_p-J_2)\neq\{0\},
    	\end{equation}
        and
		\begin{equation}\label{eq-sprj1}
			\Sigma\big(p,r,(I_p,J_1)\big)=\Sigma\big(p,r,(I_p,J_2)\big),\qquad\forall\ r>0.
		\end{equation}
	\end{enumerate}
\end{prop}

\begin{proof}
	(i) By \eqref{eq-rapj}, $p\notin\mathbb{R}$ and $J_1,J_2\notin\{ I_p\}$. Suppose that $J_2\notin \mathcal{C}_{\mathfrak{S}}^{\ker}(I_p, J_1)$. Then
    \begin{equation*}
        \mathcal{C}_{\mathfrak{S}}^{\ker}(I_p, J_1)\neq \mathcal{C}_{\mathfrak{S}}^{\ker}(I_p, J_2).
    \end{equation*}
    By Corollary \ref{co-lj}, $(I_p,J_1,J_2)$ is a slice solution, i.e.
	\begin{equation*}
		\mathcal{C}_{\mathfrak{S}}^{\ker}(I_p, J_1)\neq \mathcal{C}_{\mathfrak{S}}^{\ker}(I_p, J_2)=\bigcap_{L=I_p,J_1,J_2} \ker(1,L)=\{0\}.
	\end{equation*}
	By \eqref{pr-ijim}, $\ker(I_p-J_1)\cap \ker(I_p-J_2)=\varnothing$. Denote
	\begin{equation*}
		\lambda:=\inf_{c\in \ker(I_p-J_1)\backslash\{0\}} \frac{dist(c,\ker(I_p-J_2))}{|c|}>0.
	\end{equation*}
	It follows from $R_a^{p,J_1}>R_a$, that
	\begin{equation*}
		\limsup_{\ell\rightarrow +\infty} \left|a_{\ell,\perp}^{p,J_1}\right|^{\frac{1}{\ell}}=\limsup_{\ell\rightarrow +\infty}\dist(a_\ell,\ker(I_p-J_1))^{\frac{1}{\ell}}=\frac{1}{R_a^{p,J_1}},
	\end{equation*}
	and
	\begin{equation*}
		\limsup_{\ell\rightarrow +\infty} \left|\left(a_{\ell,\perp}^{p,J_1}\right)^2+\left(a_{\ell,=}^{p,J_1}\right)^2\right|^{\frac{1}{2\ell}}=\limsup_{\ell\rightarrow +\infty}|a_\ell|^{\frac{1}{\ell}}=
        \begin{cases}
            \frac{1}{R_a},\qquad & R_a>0,\\ 0,\qquad & R_a=0.
        \end{cases}
	\end{equation*}
	Moreover
	\begin{equation*}
		\begin{split}
			\frac{1}{R_a^{p,J_2}}
			=&\limsup_{\ell\rightarrow +\infty} \dist(a_\ell,\ker(I_p-J_2))^{\frac{1}{\ell}}
			\\=&\limsup_{\ell\rightarrow +\infty} \dist\left(a_{\ell,\perp}^{p,J_1}+a_{\ell,=}^{p,J_1},\ker(I_p-J_2)\right)^{\frac{1}{\ell}}
			\\=&\limsup_{\ell\rightarrow +\infty} \dist\left(a_{\ell,=}^{p,J_1},\ker(I_p-J_2)\right)^{\frac{1}{\ell}}
			\\\ge& \limsup_{\ell\rightarrow +\infty}\left(\lambda \left|a_{\ell,=}^{p,J_1}\right|\right)^{\frac{1}{\ell}}
			=\limsup_{\ell\rightarrow +\infty}\left( \left|a_{\ell,=}^{p,J_1}\right|\right)^{\frac{1}{\ell}}
			\\=&\limsup_{\ell\rightarrow +\infty} \left|\left(a_{\ell,\perp}^{p,J_1}\right)^2+\left(a_{\ell,=}^{p,J_1}\right)^2\right|^{\frac{1}{2\ell}}
			=\frac{1}{R_a}
            =
            \begin{cases}
                \frac{1}{R_a},\qquad & R_a>0,\\ 0,\qquad & R_a=0,
            \end{cases}
		\end{split}
	\end{equation*}
	which contradicts $R_a^{p,J_2}>R_a$.
	
	(ii) By \eqref{it-pnmr}, $J_2\in\mathcal{C}_{\mathfrak{S}}^{\ker}(I_p, J_1)\backslash\{I_p\}$. It follows from Corollary \ref{co-lj} that
	\begin{equation}\label{eq-kzipj1}
		\ker(\zeta(I_p,J_1))=\ker(\zeta(I_p,J_2)).
	\end{equation}
	According to Proposition \ref{pr-ijim}, {\eqref{eq-kerJ1J2} holds.}
    By definition, $R_a^{p,J_1}=R_a^{p,J_2}$. According to \eqref{eq-rap}, we have $R_a^{p,J_1}=R_a^{p,J_2}=R_a^p$.
	
	Moreover \eqref{eq-msmfs} and \eqref{eq-kzipj1} give that $\mathcal{C}_{\mathfrak{S}}^{\ker}(I_p,J_1)=\mathcal{C}_{\mathfrak{S}}^{\ker}(I_p,J_2)$. It follows from \eqref{eq-sqrj} that \eqref{eq-sprj1} holds.
\end{proof}

\begin{cor}
	Let $p\in\mathcal{W}_{\mathfrak{S}}$ and $a=\{a_\ell\}_{\ell\in\mathbb{N}}$ with $a_\ell\in\mathfrak{S}$. Then
    \begin{equation}\label{eq-lrapk}
		\left\{R_a^{p,K}:K\in\mathcal{C}_{\mathfrak{S}}\right\}=\{R_a\}\cup\{R_a^{p}\}.
    \end{equation}
\end{cor}

\begin{proof}
	If $R_a=R_a^p$, then by \eqref{eq-rapj} and \eqref{eq-gek},
    \begin{equation*}
		R_a^{p,K}=R_a=R_a^p,\qquad\forall\ K\in\mathcal{C}_\mathfrak{S}.
    \end{equation*}
    Therefore \eqref{eq-lrapk} holds.

    Otherwise. $R_a\neq R_a^p$. By \eqref{eq-rapj}, $p\notin\mathbb{R}$ and $I_p\in\mathcal{C}_\mathfrak{S}$ is a complex structure on $\mathfrak{S}$. It implies that
    \begin{equation*}
        \ker(I_p-(-I_p))=\{0\},\qquad\mbox{and}\qquad R_a^{p,-I_p}=R_a.
    \end{equation*}

    Let $L\in\mathcal{C}_\mathfrak{S}$ with $R_a^L\neq R_a$. It follows from \eqref{eq-gek} that $R_a^{p,L}>R_a$. By Proposition \ref{pr-lpimw} \eqref{it-rapj1}, $R_a^{p,L}=R_a^{p,I}$. Then
    \begin{equation*}
        \left\{R_a,R_a^{p,I}\right\}
        =\left\{R_a^{p,-I_p},R_a^{p,I}\right\}
        \subset\left\{R_a^{p,L}:L\in\mathcal{C}_\mathfrak{S}\right\}
        \subset\left\{R_a,R_a^{p,I}\right\}.
    \end{equation*}
    It implies that \eqref{eq-lrapk} holds.
\end{proof}

\begin{cor}\label{cor-pws}
    Let $p\in\mathcal{W}_\mathfrak{S}$, $a=\{a_\ell\}_{\ell\in\mathbb{N}}$ with $a_\ell\in\mathfrak{S}$, $R_a>0$, $r_1\in (0,R_a)$ and $r_2\in(0,R_a^p)$. Then there is $m\in\mathbb{N}$ such that such that for each $\ell>m$ we have
    \begin{equation}\label{eq-r1}
        |a_\ell|<\frac{1}{r_1^\ell},
    \end{equation}
    and
    \begin{equation}\label{eq-r2}
        \left| a_{\ell,\perp}^{p,J}\right|<\frac{1}{r_2^\ell},\qquad\qquad \forall\ J\in\mathcal{C_\mathfrak{S}}\quad\mbox{with}\quad R_a^{p,J}=R_a^p.
    \end{equation}
\end{cor}

\begin{proof}
    By \eqref{eq-Ra},
    \begin{equation*}
        \limsup_{\ell\rightarrow+\infty} |a_\ell|^{\frac{1}{\ell}}=\frac{1}{R_a}<\frac{1}{r_1}.
    \end{equation*}
    It implies that there is $m_1\in\mathbb{N}$ such that \eqref{eq-r1} holds for $\ell>m_1$.

    If $R_a=R_a^p$, then $r_2<R_a^p=R_a$. By the same method as above, there is $m_2\in\mathbb{N}$ such that
    \begin{equation*}
        \left|a_{\ell,\perp}^{p,J}\right|\le |a_\ell|<\frac{1}{r_2^\ell},\qquad\qquad\forall\ \ell>m_2,\quad\mbox{and}\quad J\in\mathcal{C}_\mathfrak{S}.
    \end{equation*}
    It is clear that \eqref{eq-r1} and \eqref{eq-r2} holds by taking $m=\max\{m_1,m_2\}$.

    Otherwise, $R_a<R_a^p$. By \eqref{eq-lrapk}, there is $K\in\mathcal{C}_\mathfrak{S}$ such that $R_a^{p,K}=R_a^p$. Then there is $m_3\in\mathbb{N}$ such that
    \begin{equation}\label{eq-2ak}
        \left|a_{\ell,\perp}^{p,K}\right|< \frac{1}{r_2^\ell},\qquad\qquad\forall\ \ell>m_3.
    \end{equation}
    Let $K'\in\mathcal{C}_\mathfrak{S}$ with $R_a^{p,K'}=R_a^p$. According to \eqref{eq-kerJ1J2},
    \begin{equation*}
        \ker(I_p-K')=\ker(I_p-K).
    \end{equation*}
    By definition,
    \begin{equation*}
        a_{\ell,\perp}^{p,K'}
        =a_{\ell,\perp}^{p,K}.
    \end{equation*}
    Take $m=\max\{m_1,m_3\}$. Then \eqref{eq-r2} holds by \eqref{eq-2ak}.
\end{proof}

\begin{remark}
	$I,J\in\mathcal{C}_{\mathfrak{S}}$, $z^I=x+yI$ and $z^J=x+yJ$. Then
	\begin{equation*}
		(-I\pm J)I=1\pm JI=J(-J\pm I)=\mp J(-I\pm J).
		\end{equation*}
		It implies that
	\begin{equation}\label{eq-1jizi}
		\begin{split}
			&(1-JI)z^I=J(-J-I)(x+yI)=J(x+yJ)(-J-I)
			\\=&(x+yJ)J(-J-I)=(x+yJ)(1-JI)=z^J(1-JI),
		\end{split}
	\end{equation}
	and
	\begin{equation}\label{eq-1jizi+}
		\begin{split}
			&(1+JI)z^I=J(-J+I)(x+yI)=J(x-yJ)(-J+I)
			\\=&(x-yJ)J(-J+I)=(x-yJ)(1+JI)=z^{-J}(1+JI).
		\end{split}
	\end{equation}
	Similarly,
    \begin{equation}\label{eq-cozi}
        (J+I)z^I=z^J(J+I),\qquad\mbox{and}\qquad(J-I)z^I=z^{-J}(J-I).
    \end{equation}
\end{remark}

Let $T:\mathfrak{S}\rightarrow\mathfrak{S}$ be real linear with $T\neq 0$. Then $[\ker T]^\perp\neq \{0\}$. Denote
\begin{equation*}\index{$\left\Vert T\right\Vert_{inf}^{\perp}$}
	\left\Vert T\right\Vert_{inf}^{\perp}:=\inf_{a\in[\ker T]^\perp}\frac{|Ta|}{|a|}.
\end{equation*}
It is easy to check that
\begin{equation*}
    \left\Vert T\right\Vert_{inf}^{\perp}>0,\qquad\forall\ T\neq 0.
\end{equation*}

By \cite{Biss2009001}*{Lemma 4 and Theorem 8.3},
\begin{equation}\label{eq-mul}
	|ab|<\sqrt{2}|a||b|,\qquad\forall\ a,b\in\mathfrak{S}.
\end{equation}

Let $p=u+ve_8$ with $u,v\in\mathbb{O}$.

Denote
{
\begin{equation*}\index{ $p^{\oz}$,$p^{\oo}$,$p^{c_8}$}
    p^{\oz}:=u,
    \qquad p^{\oo}:=v,
    \qquad p^{c_8}:=u-ve_8,
\end{equation*}
and
\begin{equation*}
    (L_p)^{d}:=p^{d},\qquad d=\mathbb{O}_L,\mathbb{O}_R,c_8.
\end{equation*}
}

\begin{prop}\label{pr-zd}
    Let $p\neq 0$ be a zero-divisor. Then there is an orthogonal decomposition of sedenions with respect to the real inner product $\langle\cdot,\cdot\rangle$:
    \begin{equation}\label{eq-od}\index{$\mathbb{O}_p$}
        \mathfrak{S}=\mathbb{O}_p \oplus_\perp \ker p\oplus_\perp\ker p^{c_8},
    \end{equation}
    where
    \begin{equation*}
        \mathbb{O}_p:=\mathbb{H}_{p}+\mathbb{H}_{p}e_8,\qquad
        \mbox{and}\qquad\mathbb{H}_{p}:=\mathbb{R}+\mathbb{R}p^{\oz}+\mathbb{R}p^{\oo}+\mathbb{R}p^{\oz}p^{\oo}.
    \end{equation*}
    {
    Moreover,
    \begin{equation}\label{eq-lkpr}
        \ker p^{c_8}=\left(\ker p\right)^{c_8}:=\left\{q^{c_8}:q\in\ker p\right\}.
    \end{equation}
    }
\end{prop}
{What is $\left(\ker p\right)^{c_8}$?}
\begin{proof}
    Let $u:=p^{\oz}$ and $v:=p^{\oo}$. According to Proposition \ref{eq-ab8},
    {
    \begin{equation*}
        u,v\in\mathbb{R}\mathbb{S}_\mathbb{O}:=\left\{rI:r\in\mathbb{R},\ I\in\mathbb{S}_\mathbb{O}\right\},
    \end{equation*}
    \begin{equation*}
        |u|=|v|\qquad\mbox{and}\qquad u\perp v.
    \end{equation*}
    }
    Therefore, $\mathbb{H}_{p}$ is an isomorphic copy of the quaternionic algebra generated by $u,v$. It is immediate to check that $\mathbb{O}_p$ is an isomorphic copy of the octonion algebra produced by Cayley-Dickson construction using $\mathbb{H}_{p}$ and $e_8$.  Let $w\in\mathbb{S}_\mathbb{O}\cap\mathbb{H}_{p}^\perp$. It follows from the equalities
    \begin{equation*}
        \begin{split}
            &u'[v'(u'w)]=v'w,\\
            &u'[v'(v'w)]=-u'w,\\
            &u'[v'((u'v')w)]=w,
        \end{split}
    \end{equation*}
    and Proposition \ref{eq-ab8} that
    \begin{equation}\label{eq-kerkerc}
        \begin{aligned}
            \ker p={\mathrm{span}}_{\mathbb{R}}\bigg(&w+[u'(v' w)]e_8,u'w+(v'w)e_8,
            \\&v'w-(u'w)e_8,(u'v')w+we_8\bigg),
        \end{aligned}
    \end{equation}
    and
    \begin{equation*}
        \begin{aligned}
            \ker p^{c_8}={\mathrm{span}}_{\mathbb{R}}\bigg(&w-[u'(v' w)]e_8,u'w-(v'w)e_8,
            \\&v'w+(u'w)e_8,(u'v')w-we_8\bigg),
        \end{aligned}
    \end{equation*}
    where $u':=\frac{u}{|u|}$ and $v':=\frac{v}{|v|}$. It is now immediate the validity of \eqref{eq-od} and \eqref{eq-lkpr}.
\end{proof}

\begin{prop}\label{prop-keruv}
    Let $p\neq 0$ be a zero-divisor. Then
    \begin{equation}\label{eq-keruv}
        a\ker p\subset\ker p,\qquad \forall\ a\in \mathfrak{C}_p,
    \end{equation}
    {where
    \begin{equation*}\index{$\mathfrak{C}_p$}
        \mathfrak{C}_p
        :=L\bigg(\mathbb{R}+\mathbb{R}p^{\oz}+\mathbb{R}p^{\oo}
        +\mathbb{R}e_8+\mathbb{R}p^{\oz} e_8+\mathbb{R}p^{\oo} e_8\bigg).
    \end{equation*}
    }
\end{prop}

\begin{proof}
    Let $b\in\ker p\backslash\{0\}$, $u:=p^{\oz}$, $v:=p^{\oo}$, $u':=\frac{u}{|u|}$ and $v':=\frac{v}{|v|}$. By Proposition \ref{eq-ab8},
    \begin{equation*}
        |u|=|v|\neq 0,\qquad\qquad
        (u'+v'e_8)b=\frac{p^{c_8}b}{|u|}=0,
    \end{equation*}
    and
    \begin{equation*}
        b=\lambda(w+((v'u')w)e_8),\qquad\mbox{for some}\quad \lambda\in\mathbb{R}\backslash\{0\},\ w\in\mathbb{S}_{\mathbb{O}}\cap\mathbb{H}_p^\perp.
    \end{equation*}

    It is easy to check by Cayley Dickson construction \eqref{eq-aell} and $\{u',v',w\}$ being a special triple of $\mathbb{O}$ that
    \begin{equation*}
        \begin{split}
            e_8 b
            =& e_8[w+((v'u')w)e_8]
            =(v'u')w-w e_8
            \\=& (v'u')w-\left[(v'u')[(v'u')w]\right] e_8.
        \end{split}
    \end{equation*}
    By Proposition \ref{eq-ab8},
    \begin{equation}\label{eq-e8bb}
        e_8 b,b\in\ker p.
    \end{equation}

    Let $c\in\{u,v\}$. Then $\{\frac{c}{|c|},u'v',w\}$ be a special triple of $\mathbb{O}$. According to Cayley-Dickson construction \eqref{eq-aell},
    \begin{equation*}
            \frac{cb}{\lambda}
            = c(w+((v'u')w)e_8)
            = cw+[((v'u')w)c]e_8
            = cw+[(v'u')(cw)]e_8,
    \end{equation*}
    and
    \begin{equation*}
        \begin{split}
            \frac{(ce_8)b}{\lambda}
            =&[ce_8][w+((v'u')w)e_8]
            =((v'u')w)c-[cw]e_8
            \\=&(v'u')(cw)+\left[(v'u')[(v'u')(cw)]\right]e_8.
        \end{split}
    \end{equation*}
    By Proposition \ref{eq-ab8},
    \begin{equation}\label{eq-cbce8}
        cb,(ce_8)b\in\ker p,\qquad c=u,v.
    \end{equation}
    Then \eqref{eq-keruv} holds by \eqref{eq-e8bb} and \eqref{eq-cbce8}.
\end{proof}

\begin{prop}\label{pr-keruv}
    Let $p\neq 0$ be a zero-divisor. Then
    \begin{equation}\label{eq-kerperp}
        a[\ker p]^\perp\subset[\ker p]^\perp,\qquad\forall\ a\in\mathfrak{C}_p.
    \end{equation}
\end{prop}

\begin{proof}
    Let $a\in\mathfrak{C}_p\subset\mathbb{O}_p$. By the Cayley-Dickson construction \eqref{eq-aell},
    \begin{equation*}
        a \mathbb{O}_p\subset \mathbb{O}_p.
    \end{equation*}
    It follows from Proposition \ref{prop-keruv} and $\mathfrak{C}_p=\mathfrak{C}_{p^{c_8}}$ that $a \left(\ker p^{c_8}\right)\subset \ker p^{c_8}$. According to \eqref{eq-od},
    \begin{equation*}
        [\ker p]^\perp=\mathbb{O}_p\oplus\ker p^{c_8}.
    \end{equation*}
    It implies that \eqref{eq-kerperp} holds.
\end{proof}

\begin{prop}
    Let $I_1,I_2\in\mathcal{C}_\mathfrak{S}$ with $I_1\neq \pm I_2$, $\ker(I_1+I_2)\neq\{0\}$ and $\ker(I_1-I_2)\neq\{0\}$. Then
    \begin{equation}\label{eq-ki1i2k}
        \ker(I_1+I_2)=\ker\left[(I_1-I_2)^{c_8}\right].
    \end{equation}
\end{prop}

\begin{proof}
    Since $\ker(I_1-I_2)\neq\{0\}$, $I_1-I_2$ is a zero divisor. By Corollary \ref{co-lj}, $(I_1,I_2)$ is not a slice-solution. Denote
    \begin{equation}\label{eq-ieie}
        \imath_\ell^*:=\imath_\ell^{[(I_1,I_2)]},
        \qquad\mbox{and}\qquad\theta_\ell:=\theta_{[I_\ell]},\qquad\ell=1,2.
    \end{equation}
    According to \eqref{eq-ref} and Proposition \ref{eq-ab8},
    \begin{equation*}
        \cos(\alpha_{[I_1]})=\cos(\alpha_{[I_2]})
        =\cos(\alpha_{[(I_1,I_2)]})
        =\frac{\langle I_1+I_2,e_8 \rangle}{2}=0.
    \end{equation*}
    It follows from \eqref{eq-js} that
    \begin{equation*}
        I_\ell
        =L\bigg(\cos^2(\theta_\ell)\imath_1^*
        +\sin(\theta_\ell)\cos(\theta_\ell)\imath_2^*
        +\left[\sin(\theta_\ell)\cos(\theta_\ell)\imath_1^*
        +\sin^2(\theta_\ell)\imath_2^*\right]e_8\bigg).
    \end{equation*}
    By direct calculation,
    \begin{equation}\label{eq-ipiql}
        I_1+I_2=L\bigg(\lambda\imath_1^*+\mu\imath_2^*+[\mu\imath_1^*+(2-\lambda)\imath_2^*]e_8\bigg),
    \end{equation}
    where
    \begin{equation*}
        \lambda:=\cos^2(\theta_1)+\cos^2(\theta_2),\qquad\mbox{and}\qquad\mu:=\sin(\theta_1)\cos(\theta_1)+\sin(\theta_2)\cos(\theta_2).
    \end{equation*}
    According to Proposition \ref{eq-ab8},
    \begin{equation*}
        \lambda\imath_1^*+\mu\imath_2^*\perp \mu\imath_1^*+(2-\lambda)\imath_2^*.
    \end{equation*}
    By \eqref{eq-iperp},
    \begin{equation*}
        \sin(\theta_1+\theta_2)=\frac{2\mu}{2}=\frac{\lambda \mu+(2-\lambda)\mu}{2}
        =\frac{\langle \lambda\imath_1^*+\mu\imath_2^*,\mu\imath_1^*+(2+\lambda)\imath_2^*\rangle}{2}=0.
    \end{equation*}
    It is easy to check from \eqref{eq-ipiql} and Proposition \ref{eq-ab8} that
    \begin{equation*}
        \lambda=|\lambda\imath_1^*+\mu\imath_2^*|
        =|\mu\imath_1^*+(2-\lambda))\imath_2^*|=2-\lambda.
    \end{equation*}
    Then $\lambda=1$. It follows from $\imath_1^*\imath_2^*=-\imath_2^*\imath_1^*$, \eqref{eq-lkpr}, Proposition \ref{eq-ab8} and Proposition \ref{pr-lemi1i2} that
    \begin{equation*}
        \begin{split}
            \ker(I_1+I_2)
            =&\left\{\nu[\kappa+((\imath_2^*\imath_1^*)\kappa )e_8]:\nu\in\mathbb{R},\kappa\in\mathbb{S}_\mathbb{O}^{\imath_1^*,\imath_2^*}\right\}
            \\=&\left\{\nu[\kappa+((\imath_1^*\imath_2^*)\kappa )e_8]:\nu\in\mathbb{R},\kappa\in\mathbb{S}_\mathbb{O}^{\imath_1^*,\imath_2^*}\right\}^{c_8}
            \\=&[\ker(I_1-I_2)]^{c_8}
            =\ker\left[(I_1-I_2)^{c_8}\right],
        \end{split}
    \end{equation*}
    where
    \begin{equation*}\index{$\mathbb{S}_\mathbb{O}^{\imath_1^*,\imath_2^*}$}
        \mathbb{S}_\mathbb{O}^{\imath_1^*,\imath_2^*}
        :=\bigg\{J\in\mathbb{S}_\mathbb{O}:J\perp\{\imath_1^*,\imath_2^*,\imath_1^*\imath_2^*\}\bigg\}.
    \end{equation*}
\end{proof}

\begin{prop}
    Let $I_1,I_2\in\mathcal{C}_\mathfrak{S}$ with $I_1\neq\pm I_2$. Then
    \begin{equation}\label{eq-ki1i2i}
        \ker(I_1+I_2)\perp\ker(I_1-I_2).
    \end{equation}
\end{prop}

\begin{proof}
 We have that   \eqref{eq-ki1i2i} holds when $\ker(I_1+I_2)=\{0\}$ or $\ker(I_1-I_2)=\{0\}$. Therefore, we can assume that $\ker(I_1+I_2)\neq\{0\}$ and $\ker(I_1-I_2)\neq\{0\}$. Then \eqref{eq-ki1i2i} holds by \eqref{eq-od} and \eqref{eq-ki1i2k}.
\end{proof}

Let $p,q\in\mathcal{W}_\mathfrak{S}\backslash\mathbb{R}$ with $q\notin\mathbb{C}_{I_p}$. We set
\begin{equation}\index{$\mathfrak{S}_{p,q}$}\label{eq-spq}
    \mathfrak{S}_{p,q}:=[\ker(I_p-I_q)\oplus_\perp\ker(I_p+I_q)]^\perp.
\end{equation}
According to \eqref{eq-ki1i2i},
\begin{equation}\label{eq-skipiq}
    \mathfrak{S}:=\ker(I_p-I_q)\oplus_\perp\ker(I_p+I_q)\oplus_\perp\mathfrak{S}_{p,q},
\end{equation}
i.e., for each $d\in\mathfrak{S}$,
\begin{equation*}
    d=d_{=}^{p,q}+d_{=}^{p,-q}+d_{\pm}^{p,q},
\end{equation*}
with $d_{=}^{p,q}\in\ker(I_p-I_q)$, $d_{=}^{p,-q}\in\ker(I_p+I_q)$ and $d_{\pm}^{p,q}\in\mathfrak{S}_{p,q}$. It implies that
\begin{equation*}
    |d|^2=\left|d_{=}^{p,q}\right|^2+\left|d_{=}^{p,-q}\right|^2+\left|d_{\pm}^{p,q}\right|^2.
\end{equation*}

\begin{prop}
    Let $I,J\in\mathcal{C}_\mathfrak{S}$ with $\ker(I-J)\neq\{0\}$. Then \begin{equation}\label{eq-ijcij}
        I,J\in\mathfrak{C}_{I-J}.
    \end{equation}
\end{prop}

\begin{proof}
    It follows from Corollary \ref{co-lj} and $I-J$ being a zero divisor that $(I,J)$ is not a slice-solution. According to \eqref{eq-ref} and  \eqref{eq-js},
    \begin{equation*}
        L\bigg((I-J)^{\oz}-(I-J)^{\oo}e_8\bigg)
        =I-J
        =L\bigg(\lambda \imath_{1}^{*}+\mu\imath_2^*+[\mu\imath_1^*-\lambda\imath_2^*]e_8\bigg),
    \end{equation*}
    where
    \begin{equation*}
        \lambda:=\cos^2(\theta_1)-\cos^2(\theta_2),\qquad\mu:=\sin(\theta_1)\cos(\theta_1)-\sin(\theta_2)\cos(\theta_2),
    \end{equation*}
    and $\imath_\ell^*,\theta_\ell,\ell=1,2$ are defined by \eqref{eq-ieie}. Then
    \begin{equation*}
        \begin{pmatrix}
            \imath_1^*\\\imath_2^*
        \end{pmatrix}
        =\begin{pmatrix}
            \lambda &\mu\\\mu &-\lambda
        \end{pmatrix}^{-1}
        \begin{pmatrix}
            (I-J)^{\oz}\\
            (I-J)^{\oo}
        \end{pmatrix}.
    \end{equation*}
    It follows from $I,J\in L(\mathbb{R}\imath_1^*+\mathbb{R}\imath_2^*+\mathbb{R}\imath_1^*e_8+\mathbb{R}\imath_2^*e_8)$ that $I,J\in\mathfrak{C}_{I-J}$.
\end{proof}

\begin{prop}
    Let $p,q\in\mathcal{W}_\mathfrak{S}$, $a\in\{1,I_p,I_q\}$ and $d\in\mathfrak{S}$. Then
    \begin{equation}\label{eq-aaa}
        a A\subset A,\qquad\qquad A=\ker(I_p-I_q),\ker(I_p+I_q),\mathfrak{S}_{p,q}.
    \end{equation}
\end{prop}

\begin{proof}
    (i) If $\ker(I_p-I_q),\ker(I_p+I_q)\neq\{0\}$, then by \eqref{eq-od}, \eqref{eq-ki1i2k} and \eqref{eq-spq},
    \begin{equation*}
        \ker(I_p+I_q)=\ker\left[(I_p-I_q)^{c_8}\right],\qquad\mbox{and}\qquad \mathfrak{S}_{p,q}=\mathbb{O}_{I_p-I_q}.
    \end{equation*}
    According to \eqref{eq-keruv} and \eqref{eq-ijcij},
    \begin{equation}\label{eq-aaa1}
        a A\subset A,\qquad\qquad A=\ker(I_p-I_q),\ker(I_p+I_q).
    \end{equation}
    \eqref{eq-ijcij} also implies that $I_p,I_q\in\mathfrak{C}_{I_p-I_q}\subset\mathbb{O}_{I_p-I_q}$. Then
    \begin{equation}\label{eq-aaa2}
        a \mathfrak{S}_{p,q}=a\mathbb{O}_{I_p-I_q}\subset\mathbb{O}_{I_p-I_q}=\mathfrak{S}_{p,q}.
    \end{equation}
    Hence \eqref{eq-aaa} holds by \eqref{eq-aaa1} and \eqref{eq-aaa2}.

    (ii) If $\ker(I_p-I_q)=\ker(I_p+I_q)=\{0\}$, then $\mathfrak{S}_{p,q}=\mathfrak{S}$. It implies that \eqref{eq-aaa} holds.

    (iii) If $\ker(I_p-I_q)=\{0\}$ and $\ker(I_p+I_q)\neq\{0\}$, then $\mathfrak{S}_{p,q}=[\ker(I_p+I_q)]^\perp$. By \eqref{eq-ijcij}, $I_p,-I_q\in\mathfrak{C}_{I_p+I_q}$. It follows from \eqref{eq-keruv} and \eqref{eq-kerperp} that \eqref{eq-aaa} holds. Similarly, \eqref{eq-aaa} holds when $\ker(I_p-I_q)\neq\{0\}$ and $\ker(I_p+I_q)=\{0\}$.
\end{proof}

\begin{prop}
    Let $p,q\in\mathcal{W}_\mathfrak{S}\backslash\mathbb{R}$, $a\in\{1,I_p,I_q\}$ and $d\in\mathfrak{S}$. Then
    \begin{equation}\label{eq-adpq}
        (ad)_{=}^{p,q}=a\cdot d_{=}^{p,q},\qquad (ad)_{=}^{p,-q}=a\cdot d_{=}^{p,-q},\qquad (ad)_{\pm}^{p,q}=a\cdot d_{\pm}^{p,q},
    \end{equation}
    and
    \begin{equation}\label{eq-adppq}
        (ad)_{\perp}^{p,q}=a\cdot d_\perp^{p,q}
    \end{equation}
\end{prop}

\begin{proof}
    It follows from \eqref{eq-aaa} and
    \begin{equation*}
        d_{=}^{p,q}\in\ker(I_p-I_q),\qquad d_{=}^{p,-q}\in\ker(I_p+I_q),\qquad d_{\pm}^{p,q}\in\mathfrak{S}_{p,q}
    \end{equation*}
    that
    \begin{equation*}
        ad_{=}^{p,q}\in\ker(I_p-I_q),\qquad ad_{=}^{p,-q}\in\ker(I_p+I_q),\qquad ad_{\pm}^{p,q}\in\mathfrak{S}_{p,q}.
    \end{equation*}
    Then \eqref{eq-adpq} holds by
    \begin{equation*}
        ad=ad_{=}^{p,q}+ad_{=}^{p,-q}+ad_{\pm}^{p,q}.
    \end{equation*}
     Since \eqref{eq-adpq} and $d_\perp^{p,q}=d_{=}^{p,-q}+d_\pm^{p,q}$, \eqref{eq-adppq} holds.
\end{proof}

\section{The domain of convergence of power series}\label{sc-domainofconvergence}

In \cite{Gentili2012001}, Gentili and Stoppato study quaternionic power series computed with the $*$-product. The domain of convergence of power series in one quaternionic variable is a so-called $\sigma$-ball. As it is well known, a $\sigma$-ball is not an open set in Euclidean space, and in order to have an open convergence set, one should use the slice topology instead. In the case of sedenions, there is an additional new feature appearing, namely the domain of convergence of power series is the intersection of a hyper-$\sigma$-ball and a $\sigma$-ball. This new feature is discussed in this section.

\begin{prop}\label{pr-=}
Let $I,J\in\mathcal{C}_\mathfrak{S}$, $z\in\mathbb{C}$, $\ell\in\mathbb{N}$ and $a\in\mathfrak{S}$. Then monomials with the variable $q\in\mathcal{W}_\mathfrak{S}$ coincide:
	\begin{equation}\label{eq-lqzi}
		(q-z^I)^{*\ell} a_{=}^{I,J}=(q-z^J)^{*\ell} a_{=}^{I,J}.
	\end{equation}
\end{prop}

\begin{proof}
	If $p\in\mathbb{R}$, then by definition $a_{\ell,=}^{p,J}=0$. Hence \eqref{eq-lqzi} holds.
	
	Otherwise, when $p\notin\mathbb{R}$ there are $b,c:\mathbb{R}\rightarrow\mathbb{R}$ such that
	\begin{equation*}
		(q-z^L)^{*\ell}=(q-z^L)^\ell=b(q)+c(q)L,\qquad\forall\ q\in\mathbb{R}\mbox{ and } L\in\{I,J\}.
	\end{equation*}
	Then
	\begin{equation}\label{eq-qzll}
		\begin{split}
			&(q-z^I)^{*\ell} a_{=}^{I,J}-(q-z^J)^{*\ell} a_{=}^{I,J}
			\\=&\left[b(q)+c(q)I\right]a_{=}^{I,J}-\left[b(q)+c(q)J\right]a_{=}^{I,J}
			\\=& c(q)(I-J)a_{=}^{I,J}
			=0,
		\end{split}
	\end{equation}
	i.e., \eqref{eq-lqzi} holds for $q\in\mathbb{R}$. Since both sides of \eqref{eq-lqzi} are slice regular
    with the variable $q$, it follows from Identity Principle \cite{Dou2023002}*{Theorem 3.5} and \eqref{eq-qzll}
    that \eqref{eq-lqzi} holds.
\end{proof}

Let $r\in[0,+\infty)$ and $x+yI\in\mathcal{W}_\mathfrak{S}$ with $x,y\in\mathbb{R}$ and $I\in\mathcal{C}_\mathfrak{S}$. Denote
\begin{equation*}\index{$B_I^*(x+yI,r)$,$B_I(x+yI,r)$}
    B_I^*(x+yI,r):=\Psi_i^I(B^*(x+yi,r)).
\end{equation*}
We shall abbreviate by writing $B_I^*(x+yI,r)$ instead of $B_I(x+yI,r)$, when $r>0$.

\begin{prop}\label{pr-lipms}
	Let $I_p\in\mathcal{C}_\mathfrak{S}$, $p\in\mathbb{C}_I$, and $a=\{a_\ell\}_{\ell\in\mathbb{N}}$ with $a_\ell\in\mathfrak{S}$. Then the power series
	\begin{equation*}
		\begin{split}
			P:\ \mathbb{C}_I\ &\xlongrightarrow[\hskip1cm]{}\ \mathfrak{S},
			\\ q\ &\shortmid\!\xlongrightarrow[\hskip1cm]{}\  {P(q)=}\sum_{\ell\in\mathbb{N}} (q-p)^{\ell} a_\ell.
		\end{split}
	\end{equation*}
	converges absolutely in $B^*_I(p,R_a)$ and diverges in ${\ext}_I\left(B^*_I(p,R_a)\right)$, where \index{${\ext}_I$}${\ext}_I:={\ext}_{\tau(\mathbb{C}_I)}$.
    Moreover, let $r\in(0,R_a)$. Then the power series converges uniformly in ${\clo}_I\left(B_I(p,r)\right)$, where \index{${\clo}_I$}${\clo}_I:={\clo}_{\tau(\mathbb{C}_I)}$.
\end{prop}

\begin{proof}
	Since $I$ is a complex structure on $\mathfrak{S}$, there is $\theta_1,...,\theta_8\in\mathfrak{S}$ such that $\{\theta_1,...,\theta_8,I(\theta_1),...,I(\theta_8)\}$ is a real basis of $\mathfrak{S}$. For each $\ell\in\mathbb{N}$, write
	\begin{equation}\label{eq-aes}
		a_\ell=\sum_{\imath=1}^8 (c_\imath)_\ell\theta_\imath,\qquad\mbox{for some}\qquad (c_\imath)_\ell\in\mathbb{C}_I.
	\end{equation}
	Then
	\begin{equation*}
		\max_{\imath=1,...,8}(|(c_\imath)_\ell|)\le |a_\ell|=\left(\sum_{\imath=1}^8|(c_\imath)_\ell|^2\right)^{\frac{1}{2}}\le 8\max_{\imath=1,...,8}(|(c_\imath)_\ell|).
	\end{equation*}
	It is easy to check that
	\begin{equation}\label{eq-mi18}
		\min_{\imath=1,...,8}(R_{c_\imath})=R_a,
	\end{equation}
	where $c_\imath:=\{(c_\imath)_\ell\}_{\ell\in\mathbb{N}}$. Then the power series
	\begin{equation*}
		\begin{split}
			P_\imath:\ \mathbb{C}_I\ &\xlongrightarrow[\hskip1cm]{}\ \mathfrak{S},
			\\ q\ &\shortmid\!\xlongrightarrow[\hskip1cm]{}\  {P_\imath(q)=} \sum_{\ell\in\mathbb{N}} (q-p)^{\ell} (c_\imath)_\ell.
		\end{split}
	\end{equation*}
	converges absolutely in $B^*_I(p,R_{c_\imath})$ and diverges in ${\ext}_I\left(B^*_I(p,R_a)\right)$. Moreover, let $r\in(0,R_{c_\imath})$. Then the power series converges uniformly in ${\clo}_I\left(B_I(p,r)\right)$ of $B_I(p,r)$. It follows from \eqref{eq-mi18} and \eqref{eq-aes} that the statement holds.
\end{proof}

\begin{prop}\label{pr-3diverges}
	Let $I_p\in\mathcal{C}_\mathfrak{S}$, $p_\imath\in\mathbb{C}_I$, and $a^\imath=\{a_\ell^\imath\}_{\ell\in\mathbb{N}}$ with $a_\ell^\imath\in\mathfrak{S}$
    and $R_{a^\imath}>0$, $\imath=1,2,3$. Then the power series
	\begin{equation*}
		\begin{split}
			P:\ \mathbb{C}_I\ &\xlongrightarrow[\hskip1cm]{}\ \mathfrak{S},
			\\ q\ &\shortmid\!\xlongrightarrow[\hskip1cm]{}\  {P(q)=} \sum_{\ell\in\mathbb{N}}\sum_{\imath=1}^3 \left[(q-p_\imath)^{\ell} a_\ell^\imath\right]
		\end{split}
	\end{equation*}
	converges absolutely in $\mathbb{B}:=\cap_{\imath=1}^3 B_I(p_\imath,R_{a^\imath})$, and diverges in the exterior $\ext_{\tau(\mathbb{C}_I)}(\mathbb{B})$ of $\mathbb{B}$ unless \begin{equation}\label{eq-ler_i}
		\limsup_{\ell\rightarrow +\infty} |q-p_\imath||a_\ell^\imath|^{\frac{1}{\ell}}=\limsup_{\ell\rightarrow +\infty} |q-p_\jmath||a_\ell^\jmath|^{\frac{1}{\ell}}=\max_{\kappa=1,2,3}\left( \limsup_{\ell\rightarrow +\infty} |q-p_\kappa||a_\ell^\kappa|^{\frac{1}{\ell}}\right)
	\end{equation}
	for some $\imath,\jmath\in\{1,2,3\}$ with $\imath\neq\jmath$.
\end{prop}

\begin{proof}
	$P$ converges absolutely in $\mathbb{B}$ directly by Proposition \ref{pr-lipms}. Let $q$ in the exterior of $\mathbb{B}$ such that \eqref{eq-ler_i} does not hold. Then there is $\imath\in\{1,2,3\}$ and $c_1\in\mathbb{R}$ such that
	\begin{equation*}
		c_0:=\limsup_{\ell\rightarrow +\infty} |q-p_\imath||a_\ell^\imath|^{\frac{1}{\ell}} >c_1> \limsup_{\ell\rightarrow +\infty} |q-p_\jmath||a_\ell^\jmath|^{\frac{1}{\ell}}>1,\qquad\jmath\neq\imath.
	\end{equation*}
	Let $c_2\in(c_1,c_0)$. There is $\ell>1$ such that $\big(\frac{c_2}{c_1}\big)^\ell>3$ and
	\begin{equation*}
		|q-p_\imath||a_\ell^\imath|^{\frac{1}{\ell}}>c_2>c_1>|q-p_\jmath||a_\ell^\jmath|^{\frac{1}{\ell}},\qquad\jmath\neq\imath.
	\end{equation*}
	It implies that
	\begin{equation*}
	 	\frac{1}{3}|q-p_\imath|^\ell|a_\ell^\imath|>\frac{1}{3}(c_2)^\ell>c_1^\ell>|q-p_\jmath|^\ell|a_\ell^\jmath|,\qquad\jmath\neq\imath.
	\end{equation*}
	Therefore
	\begin{equation*}
		\left|\sum_{m=1}^3 \left[(q-p_m)^{\ell} a_\ell^m\right]\right|\ge \left|\left[(q-p_\imath)^{\ell} a_\ell^\imath\right]\right|-\frac{2}{3}\left|\left[(q-p_\imath)^{\ell} a_\ell^\imath\right]\right|\ge\frac{1}{3} (c_2)^\ell\ge\frac{1}{3}.
	\end{equation*}
	It implies that $P$ diverges at $q$, and Proposition \ref{pr-3diverges} holds.
\end{proof}

\begin{prop}
	Let $p,q\in\mathcal{W}_{\mathfrak{S}}\backslash\mathbb{R}$, and $a=\{a_\ell\}_{\ell\in\mathbb{N}}$ with $a_\ell\in\mathfrak{S}$ and $0<R_a<R_a^p$. Then
	\begin{equation}\label{eq-oaeq}
		R_o=R_a,
	\end{equation}
	where $o=\{o_\ell\}_{\ell\in\mathbb{N}}$ with $o_\ell:=a_{\ell,=}^{p,q}$.
\end{prop}

\begin{proof}
	Let $k\in\mathbb{N}$, and $c_1,c_2,c_3\in\mathbb{R}$ with
	\begin{equation*}
		\limsup_{\ell\rightarrow +\infty} |a_{\ell,\perp}^{p,q}|^{\frac{1}{\ell}}=\frac{1}{R_a^p}<c_1<c_2<c_3<\frac{1}{R_a}=\limsup_{\ell\rightarrow +\infty} |a_\ell|^{\frac{1}{\ell}}.
	\end{equation*}
	Then there is $\ell>k$ such that
	\begin{equation*}
		|a_{\ell,\perp}^{p,q}|<(c_1)^\ell<(c_2)^\ell<(c_3)^\ell<|a_\ell|,\qquad\mbox{and}\qquad (c_1)^\ell+(c_2)^\ell<(c_3)^\ell.
	\end{equation*}
	It implies that
	\begin{equation*}
		|a_{\ell,=}^{p,q}|=|a_\ell-a_{\ell,\perp}^{p,q}|\ge|a_\ell|-|a_{\ell,\perp}^{p,q}|> (c_3)^\ell-(c_1)^\ell>(c_2)^\ell.
	\end{equation*}
	Hence $\frac{1}{R_o}\ge c_2$, for each $c_2\in\left(\frac{1}{R_a^p},\frac{1}{R_a}\right)$. Therefore, $R_o\le R_a$. By $|a_{\ell,=}^{p,q}|<|a_\ell|$, $R_o\ge R_a$. We have $R_o=R_a$.
\end{proof}

\begin{prop}\label{pr-im}
    Let $I\in\mathbb{S}$, $p_1,p_2\in\mathbb{C}_I$, $a=\{a_\ell\}_{\ell\in\mathbb{N}}$, $b=\{b_\ell\}_{\ell\in\mathbb{N}}$ and $c=\{c_\ell\}_{\ell\in\mathbb{N}}$ with $a_\ell,b_\ell,c_\ell\in\mathfrak{S}$, $R_b>R_a>0$, $R_c>0$ and $\mathbb{C}_I a_\ell\perp \mathbb{C}_I c_\ell$. Then the power series
    \begin{equation*}
		\begin{split}
			P:\ \mathbb{C}_I\ &\xlongrightarrow[\hskip1cm]{}\ \mathfrak{S},
			\\ q\ &\shortmid\!\xlongrightarrow[\hskip1cm]{}\  {P(q)=}\sum_{\ell\in\mathbb{N}} (q-p_1)^{\ell} (a_\ell+b_\ell)+(q-p_2)^{\ell} (c_\ell),
		\end{split}
	\end{equation*}
	converges in $\mathbb{B}:=B_I(p_1,R_a)\cap B_I(p_2, R_c)$ and diverges in the exterior of $\mathbb{B}$ in $\tau(\mathbb{C}_I)$.
\end{prop}

\begin{proof}
By $R_b>R_a$,
\begin{equation}\label{eq-ls}
    \limsup_{\ell\rightarrow +\infty} |q-p_1|^\ell|a_\ell|>\limsup_{\ell\rightarrow +\infty} |q-p_1|^\ell|b_\ell|.
\end{equation}
According to Proposition \ref{pr-3diverges}, we only need to prove that $P(q)$ diverges when $q$ is in the exterior of $\mathbb{B}$ and
\begin{equation*}
    \limsup_{\ell\rightarrow +\infty} |q-p_1|^\ell|a_\ell|=\limsup_{\ell\rightarrow +\infty} |q-p_2|^\ell|c_\ell|.
\end{equation*}
Let $R_a>c_2>c_1>R_a$ and $k\in\mathbb{N}$. According to \eqref{eq-ls}, there is $\ell\ge k$ such that
\begin{equation*}
    |q-p_1|^\ell|a_\ell|>c_2^\ell>c_1^\ell>|q-p_1|^\ell|b_\ell|, \qquad\mbox{and}\qquad\left(\frac{c_2}{c_1}\right)^\ell>2.
\end{equation*}
Then
\begin{equation*}
    \begin{split}
        &\left|(q-p_1)^\ell(a_\ell+b_\ell)+(q-p_2)^\ell c_\ell\right|
        \\ \ge&\left|(q-p_1)^\ell a_\ell+(q-p_2)^\ell c_\ell\right|-\left|(q-p_1)^\ell b_\ell \right|
        \\ = & \sqrt{\left|(q-p_1)^\ell a_\ell\right|^2+\left|(q-p_2)^\ell c_\ell\right|^2}-\left|(q-p_1)^\ell b_\ell \right|
        \\ \ge & |q-p_1|^\ell|a_\ell|-|q-p_1|^\ell|b_\ell|\ge c_2^\ell-c_1^\ell \ge \frac{c_2^\ell}{2}.
    \end{split}
\end{equation*}
Therefore, $P(q)$ diverges and the statement holds.
\end{proof}

\begin{prop}\label{pr-ljim}
    Let $J\in\mathcal{C}_\mathfrak{S}$, $p\in\mathcal{W}_\mathfrak{S}\backslash\mathbb{C}_J$, $a=\{a_\ell\}_{\ell\in\mathbb{N}}$ with $a_\ell\in\mathfrak{S}$ and $R_a>0$. Then the power series
    \begin{equation*}
		\begin{split}
			P:\ \mathbb{C}_{J}^+\ &\xlongrightarrow[\hskip1cm]{}\ \mathfrak{S},
			\\ q\ \ &\shortmid\!\xlongrightarrow[\hskip1cm]{}\ \sum_{\ell\in\mathbb{N}} (q-p)^{*\ell} (a_\ell),
		\end{split}
	\end{equation*}
	converges in
    \begin{equation*}
        \mathbb{B}:=B_J\left(z^{J},R_a\right)\cap B_J\left(z^{-J}, R_a^{p,J}\right)\cap \mathbb{C}_{J}^+
    \end{equation*}
    and diverges in the exterior $\ext_{\tau(\mathbb{C}_J^+)}(\mathbb{B})$ of $\mathbb{B}$ in $\tau\left(\mathbb{C}_{J}^+\right)$, where
    \begin{equation*}
        \mathbb{C}_J^+:=\{x+yJ\in\mathbb{C}_J: y>0\}.
    \end{equation*}
\end{prop}

\begin{proof}
	Let $q\in\mathbb{C}_J^+$. Then $I_q=J$. There is $z,w\in\mathbb{C}$ such that
	\begin{equation*}
		p=z^{I_p},\qquad\mbox{and}\qquad q=w^{I_q}.
	\end{equation*}
	According to \eqref{eq-1jizi}, \eqref{eq-1jizi+}, and the representation formula, see \cite{Dou2023002}*{Theorem 6.6}, we have
	\begin{equation}\label{eq-sqpea}
		\begin{split}
			&(q-p)^{*\ell} a_{\ell,\perp}^{p,q}
			=\left(w^{I_q}-z^{I_p}\right)^{*\ell} a_{\ell,\perp}^{p,q}
			\\=&\frac{1-I_qI_p}{2} \left(w^{I_p}-z^{I_p}\right)^\ell a_{\ell,\perp}^{p,q}+\frac{1+I_qI_p}{2}\left(w^{-I_p}-z^{I_p}\right)^\ell a_{\ell,\perp}^{p,q}
			\\=& \left(w^{I_q}-z^{I_q}\right)^\ell\frac{1-I_qI_p}{2} a_{\ell,\perp}^{p,q}+\left(w^{I_q}-z^{-I_q}\right)^\ell \frac{1+I_qI_p}{2} a_{\ell,\perp}^{p,q}.
		\end{split}
	\end{equation}
Let us set
	\begin{equation}\label{eq-cldl}
		c_\ell:=\frac{1-I_qI_p}{2} a_{\ell,\perp}^{p,q},\qquad
		\mbox{and}\qquad d_\ell:=\frac{1+I_qI_p}{2} a_{\ell,\perp}^{p,q}.
	\end{equation}
	By \eqref{eq-mul},
	\begin{equation*}
		|c_\ell|=\left|I_q\frac{-I_q-I_p}{2}\cdot a_{\ell,\perp}^{p,q}\right|\le\sqrt{2}\left|\frac{-I_q-I_p}{2}\cdot a_{\ell,\perp}^{p,q}\right|\le 2|a_{\ell,\perp}^{p,q}|.
	\end{equation*}
	It implies that $R_c\ge R_a^{p,J}$. Since $\ker I_q=\{0\}$ and $a_{\ell,\perp}^{p,q}\perp \ker(-I_q+I_p)$,
	\begin{equation*}
		\begin{split}
			2|a_{\ell,\perp}^{p,q}|\ge|d_\ell|=\left|I_q\frac{-I_q+I_p}{2}\cdot a_{\ell,\perp}^{p,q}\right|
			\ge &\left\Vert I_q \right\Vert_{inf}^\perp\left|\frac{-I_q+I_p}{2}\cdot a_{\ell,\perp}^{p,q}\right|
			\\\ge&\left\Vert I_q \right\Vert_{inf}^\perp\left\Vert \frac{-I_q+I_p}{2} \right\Vert_{inf}^\perp\left| a_{\ell,\perp}^{p,q}\right|.
		\end{split}
	\end{equation*}
	Therefore
	\begin{equation}\label{eq-rdra}
		R_d=R_a^{p,J}\le R_c.
	\end{equation}
    According to Proposition \ref{pr-=} and \eqref{eq-sqpea},
	\begin{equation*}
		\begin{split}
			&\sum_{\ell\in\mathbb{N}} (q-p)^{*\ell} a_\ell
			=\sum_{\ell\in\mathbb{N}} \left(q-z^{I_p}\right)^{*\ell} a_{\ell,=}^{p,q}+\sum_{\ell\in\mathbb{N}} \left(q-z^{I_p}\right)^{*\ell} a_{\ell,\perp}^{p,q}
			\\=&\sum_{\ell\in\mathbb{N}} \left(q-z^{I_q}\right)^{\ell} o_\ell+\sum_{\ell\in\mathbb{N}} \left(q-z^{I_q}\right)^{\ell} c_\ell+\sum_{\ell\in\mathbb{N}} \left(q-z^{-I_q}\right)^{\ell} d_\ell,
		\end{split}
	\end{equation*}
	where
    \begin{equation}\label{eq-oell}
        o_\ell:=a_{\ell,=}^{p,q}.
    \end{equation}
    It follows from Proposition \ref{pr-3diverges} and
	\begin{equation*}
	    q\in\mathbb{C}^+_{I_q}:=\{\lambda+\mu I_q:\mu\ge 0\},
	\end{equation*}
	that the series $P$ converges absolutely in
	\begin{equation*}
		\mathbb{B}=B_{I_q}\left(z^{I_q},R_a\right)\cap B_{I_q}\left(z^{-I_q},R_a^{p,J}\right)\cap\mathbb{C}_{I_q}^+,
	\end{equation*}
	and diverges in $\ext_{\mathbb{C}_J^+}(\mathbb{B})$, unless
	\begin{equation}\label{eq-qunless}
		\frac{\left|q-z^{I_p}\right|}{\left|q-z^{-I_q}\right|}
		=\frac{\limsup\limits_{\ell\rightarrow +\infty}|d_\ell|^{\frac{1}{\ell}}}{\limsup\limits_{\ell\rightarrow +\infty}|o_\ell|^{\frac{1}{\ell}}}
		=\frac{R_o}{R_d}.
	\end{equation}
    We now need to prove that $P$ diverges, when
    \begin{equation}\label{eq-q}
        q\in{\ext}_{\tau\left(\mathbb{C}_J^+\right)}(\mathbb{B}),\qquad\mbox{and}\qquad\mbox{\eqref{eq-qunless} holds.}
    \end{equation}
    We assume that \eqref{eq-q} holds.

    (i) If $R_a^{p,J}=R_a$, then by \eqref{eq-oaeq} and \eqref{eq-rdra},
    \begin{equation*}
        R_d=R_a^{p,J}=R_a=R_o.
    \end{equation*}
    It implies by \eqref{eq-qunless} that $|q-z^{I_p}|=|q-z^{-I_q}|$ and $q\in\mathbb{R}$. Therefore $q\notin\mathbb{C}_J^+$, a contradiction.

    (ii) Otherwise, $R_a^{p,J}>R_a$. According to \eqref{eq-lrapk}, $R_a^{p,J}=R_a^p>R_a$. By \eqref{eq-oaeq} and \eqref{eq-rdra},
    \begin{equation*}
        R_o=R_a<R_a^p\le R_c.
    \end{equation*}
    According to \eqref{eq-rapj}, $\ker(I_p-I_q)\neq\{0\}$ and $I_p-I_q$ is a zero divisor. By \eqref{eq-parallel},
    \begin{equation*}
        o_\ell=a^{p,q}_{\ell,=}=(a_\ell)^{p,q}_{=}\in\ker(I_p-I_q).
    \end{equation*}
    It follows from Proposition \ref{prop-keruv} that
    \begin{equation*}
        \left(q-z^{I_q}\right)^{\ell} o_\ell\in\ker(I_p-I_q).
    \end{equation*}
    Similarly, by $a^{p,q}_{\ell,\perp}\in[\ker(I_p-I_q)]^\perp$ and Proposition \ref{pr-keruv},
    \begin{equation*}
        \begin{split}
            \left(q-z^{-I_q}\right)^{\ell} d_\ell=&\left(q-z^{-I_q}\right)\left(\frac{1+I_qI_p}{2} a_{\ell,\perp}^{p,q}\right)
            \\=&\left(q-z^{-I_q}\right)\left[\frac{-I_p+I_q}{2}\left(I_p \cdot a_{\ell,\perp}^{p,q}\right)\right]\in\left[\ker(I_p-I_q)\right]^\perp.
        \end{split}
    \end{equation*}
    By Proposition \ref{pr-im} and \eqref{eq-q}, $P$ diverges where stated.
\end{proof}
Next result clarifies a geometric description of the set  $ \Sigma(p,a)\cap\mathbb{C}_I^+$:
\begin{prop}
    Let $p\in\mathcal{W}_{\mathfrak{S}}$, $I\in\mathcal{C}_{\mathfrak{S}}$ and $a=\{a_\ell\}_{\ell\in\mathbb{N}}$ with $a_\ell\in\mathfrak{S}$ and $p\notin\mathbb{C}_I$. Then
    \begin{equation}\label{eq-sigmapa}
        \Sigma(p,a)\cap\mathbb{C}_I^+=B_I\left(z^{I},R_a\right)\cap B_I\left(z^{-I}, R_a^{p,I}\right)\cap \mathbb{C}_{I}^+.
    \end{equation}
\end{prop}

\begin{proof}
    Since $p=z^J$ for some $z\in\mathbb{C}$ and $J\in\mathcal{C}_\mathfrak{S}$.

    (i) If $R_a=R_a^p$, then by \eqref{eq-rap},
    \begin{equation*}
        R_a^{p,I}=R_a=R_a^p.
    \end{equation*}
    According to \eqref{eq-sigmapr} and \eqref{eq-sigmapab},
    \begin{equation*}
        \begin{split}
            \Sigma(p,a)\cap\mathbb{C}_I^+=\Sigma(p,R_a)\cap\mathbb{C}_I^+
            =&\Psi_i^I \left[B(z,R_a)\cap B(\overline{z}, R_a)\right]\cap\mathbb{C}_I^+
            \\=&B_I(z^{I},R_a)\cap B_I(z^{-I}, R_a)\cap\mathbb{C}_I^+
            \\=&B_I(z^{I},R_a)\cap B_I(z^{-I}, R_a^p)\cap\mathbb{C}_I^+,
        \end{split}
    \end{equation*}
    i.e. \eqref{eq-sigmapa} holds.

    (ii) If $R_a^{p,I}=R_a<R_a^p$, then by \eqref{eq-lrapk}, there is $K\in\mathcal{C}_\mathfrak{S}$ such that $R_a^{p,K}=R_a^p$. By \eqref{eq-sigmapab},
    \begin{equation}\label{eq-sspr}
        \Sigma(p,a)=\Sigma(p,R_a,(I_p,K))\cap\Sigma(p,R_a).
    \end{equation}
    According to Proposition \ref{pr-lpimw}, $I\notin\mathcal{C}_{\mathfrak{S}}^{\ker} (I_p,J)$. It is easy to check by \eqref{eq-sqrj} and \eqref{eq-sigmapr} that
    \begin{equation*}
        \Sigma(p,R_a)\cap\mathbb{C}_I=B_I(z^{I},R_a)\cap B_I(z^{-I}, R_a)=\Sigma(p,R_a,(I_p,K))\cap\mathbb{C}_I.
    \end{equation*}
    Then \eqref{eq-sigmapa} holds by \eqref{eq-sspr}.

    (iii) Otherwise, by \eqref{eq-lrapk} we have $R_a<R_a^{p,I}=R_a^p$. According to \eqref{eq-sigmapab},
    \begin{equation*}
        \Sigma(p,a)=\Sigma(p,R_a,(I_p,I))\cap\Sigma(p,R_a^{p,I}).
    \end{equation*}
    It is easy to check by
    \begin{equation*}
        \Sigma(p,R_a,(I_p,I))\cap\mathbb{C}_I=B_I(z^I,R_a)\subset B_I(z^I,R_a^{p,I})
    \end{equation*}
    and
    \begin{equation*}
        \Sigma(p,R_a^{p,I})\cap\mathbb{C}_I=B_I(z^I,R_a^{p,I})\cap B_I(z^{-I},R_a^{p,I})
    \end{equation*}
    that \eqref{eq-sigmapa} holds.
\end{proof}

\begin{prop}
    Let $U\subset\mathcal{W}_\mathfrak{S}$. If
    \begin{equation}\label{eq-clo}
        {\clo}_I(U_I)\cap\mathbb{R}
        ={\clo}_{\mathbb{R}}(U\cap\mathbb{R}),\qquad\forall\ I\in\mathcal{C}_\mathfrak{S},
    \end{equation}
    then
    \begin{equation}\label{eq-exttau}
        {\ext}_{\tau_s}(U)\cap\mathbb{C}_I
        ={\ext}_I(U_I),\qquad\forall\ I\in\mathcal{C}_\mathfrak{S},
    \end{equation}
    where ${\clo}_\mathbb{R}:={\clo}_{\tau(\mathbb{R})}$.
    \index{${\clo}_\mathbb{R}$}
\end{prop}

\begin{proof}
Fix $I\in\mathcal{C}_{\mathfrak{S}}$.

(i) Let $q\in{\ext}_{\tau_s}(U)\cap\mathbb{C}_I$. Then there is $\mathcal{O}\in\tau_s(\mathcal{W}_{\mathfrak{S}})$ with $q\in\mathcal{O}\subset{\ext}_{\tau_s}(U)$. It follows that
\begin{equation*}
    q\in\mathcal{O}_I,\qquad
    \mathcal{O}_I\in\tau(\mathbb{C}_I)
    \qquad\mbox{and}\qquad
    \mathcal{O}\subset\mathbb{C}_I\backslash U_I.
\end{equation*}
It follows that $q\in{\ext}_I(U_I)$ and
\begin{equation*}
    {\ext}_{\tau_s}(U)\cap\mathbb{C}_I
    \subset{\ext}_I(U_I).
\end{equation*}

(ii) Let $q\in{\ext}_I(U_I)$. If $q\notin\mathbb{R}$, then there is $\mathcal{U}\in\tau(\mathbb{C}_I)$ such that
\begin{equation*}
    \mathcal{U}\cap\mathbb{R}=\varnothing\qquad\mbox{and}\qquad q\in \mathcal{U}\in\tau(\mathbb{C}_I\backslash U_I).
\end{equation*}
It is easy to check that
\begin{equation*}
    \mathcal{U}\in\tau_s(\mathcal{W}_\mathfrak{S})\qquad\mbox{and}\qquad q\in\mathcal{U}\subset\mathcal{W}_\mathfrak{S}.
\end{equation*}

Otherwise, $q\in\mathbb{R}$. There is $r>0$ such that $B_I(q,r)\subset \mathbb{C}_I\backslash U_I$. According to \eqref{eq-clo}, for each $p\in (q-r,q+r)$ and $J\in\mathcal{C}_\mathfrak{S}$, there is $r_p^I>0$ such that
\begin{equation*}
    \left(p-r_p^J,p+r_p^J\right)\subset(q-r,q+r)\qquad\mbox{and}\qquad B_J(p,r_p^I)\subset \mathbb{C}_J\backslash U_J.
\end{equation*}
Then we set
\begin{equation*}
    V_{[J]}:=\bigcup_{p\in(q-r,q+r)}B_J(p,r_p^J)\in\tau(\mathbb{C}_I)
\end{equation*}
and
\begin{equation*}
    V:=\bigcup_{I\in\mathcal{C}_\mathfrak{S}}V_{[J]}
    \in\tau_s(\mathcal{W}_\mathfrak{S}).
\end{equation*}
It is easy to check that $q\in V\subset\mathcal{W}_\mathfrak{S}\backslash U$. It implies that $q\in {\ext}_{\tau_s}(U)\cap\mathbb{C}_I$,
\begin{equation*}
    {\ext}_{\tau_s}(U)\cap\mathbb{C}_I
    \supset{\ext}_I(U_I)
\end{equation*}
and \eqref{eq-exttau} follows.
\end{proof}

\begin{prop}
    Let $p\in\mathcal{W}_\mathfrak{S}\backslash\mathbb{R}$, $a=\{a_\ell\}_{\ell\in\mathbb{N}}$ with $a_\ell\in\mathfrak{S}$ and $R_a=0$. Then
    \begin{equation}\label{eq-sigmapa0}
        \Sigma(p,a)=
        \begin{cases}
            \{p\},\qquad\qquad &R_a^p\le 2 Im(p),\\
            \left\{\Psi_{I_p}^{J}(p):R_a^{p,J}>0\right\},\qquad\qquad & \mbox{otherwise}.
        \end{cases}
    \end{equation}
\end{prop}

\begin{proof}
    If $R_a^p=R_a(=0)$, then by \eqref{eq-sigmapab},
    \begin{equation*}
        \begin{cases}
            \Sigma(p,a)=B_{\mathcal{W}_\mathfrak{S}}^*(p,0)=\{p\},\qquad\qquad & p\in\mathbb{R},\\
            \Sigma(p,a)=\Sigma(p,0)=\{p\},\qquad\qquad &\mbox{otherwise},
        \end{cases}
    \end{equation*}
    and \eqref{eq-sigmapa0} holds.

    Otherwise, $R_a^p\neq R_a$. By \eqref{eq-lrapk}, there is $J\in\mathcal{C}_\mathfrak{S}$ such that $R_a^{p,J}=R_a^p$. It follows from \eqref{eq-sigmapr} and for each $K\in\mathcal{C}_\mathfrak{S}\backslash\{I_p\}$
    \begin{equation*}
        \begin{cases}
            \Psi_{I_p}^K(p)\notin \Psi_i^K\left[B^*(z_p,R_a^p)\cap B^*(\overline{z_p},R_a^p)\right],\qquad\qquad & R_a^p\le 2 Im(p),\\
            \Psi_{I_p}^K(p)\in \Psi_i^K\left[B^*(z_p,R_a)\cap B^*(\overline{z_p},R_a^p)\right],\qquad\qquad &\mbox{otherwise},
        \end{cases}
    \end{equation*}
    that
    \begin{equation}\label{eq-sigma1}
        \left\{\Psi_{I_p}^L(p)\right\}_{L\in\mathcal{C}_\mathfrak{S}}\cap\Sigma(p,R_a^p)=
        \begin{cases}
            \{p\},\qquad\qquad & R_a^p\le 2 Im(p),\\
            \left\{\Psi_{I_p}^L(p)\right\}_{L\in\mathcal{C}_\mathfrak{S}},\qquad\qquad & \mbox{otherwise}.
        \end{cases}
    \end{equation}
    It is easy to check by \eqref{eq-sqrj} that
    \begin{equation}\label{eq-sigma2}
        \begin{split}
            \Sigma(p,R_a,(I_p,J))=&\Sigma(p,0,(I_p,J))
            \\=&\left\{\Psi_{I_p}^{K}(p):R_a^{p,K}>0\right\}
            \subset \left\{\Psi_{I_p}^L(p)\right\}_{L\in\mathcal{C}_\mathfrak{S}}.
        \end{split}
    \end{equation}
    By \eqref{eq-sigmapab}, $\Sigma(p,a)=\Sigma(p,R_a,(I_p,J))\cap\Sigma(p, R_a^p)$. Then \eqref{eq-sigmapa0} holds by \eqref{eq-sigma1} and \eqref{eq-sigma2}.
\end{proof}

\begin{prop}\label{pr-ra0}
    Theorem \ref{thm-lpmw} holds when $R_a=0$ and $p\notin\mathbb{R}$.
\end{prop}

\begin{proof}
    According to Proposition \ref{pr-lipms}, the $*$-power series $P$ converges in $\{p\}$ and diverges in $\mathbb{C}_{I_p}\backslash\{p\}$. It implies that Theorem \ref{thm-lpmw} holds when $q\in\mathbb{C}_{I_p}$. Therefore we can assume that $q\notin\mathbb{C}_{I_p}$.

    (i) Suppose that $q\in\Sigma(p,a)$. By \eqref{eq-sigmapa0}, $q=\Psi_{I_p}^J(p)$, for some $J\in\mathcal{C}_\mathfrak{S}\backslash\{\pm I_p\}$ and
    \begin{equation*}
        R_a^{p,J}=R_a^p>2 Im(p) > 0=R_a.
    \end{equation*}
    By \eqref{eq-rapq} and $\frac{1}{R_a^{p,J}}<\frac{2}{R_a^p+2 Im(p)}$, there is $m\in\mathbb{N}$ such that
    \begin{equation*}
        \left|a_{\ell,\perp}^{p,q}\right|<\left(\frac{2}{R_a^p+2 Im(p)}\right)^\ell,\qquad\forall\ \ell>m.
    \end{equation*}
    According to the representation formula, see \cite{Dou2023002}*{Theorem 6.6},
    \begin{equation*}
        \begin{split}
            (q-p)^{*\ell}a_\ell=&(1,J)
            \begin{pmatrix}
                1 & I_p \\ 1 & -I_p
            \end{pmatrix}^{-1}
            \begin{pmatrix}
                (p-p)^\ell a_\ell\\
                (\overline{p}-p)^\ell a_\ell
            \end{pmatrix}
            \\=&\frac{J}{2}(I_p-J)(\overline{p}-p)^\ell a_\ell
            \\=&\frac{J}{2}(I_p-J)(2 Im(p))^\ell I_p^\ell\cdot a_\ell
            \\=&\frac{J}{2}(-J)^\ell(2 Im(p))^\ell\left[(I_p-J)a_\ell\right]
            \\=&-\frac{(-J)^{\ell+1}}{2}(2 Im(p))^\ell\left[(I_p-J)a_{\ell,\perp}^{p,q}\right].
        \end{split}
    \end{equation*}
    It follows from \eqref{eq-mul} and $(-J)^{\ell+1}\in\{\pm J\}$ that
    \begin{equation*}
        \left|(q-p)^{*\ell}a_\ell\right|
        \le \frac{\sqrt{2}}{2}(2 Im(p))^\ell \cdot2\sqrt{2}\cdot\left|a_{\ell,\perp}^{p,q}\right|
        <2\left(\frac{4 Im(p)}{R_a^p+2 Im(p)}\right)^\ell.
    \end{equation*}
    Since $\frac{4 Im(p)}{R_a^p+2 Im(p)}<1$, $P$ converges. Then Theorem \ref{thm-lpmw} holds when $q\in\Sigma(p,a)$.

    (ii) Suppose that $q\notin\Sigma(p,a)$ with $R_a^{p,I_q}>0$. Suppose that $z_p=z_q$. By \eqref{eq-sigmapa0}, $q=\Psi_{I_p}^{I_q}(p)\in\Sigma(p,a)$, a contradiction. Therefore, $z_p\neq z_q$. By \eqref{eq-rapq}, there is $m\in\mathbb{N}$ such that
    \begin{equation*}
        \left|a_{\ell,\perp}^{p,q}\right|<\left(\frac{2}{R_a^{p,I_q}}\right)^\ell,\qquad\forall\ \ell>m.
    \end{equation*}

    If $R_a^{p,-I_q}>R_a=0$, then from \eqref{eq-kerJ1J2} we deduce that
    \begin{equation*}
        \begin{split}
            \{0\}\neq\ker(I_p-I_q)=&\ker(I_p-(-I_q))
            \\=&\ker((I_p-I_q)+(I_p-(-I_q)))=\ker(2\cdot I_p)=\{0\},
        \end{split}
    \end{equation*}
    a contradiction. Therefore, $R_a^{p,-I_q}=R_a=0$. Let
    \begin{equation*}
        \lambda>\frac{2}{R_a^{p,I_q}}\frac{\left|z_q^{I_q}-z_p^{-I_q}\right|}{\left|z_q^{I_q}-z_p^{I_q}\right|}+\frac{1}{\left|z_q^{I_q}-z_p^{I_q}\right|}.
    \end{equation*}
    According to \eqref{eq-rapj}, \eqref{eq-paperp} and $R_a^{p,-I_q}=0$, there is $\{\ell_\imath\}_{\imath\in\mathbb{N}}$ such that $m<\ell_\imath<\ell_{\imath+1}$ and
    \begin{equation*}
        \left|a_{\ell_\imath,\perp}^{p,-I_q}\right|>\frac{64}{|I_p+I_q|_{\inf}^\perp}\lambda^{\ell_\imath},\qquad\forall\ \imath\in\mathbb{N}.
    \end{equation*}
    By \eqref{eq-ipiq} and \eqref{eq-cozi},
    \begin{equation*}
        \begin{split}
            \left|-I_q\frac{I_q+I_p}{2}(z_q^{I_p}-p)a_{\ell_\imath}\right|
            =&\left|-I_q\left(z_q^{I_q}-z_p^{I_q}\right)^{\ell_\imath} \frac{I_q+I_p}{2}a_{\ell_\imath}\right|
            \\=&\left|-I_q \left(z_q^{I_q}-z_p^{I_q}\right)^{\ell_\imath}\frac{I_q+I_p}{2} a_{{\ell_\imath},\perp}^{p,-I_q}\right|
            \\ \ge & \frac{|I_q|}{2}\left|z_q^{I_q}-z_p^{I_q}\right|^{\ell_\imath} \frac{|I_q+I_p|_{\inf}^{\perp}}{2}\left[\frac{64}{|I_p+I_q|_{\inf}^{\perp}} \lambda^{\ell_\imath}\right]
            \\=& 16|I_q|\left|z_q^{I_q}-z_p^{I_q}\right|^{\ell_\imath}\lambda^{\ell_\imath}
            \\\ge & 16|I_q|,
        \end{split}
    \end{equation*}
    and
    \begin{equation*}
        \begin{split}
            \left|-I_q\frac{I_q-I_p}{2}\left(z_q^{-I_p}-p\right)^{\ell_\imath} a_{\ell_\imath}\right|
            = & \left|-I_q\left(z_q^{I_q}-z_p^{-I_q}\right)^{\ell_\imath} \frac{I_q-I_p}{2} a_{\ell_\imath}\right|
            \\ = & \left|-I_q\left(z_q^{I_q}-z_p^{-I_q}\right)^{\ell_\imath} \frac{I_q-I_p}{2} a_{{\ell_\imath},\perp}^{p,I_q}\right|
            \\ \le & 4|I_q|\left|z_q^{I_q}-z_p^{-I_q}\right|^{\ell_\imath}\left|\frac{I_q-I_p}{2}\right|\left(\frac{2}{R_a^{p,I_q}}\right)^{\ell_\imath}
            \\  \le & 8 |I_q|\left(\frac{2}{R_a^{p,I_q}}\frac{\left|z_q^{I_q}-z_p^{-I_q}\right|}{\left|z_q^{I_q}-z_p^{I_q}\right|}\right)^{\ell_\imath}\left|z_q^{I_q}-z_p^{I_q}\right|^{\ell_\imath}
            \\\le & 8|I_q|\left|z_q^{I_q}-z_p^{I_q}\right|^{\ell_\imath}\lambda^{\ell_\imath}
            \\\le & \frac{1}{2}\left|-I_q\frac{I_q+I_p}{2}(z_q^{I_p}-p)^{\ell_\imath}a_{\ell_\imath}\right|.
        \end{split}
    \end{equation*}

    According to representation formula, see \cite{Dou2023002}*{Theorem 6.6},
    \begin{equation*}
        \begin{split}
            \left|(q-p)^{*\ell_\imath}a_{\ell_\imath}\right|
            =&\left|(1,I_q)
            \begin{pmatrix}
                1 & I_p \\ 1 & -I_p
            \end{pmatrix}^{-1}
            \begin{pmatrix}
                (z_q^{I_p}-p)^{\ell_\imath} a_{\ell_\imath}\\
                (z_q^{-I_p}-p)^{\ell_\imath} a_{\ell_\imath}
            \end{pmatrix}\right|
            \\ = &\left|-I_q\frac{I_q+I_p}{2}\left(z_q^{I_p}-p\right)^{\ell_\imath} a_{\ell_\imath}-I_q\frac{I_q-I_p}{2}\left(z_q^{-I_p}-p\right)^{\ell_\imath} a_{\ell_\imath}\right|
            \\ = &\left|-I_q\frac{I_q+I_p}{2}\left(z_q^{I_p}-p\right)^{\ell_\imath} a_{\ell_\imath}\right|
            -\left|I_q\frac{I_q-I_p}{2}\left(z_q^{-I_p}-p\right)^{\ell_\imath} a_{\ell_\imath}\right|
            \\ \ge & \frac{1}{2} \left|-I_q\frac{I_q+I_p}{2}\left(z_q^{I_p}-p\right)^{\ell_\imath} a_{\ell_\imath}\right|
            \\ \ge & 16|I_q|.
        \end{split}
    \end{equation*}
    It implies that $P$ diverges. Therefore Theorem \ref{thm-lpmw} holds when $q\notin\Sigma(p,a)$ with $R_a^{p,I_q}>0$.

    (iii) Suppose that $q\notin\Sigma(p,a)$ with $R_a^{p,I_q}=0$. By \eqref{eq-rapq},
    \begin{equation*}
        R_{a_{\perp}^{p,q}}=R_a^{p,q}=R_a^{p,I_q}=0.
    \end{equation*}
    By $|\overline{z_q}-z_p|>|z_q-z_p|$, there is $m\in\mathbb{N}$ with
    \begin{equation*}
        \left(\frac{|z_q-z_p|}{|\overline{z_q}-z_p|}\right)^\ell
        <\frac{|I_q-I_p|_{\inf}^{\perp}}{32|I_q+I_p|},\qquad\qquad\forall\ \ell>m.
    \end{equation*}
    Then
    \begin{equation*}
        \begin{split}
            \left|-\frac{I_q+I_p}{2}I_p\left(z_q^{I_p}-p\right)^{\ell} a_{\ell,\perp}^{p,q}\right|
            < & 4\cdot\frac{|I_q+I_p|}{2}\cdot 1 \cdot |z_q-z_p|^{\ell}\cdot \left|a_{\ell,\perp}^{p,q}\right|
            \\< & 2\cdot|I_q+I_p|\cdot \frac{|I_q-I_p|_{\inf}^\perp}{16|I_q+I_p|}\cdot |\overline{z_q}-z_p|^\ell\cdot\left|a_{\ell,\perp}^{p,q}\right|
            \\=& \frac{1}{8}\frac{|I_q-I_p|_{\inf}^\perp}{2}\cdot |\overline{z_q}-z_p|^\ell\cdot\left|a_{\ell,\perp}^{p,q}\right|
            \\< & \frac{1}{2}\left|\frac{I_q-I_p}{2}I_p\left(z_q^{-I_p}-p\right)^{\ell}a_{\ell,\perp}^{p,q}\right|.
        \end{split}
    \end{equation*}
    Since $R_{a_\perp^{p,q}}=0$, there is $\{\jmath_\imath\}_{\imath\in\mathbb{N}}$ such that for each $\imath\in\mathbb{N}$,
    \begin{equation*}
        \jmath_\imath\in\mathbb{N},\qquad
        \jmath_{\imath+1}>\jmath_\imath>\ell,\qquad\mbox{and}\qquad
        \left|a_{\jmath_\imath,\perp}^{p,q}\right|>\frac{1}{|\overline{z_q}-z_p|^{\jmath_\imath}}.
    \end{equation*}
    According to \eqref{eq-adppq},
    \begin{equation*}
        \begin{split}
            &\left|(q-p)^{*\jmath_\imath}a_{\jmath_\imath}\right|
            \ge \left|\left[(q-p)^{*\jmath_\imath}a_{\jmath_\imath}\right]_{\perp}^{p,q}\right|
            \\=&\left|\left[-\frac{I_q+I_p}{2}I_p\left(z_q^{I_p}-p\right)^{\jmath_\imath} a_{\jmath_\imath}+\frac{I_q-I_p}{2}I_p\left(z_q^{-I_p}-p\right)^{\jmath_\imath}a_{\jmath_\imath}\right]_{\perp}^{p,-q}\right|
            \\=&\left|-\frac{I_q+I_p}{2}I_p\left(z_q^{I_p}-p\right)^{\jmath_\imath} a_{\jmath_\imath,\perp}^{p,-q}+\frac{I_q-I_p}{2}I_p\left(z_q^{-I_p}-p\right)^{\jmath_\imath}a_{\jmath_\imath,\perp}^{p,-q}\right|
            \\\ge&\left|\frac{I_q-I_p}{2}I_p\left(z_q^{-I_p}-p\right)^{\jmath_\imath}a_{\jmath_\imath,\perp}^{p,-q}\right|-\left|-\frac{I_q+I_p}{2}I_p\left(z_q^{I_p}-p\right)^{\jmath_\imath} a_{\jmath_\imath,\perp}^{p,-q}\right|
            \\\ge&\frac{1}{2}\left|\frac{I_q-I_p}{2}I_p\left(z_q^{-I_p}-p\right)^{\jmath_\imath}a_{\jmath_\imath,\perp}^{p,-q}\right|
            \\ \ge &\frac{1}{8}\frac{\left|I_q-I_p\right|_{\inf}^\perp}{2}\cdot1\cdot|\overline{z_q}-z_p|^{\jmath_\imath}\cdot\frac{1}{|\overline{z_q}-z_p|^{\jmath_\imath}}
            \\=&\frac{\left|I_q-I_p\right|_{\inf}^\perp}{16},
        \end{split}
    \end{equation*}
    which implies that $P$ diverges and Theorem \ref{thm-lpmw} holds when $q\notin\Sigma(p,a)$ with $R_a^{p,I_q}=0$.
\end{proof}

\begin{prop}\label{pr-0inf}
    Theorem \ref{thm-lpmw} holds when $R_a=0$.
\end{prop}

\begin{proof}
    Assume that $p\in\mathbb{R}$. According to Proposition \ref{pr-lipms}, for each $I\in\mathcal{C}_\mathfrak{S}$ the power series $P$ converges in
    \begin{equation*}
        \{p\}=B_{\mathcal{W}_\mathfrak{S}}^*(p,R_a)=\Sigma(p,a)
    \end{equation*}
    and diverges in $\mathbb{C}_I\backslash\{p\}={\ext}_I\big(B_I^*(p,R_a)\big)$. Therefore, $P$ diverges in
    \begin{equation*}
        \bigcup_{I\in\mathcal{C}_\mathfrak{S}}\left(\mathbb{C}_I\backslash\{p\}\right)
        =\mathcal{W}_\mathfrak{S}\backslash\{p\}
        ={\ext}_{\tau_s}\{p\}
        ={\ext}_{\tau_s}[\Sigma(p,a)],
    \end{equation*}
    and Theorem \ref{thm-lpmw} holds when $R_a=0$ and $p\in\mathbb{R}$. It follows from Proposition \ref{pr-ra0} that Proposition \ref{pr-0inf} holds.
\end{proof}

\begin{prop}\label{pr-0inf1}
    Theorem \ref{thm-lpmw} holds when $R_a=+\infty$.
\end{prop}

\begin{proof}
     {Let $\varepsilon\in(0,+\infty)$ be arbitrary and since}
    \begin{equation*}
        \mathcal{W}_\mathfrak{S}=\bigcup_{r\in(0,+\infty)}\Sigma(p,r),
    \end{equation*}
    we only need to prove that $P$ converges uniformly in $\clo_{\tau_s}\left[\Sigma(p,\varepsilon)\right]$.

    Let $\delta>0$. By Proposition \ref{pr-lipms}, $P$ converges uniformly in $B_{I_p}(p,\varepsilon+1)$. Then there is $m\in\mathbb{N}$ such that
    \begin{equation*}
        \left|\sum_{\imath=\ell}^{+\infty}(w-p)^\imath a_\imath\right|<\frac{\delta}{8},
        \qquad\qquad\forall\ \ell>m\quad\mbox{and}\quad w\in B_{I_p}(p,\varepsilon+1).
    \end{equation*}
    Let $q\in\clo_{\tau_s}(\Sigma(p,\varepsilon))\subset\Sigma(p,\varepsilon+1)$. Then
    \begin{equation*}
        z_p^{I_p},z_p^{-I_p}\in B_{I_p}(p,\varepsilon+1).
    \end{equation*}
    By the representation formula, for each $\ell>m$, we have
    \begin{equation*}
        \begin{split}
            &\left|\sum_{\imath=\ell}^{+\infty}(q-p)^{*\imath} a_\imath\right|
            \\=&\left|-\frac{I_q+I_p}{2}I_p\sum_{\imath=\ell}^{+\infty}\left(z_q^{I_p}-p\right)^\imath a_\imath+\frac{I_q-I_p}{2}I_p\sum_{\imath=\ell}^{+\infty}\left(z_q^{-I_p}-p\right)^\imath a_\imath\right|
            \\\le & 4\left|\frac{I_q+I_p}{2}\right||I_p|\left|\sum_{\imath=\ell}^{+\infty}\left(z_q^{I_p}-p\right)^\imath a_\imath\right|+4\left|\frac{I_q-I_p}{2}\right||I_p|\left|\sum_{\imath=\ell}^{+\infty}\left(z_q^{-I_p}-p\right)^\imath a_\imath\right|
            \\\le & 4\cdot 1\cdot 1\cdot\frac{\delta}{8}+4\cdot 1\cdot 1\cdot\frac{\delta}{8}=\delta.
        \end{split}
    \end{equation*}
    It implies that $P$ converges uniformly in $\clo_{\tau_s}\left[\Sigma(p,\varepsilon)\right]$.
\end{proof}

Let $(X,\tau)$ be a topological space, $V\subset X$ and $U\in\tau(X)$. It is easy to check that
\begin{equation}\label{eq-taux}
    {\ext}_{\tau(X)}(V)\cap U={\ext}_{\tau(X)}(V\cap U)\cap U.
\end{equation}

We are now in position to prove our main result that we repeat for the reader's convenience.\\
{\bf Theorem \ref{thm-lpmw}}
{\em
	Let $p\in\mathcal{W}_{\mathfrak{S}}$ and $a=\{a_\ell\}_{\ell\in\mathbb{N}}$ with $a_\ell\in\mathfrak{S}$. Then the power series
	\begin{equation*}
		\begin{split}
			P:\ \mathcal{W}_\mathfrak{S}\ &\xlongrightarrow[\hskip1cm]{}\ \mathfrak{S},
			\\ q\ \ &\shortmid\!\xlongrightarrow[\hskip1cm]{}\ P(q)=\sum_{\ell\in\mathbb{N}} (q-p)^{*\ell} a_\ell,
		\end{split}
	\end{equation*}
	converges absolutely and  {it is} slice regular in $\Sigma(p,a)$, moreover it diverges in
    ${\ext}_{\tau_s}\left[\Sigma(p,a)\right]$.
	
	Moreover, if $R_a\in(0,+\infty]$ and $\varepsilon\in\left(0,R_a\right)$, then the power series $P$ converges uniformly in
    \begin{equation*}
        \begin{cases}
            \clo_{\tau_s}\left[\Sigma(p,\varepsilon)\right],\qquad\qquad&R_a=+\infty,
            \\\clo_{\tau_s}\left[\Sigma(p,a,\varepsilon)\right],\qquad\qquad&\mbox{otherwise.}
        \end{cases}
    \end{equation*}
}

\begin{proof}[Proof of Theorem \ref{thm-lpmw}]
	 Theorem \ref{thm-lpmw} holds when $R_a=0,+\infty$ by Proposition \ref{pr-0inf} and Proposition \ref{pr-0inf1}. Assume that $R_a\in (0,+\infty)$. Let
    \begin{equation*}
        \delta>0,\qquad r_1:=R_a-\frac{\varepsilon}{2},\qquad r_1^p:=R_a^p-\frac{\varepsilon}{2},
    \end{equation*}
    \begin{equation*}
        r_2\in\left(r_1,R_a\right),\qquad\mbox{and}\qquad r_2^p\in\left(r_1^p,R_a^p\right).
    \end{equation*}
    By Corollary \ref{cor-pws}, there is $m\in\mathbb{N}$ such that for each $\ell>m$ we have
    \begin{equation*}
        |a_\ell|<\frac{1}{r_2^\ell},
    \end{equation*}
    \begin{equation}\label{eq-2ai}
        \left|a_{\ell,\perp}^{p,J}\right|<\frac{1}{(r_2^p)^\ell},\qquad\forall\ J\in\mathcal{C}_\mathfrak{S}\quad\mbox{and}\quad R_a^{p,J}=R_a^p,
    \end{equation}
    and
    \begin{equation*}
        \max\left\{\frac{\left(\frac{r_1}{r_2}\right)^m}{1-\frac{r_1}{r_2}},\frac{\left(\frac{r_1^p}{r_2^p}\right)^m}{1-\frac{r_1^p}{r_2^p}}\right\}<\frac{\delta}{12}.
    \end{equation*}
	
	(i) Let $I\in\mathcal{C}_{\mathfrak{S}}$ with $p\in\mathbb{C}_I$, then by Proposition \ref{pr-lipms} and \eqref{eq-exttau}, the series $P$ converges absolutely in $B_I(p,R_a)=\left[\Sigma(p,a)\right]_I$, diverges in \begin{equation*}
	    {\ext}_{I} \left[B_I(p,R_a)\right]
        ={\ext}_{I} \left[\Sigma(p,a)\cap\mathbb{C}_I\right]
        ={\ext}_{\tau_s}\left[\Sigma(p,a)\right]\cap\mathbb{C}_I.
	\end{equation*}
    For each
	\begin{equation*}
		q\in{\clo}_{I}(B_I(p,R_a-\varepsilon))=\left({\clo}_{\tau_s}\left[\Sigma(p,a,\varepsilon)\right]\right)_I,
	\end{equation*}
	it follows from \eqref{eq-mul} that
	\begin{equation}\label{eq-bsls}
		\begin{split}
			\left|\sum_{\ell=m}^{+\infty} (q-p)^{*\ell}a_\ell\right|
			\le &\sqrt{2}\cdot \sum_{\ell=m}^{+\infty} |q-p|^{\ell}\left|a_\ell\right|
			\\< &\sqrt{2}\cdot\sum_{\ell=m}^{+\infty}(r_1)^\ell\frac{1}{(r_2)^\ell}=\sqrt{2}\cdot \frac{\left(\frac{r_1}{r_2}\right)^m}{1-\frac{r_1}{r_2}}<\frac{\delta}{6}.
		\end{split}
	\end{equation}
	
	(ii) Suppose that there is $J\in\mathcal{C}_{\mathfrak{S}}$ such that $p\notin\mathbb{C}_J$. By \eqref{eq-taux} and $\mathbb{C}_J^+\in\tau(\mathbb{C}_J)$,
    \begin{equation*}
        {\ext}_J\left(\Sigma(p,a)\cap\mathbb{C}_J^+\right)={\ext}_J(\Sigma(p,a)\cap\mathbb{C}_J).
    \end{equation*}
    According to Proposition \ref{pr-ljim}, \eqref{eq-sigmapa} and \eqref{eq-exttau}, $P$ converges in
    \begin{equation*}
        \mathbb{B}:=B_J\left(z^{J},R_a\right)\cap B_J\left(z^{-J}, R_a^{p,J}\right)\cap \mathbb{C}_{J}^+
        =
        \Sigma(p,a)\cap \mathbb{C}_{J}^+,
    \end{equation*}
    and diverges in
    \begin{equation*}
        \begin{split}
            {\ext}_{\tau\left(\mathbb{C}_J^+\right)}(\mathbb{B})
            =&{\ext}_J(\mathbb{B})\cap\mathbb{C}_J^+
            ={\ext}_J\left[\Sigma(p,a)\cap\mathbb{C}_J^+\right]\cap\mathbb{C}_J^+
            \\=&{\ext}_J\left[\Sigma(p,a)\cap\mathbb{C}_J\right]\cap\mathbb{C}_J^+
            ={\ext}_{\tau_s}\left[\Sigma(p,a)\right]\cap\mathbb{C}_J\cap\mathbb{C}_J^+
            \\=&{\ext}_{\tau_s}\left[\Sigma(p,a)\right]\cap\mathbb{C}_J^+.
        \end{split}
    \end{equation*}
    By (i), $P$ converges in $\Sigma(p,a)$ and converges in $\ext_{\tau_s}\left[\Sigma(p,a)\right]$.
	
	If
    \begin{equation*}
        \begin{split}
            q \in & {\clo}_{\tau_s}\left[\Sigma\left(p,a,\varepsilon\right)\right]
            \subset \Sigma\left(p,a,\frac{\varepsilon}{2}\right)
            \\=&\Sigma\left(p,R_a-\frac{\varepsilon}{2},\left(I_p,I_q\right)\right)\cap\Sigma\left(p, R_a^p-\frac{\varepsilon}{2}\right),
        \end{split}
    \end{equation*}
    then by \eqref{eq-oaeq} and \eqref{eq-rdra},
    \begin{equation*}
        R_o=R_a\le R_a^p=R_d\le R_c,
    \end{equation*}
    where $o$, $c$ and $d$ are defined in \eqref{eq-cldl} and \eqref{eq-oell}.
    By \eqref{eq-2ai}, for each $\ell>m$, we have
    \begin{equation*}
        |d_\ell|=\left|\frac{1+I_qI_p}{2} a_{\ell,\perp}^{p,q}\right|
        =\left|I_q\frac{-I_q+I_p}{2} a_{\ell,\perp}^{p,q}\right|\le (\sqrt 2 )^2\left|a_{\ell,\perp}^{p,q}\right|<\frac{2}{(r_2^p)^\ell}.
    \end{equation*}
    Similarly,
    \begin{equation*}
        |o_\ell|,|c_\ell|\le2|a_\ell|<\frac{2}{(r_2)^\ell}.
    \end{equation*}

    Therefore,
	\begin{equation*}
		\begin{split}
			&\left|\sum_{\ell=m}^{+\infty} (q-p)^{*\ell} a_\ell\right|
			\\\le&\sqrt{2}\left( \sum_{\ell=m}^{+\infty} \left|q-z^{I_q}\right|^{\ell} |o_\ell|+\sum_{\ell=m}^{+\infty} \left|q-z^{I_q}\right|^{\ell} |c_\ell|+\sum_{\ell=m}^{+\infty} \left|q-z^{-I_q}\right|^{\ell} |d_\ell|\right)
			\\<&\sqrt{2}\left( \sum_{\ell=m}^{+\infty}(r_1)^\ell\frac{2}{(r_2)^\ell}+\sum_{\ell=m}^{+\infty}(r_1)^\ell\frac{2}{(r_2)^\ell}+\sum_{\ell=m}^{+\infty}(r_1^p)^\ell\frac{2}{(r_2^p)^\ell}\right)
			\\=&\sqrt{2}\left(  4\cdot\frac{\left(\frac{r_1}{r_2}\right)^m}{1-\frac{r_1}{r_2}}+2\cdot\frac{\left(\frac{r_1^p}{r_2^p}\right)^m}{1-\frac{r_1^p}{r_2^p}}\right)<\sqrt{2}\left(\frac{4\delta}{12}+\frac{2\delta}{12}\right)<\delta.
		\end{split}
	\end{equation*}
	By \eqref{eq-bsls}, $P$ converges uniformly in  $\clo_{\tau_s}\left[\Sigma(p,a,\varepsilon)\right]$. It implies that $P$ is a slice regular function in $\Sigma(p,a)$.
\end{proof}

\newpage



\printindex

\newpage

\section*{Appendix: Sedenion multiplication table}

\begin{table}[htbp]
	\centering\small
	\resizebox{\textwidth}{3.5cm}{
		\begin{tabular}{|>{\columncolor[gray]{.9}}c|c c c c|c c c c|c c c c|c c c c|}
			\hline\rowcolor[gray]{.9}$\cdot$&$1$&$e_1$&$e_2$&$e_3$&$e_4$&$e_5$&$e_6$&$e_7$&$e_8$&$e_9$&$e_{10}$&$e_{11}$&$e_{12}$&$e_{13}$&$e_{14}$&$e_{15}$\\
			\hline$1$&$1$&$e_1$&$e_2$&$e_3$&$e_4$&$e_5$&$e_6$&$e_7$&$e_8$&$e_9$&$e_{10}$&$e_{11}$&$e_{12}$&$e_{13}$&$e_{14}$&$e_{15}$\\
			$e_1$&$e_1$&$-1$&$e_3$&$-e_2$&$e_5$&$-e_4$&$-e_7$&$e_6$&$e_9$&$-e_8$&$-e_{11}$&$e_{10}$&$-e_{13}$&$e_{12}$&$e_{15}$&$-e_{14}$\\
			$e_2$&$e_2$&$-e_3$&$-1$&$e_1$&$e_6$&$e_7$&$-e_4$&$-e_5$&$e_{10}$&$e_{11}$&$-e_8$&$-e_9$&$-e_{14}$&$-e_{15}$&$e_{12}$&$e_{13}$\\
			$e_3$&$e_3$&$e_2$&$-e_1$&$-1$&$e_7$&$-e_6$&$e_5$&$-e_4$&$e_{11}$&$-e_{10}$&$e_9$&$-e_8$&$-e_{15}$&$e_{14}$&$-e_{13}$&$e_{12}$\\
			\hline$e_4$&$e_4$&$-e_5$&$-e_6$&$-e_7$&$-1$&$e_1$&$e_2$&$e_3$&$e_{12}$&$e_{13}$&$e_{14}$&$e_{15}$&$-e_8$&$-e_9$&$-e_{10}$&$-e_{11}$\\
			$e_5$&$e_5$&$e_4$&$-e_7$&$e_6$&$-e_1$&$-1$&$-e_3$&$e_2$&$e_{13}$&$-e_{12}$&$e_{15}$&$-e_{14}$&$e_9$&$-e_8$&$e_{11}$&$-e_{10}$\\
			$e_6$&$e_6$&$e_7$&$e_4$&$-e_5$&$-e_2$&$e_3$&$-1$&$-e_1$&$e_{14}$&$-e_{15}$&$-e_{12}$&$e_{13}$&$e_{10}$&$-e_{11}$&$-e_8$&$e_9$\\
			$e_7$&$e_7$&$-e_6$&$e_5$&$e_4$&$-e_3$&$-e_2$&$e_1$&$-1$&$e_{15}$&$e_{14}$&$-e_{13}$&$-e_{12}$&$e_{11}$&$e_{10}$&$-e_9$&$-e_8$\\
			\hline$e_8$&$e_8$&$-e_9$&$-e_{10}$&$-e_{11}$&$-e_{12}$&$-e_{13}$&$-e_{14}$&$-e_{15}$&$-1$&$e_1$&$e_2$&$e_3$&$e_4$&$e_5$&$e_6$&$e_7$\\
			$e_9$&$e_9$&$e_8$&$-e_{11}$&$e_{10}$&$-e_{13}$&$e_{12}$&$e_{15}$&$-e_{14}$&$-e_1$&$-1$&$-e_3$&$e_2$&$-e_5$&$e_4$&$e_7$&$-e_6$\\
			$e_{10}$&$e_{10}$&$e_{11}$&$e_8$&$-e_9$&$-e_{14}$&$-e_{15}$&$e_{12}$&$e_{13}$&$-e_2$&$e_3$&$-1$&$-e_1$&$-e_6$&$-e_7$&$e_4$&$e_5$\\
			$e_{11}$&$e_{11}$&$-e_{10}$&$e_9$&$e_8$&$-e_{15}$&$e_{14}$&$-e_{13}$&$e_{12}$&$-e_3$&$-e_2$&$e_1$&$-1$&$-e_7$&$e_6$&$-e_5$&$e_4$\\
			\hline$e_{12}$&$e_{12}$&$e_{13}$&$e_{14}$&$e_{15}$&$e_8$&$-e_9$&$-e_{10}$&$-e_{11}$&$-e_4$&$e_5$&$e_6$&$e_7$&$-1$&$-e_1$&$-e_2$&$-e_3$\\
			$e_{13}$&$e_{13}$&$-e_{12}$&$e_{15}$&$-e_{14}$&$e_9$&$e_8$&$e_{11}$&$-e_{10}$&$-e_5$&$-e_4$&$e_7$&$-e_6$&$e_1$&$-1$&$e_3$&$-e_2$\\
			$e_{14}$&$e_{14}$&$-e_{15}$&$-e_{12}$&$e_{13}$&$e_{10}$&$-e_{11}$&$e_8$&$e_9$&$-e_6$&$-e_7$&$-e_4$&$e_5$&$e_2$&$-e_3$&$-1$&$e_1$\\
			$e_{15}$&$e_{15}$&$e_{14}$&$-e_{13}$&$-e_{12}$&$e_{11}$&$e_{10}$&$-e_9$&$e_8$&$-e_7$&$e_6$&$-e_5$&$-e_4$&$e_3$&$e_2$&$-e_1$&$-1$\\\hline
	\end{tabular}}
\end{table}

\end{document}